\documentclass[11pt,reqno]{amsart}
\usepackage{amsthm,amssymb,amsmath}
\usepackage{graphicx}

\newtheorem{theorem}{Theorem}
\newtheorem*{theorem*}{Theorem}
\newtheorem{lemma}{Lemma}
\newtheorem{corollary}{Corollary}
\newtheorem*{corollary*}{Corollary}
\newtheorem{proposition}{Proposition}
\newtheorem{claim}{\it Claim}
\newtheorem*{claim*}{\it Claim}

\newtheorem*{theoremA}{Theorem A}

\theoremstyle{definition}

\newtheorem*{definition*}{\sc Definition}

\newtheorem{remark}{\bf Remark}
\newtheorem*{remark*}{\bf Remark}
\newtheorem*{remarks}{\bf Remarks}

\newtheorem*{example*}{\bf Example}

\newcommand{\loc}{{\rm loc}}

\newcommand{\Real}{{\rm Re}}

\newcommand{\sprt}{{\rm sprt\,}}

\def\dotminus{\mathbin{\ooalign{\hss\raise1ex\hbox{.}\hss\cr
  \mathsurround=0pt$-$}}}

\begin{document}

\title{Heat kernel of fractional Laplacian with Hardy drift via desingularizing weights}


\author{D.\,Kinzebulatov}

\address{Universit\'{e} Laval, D\'{e}partement de math\'{e}matiques et de statistique, 1045 av.\,de la M\'{e}decine, Qu\'{e}bec, QC, G1V 0A6, Canada}

\email{damir.kinzebulatov@mat.ulaval.ca}

\author{Yu.\,A.\,Sem\"{e}nov}

\address{University of Toronto, Department of Mathematics, 40 St.\,George Str, Toronto, ON, M5S 2E4, Canada}

\email{semenov.yu.a@gmail.com}

\author{K.\,Szczypkowski}

\address{Politechnika Wroc\l awska, Wydzia\l\;Matematyki, Wyb. Wyspia\`{n}skiego
27, 50-370 Wroc\l aw, Poland}

\email{karol.szczypkowski@pwr.edu.pl}

\begin{abstract}
We establish sharp two-sided bounds on the heat kernel of the fractional Laplacian, perturbed by a drift having critical-order singularity, by transferring it to appropriate weighted space with singular weight.
\end{abstract}


\thanks{The research of D.K. is supported by grants from NSERC and FRQNT}

\maketitle

\section{Introduction}

The subject of this paper are sharp two-sided weighted bounds on the heat kernel $e^{-t\Lambda}(x,y)$ of the fractional Kolmogorov operator ($1<\alpha<2$)
\begin{equation*}
\Lambda \equiv (-\Delta)^{\frac{\alpha}{2}} + b \cdot \nabla, \quad b(x)=\kappa|x|^{-\alpha}x,\;\;\kappa>0,
\end{equation*}
on $\mathbb R^d$, $d \geq 3$, where
\begin{equation}
\label{kappa_}
\kappa:=\sqrt{\delta} \frac{2^{\alpha+1}}{d-\alpha}\frac{\Gamma^2(\frac{d}{4}+\frac{\alpha}{4})}{\Gamma^2(\frac{d}{4}-\frac{\alpha}{4})}, \quad 0<\delta<4.
\end{equation}
The model vector field $b$ exhibits critical behaviour: the standard upper bound on $e^{-t\Lambda}(x,y)$ in terms of the heat kernel of $(-\Delta)^{\frac{\alpha}{2}}$ does not hold. Instead, both upper and lower bounds depend explicitly on the value of relative bound $\delta$ via the presence of a singular weight:
\begin{equation}
\label{lub}
e^{-t\Lambda}(x,y) \approx e^{-t(-\Delta)^\frac{\alpha}{2}}(x,y)(t^{-\frac{1}{\alpha}}|y| \wedge 1)^{-d+\beta}, 
\end{equation}
for all $t>0$, $x$, $y \in \mathbb R^d$, $y \neq 0$,
where $\beta \in ]\alpha,d[$ defined as the unique solution to the equation
\begin{equation}
\label{beta_eq}
\frac{2^\alpha}{\beta-\alpha}\frac{\Gamma(\frac{\beta}{2})\Gamma(\frac{d}{2}-\frac{\beta-\alpha}{2})}{\Gamma(\frac{d}{2}-\frac{\beta}{2})\Gamma(\frac{\beta-\alpha}{2})}=\kappa;
\end{equation}
\eqref{beta_eq} entails that $|x|^{-d+\beta}$ is a Lyapunov function to the formal operator $\Lambda^*=(-\Delta)^{\frac{\alpha}{2}}-\nabla \cdot b$, i.e.\,$\Lambda^*|x|^{-d+\beta}=0$. 
Here, given $a(z)$, $b(z) \geq 0$, we write $a(z) \approx b(z)$ if $c^{-1}b(z)\leq a(z) \leq cb(z)$ for some constant $c>1$ and all admissible $z$. Recall that $e^{-t(-\Delta)^\frac{\alpha}{2}}(x,y) \approx t^{-\frac{d}{\alpha}}\wedge\frac{t}{|x-y|^{d+\alpha}}$.

The operator $(-\Delta)^{\frac{\alpha}{2}} + \mathsf{f} \cdot \nabla$, $\mathsf{f}:\mathbb R^d \rightarrow \mathbb R^d$ has been the subject of intensive study over the past few decades motivated, in particular, by applications in probability theory. Indeed, $(-\Delta)^{\frac{\alpha}{2}} + \mathsf{f} \cdot \nabla$ arises as the generator of $\alpha$-stable process with a drift which, in contrast to diffusion processes, has long range interactions.
The search for the largest class of admissible vector fields $\mathsf{f}$ led to the Kato class corresponding to $(-\Delta)^{\frac{\alpha}{2}}$, which is recognized as the class responsible for existence of two-sided 
estimates on the heat kernel of $(-\Delta)^{\frac{\alpha}{2}} + \mathsf{f} \cdot \nabla$ in terms  of 
the heat kernel of $(-\Delta)^{\frac{\alpha}{2}}$ \cite{BJ}.

The vector field $b(x)=\kappa|x|^{-\alpha}x$ has a stronger singularity than the ones covered by the Kato class, and so the corresponding semigroup $e^{-t\Lambda}$ is not $L^1 \rightarrow L^\infty$ ultracontractive. It turns out that $e^{-t\Lambda}$ is ultracontractive as a mapping $L^1_{\scriptscriptstyle \sqrt{\varphi}} \rightarrow L^\infty$, where $L^1_{\scriptscriptstyle \sqrt{\varphi}} \equiv L^1(\mathbb R^d,\varphi dx)$, for appropriate singular weight $\varphi(y) \equiv \varphi_s(y) \approx (s^{-\frac{1}{\alpha}}|y| \wedge 1)^{-d+\beta}$, $0<t \leq s$. This observation is the key to the proof of the following crucial bound on the heat kernel  $e^{-t\Lambda}(x,y)$ (we discuss its construction  below):
\begin{equation}
\label{nie_intro}
e^{-t\Lambda}(x,y) \leq Ct^{-\frac{d}{\alpha}}\varphi_t(y), 
\end{equation}
(the weighted Nash initial estimate).
In turn, the proof of the $L^1_{\scriptscriptstyle \sqrt{\varphi}} \rightarrow L^\infty$ ultracontractivity of $e^{-t\Lambda}$ depends on the ``desingularizing'' $(L^1,L^1)$ bound on the weighted semigroup $$\|\varphi_s e^{-t\Lambda}\varphi_s^{-1}f\|_{1} \leq e^{c\frac{t}{s}}\|f\|_{1}, \quad 0<t \leq s, \quad c>0,$$  
and a variant of the Sobolev embedding property  of $\Lambda$. See  Theorem \ref{thm2}.

Applying carefully bound \eqref{nie_intro}, the Duhamel formula, and taking into account the properties of $e^{-t(-\Delta)^{\frac{\alpha}{2}}}(x,y)$ and the explicit form of $b$, we prove the upper bound in \eqref{lub} in Theorem \ref{thm2_}. 

To establish the lower bound in \eqref{lub},
we first prove an ``integral lower bound'' $\langle \varphi_s e^{-t\Lambda} \varphi_s^{-1}g \rangle \geq e^{-\mu\frac{t}{s}}\|g\|_1$, $g \geq 0$, $0<t \leq s$, $\mu>0$. Here and below,
$$
\langle u \rangle :=\int_{\mathbb R^d}u dx, \qquad \langle u,v\rangle :=\langle u\bar{v}\rangle.
$$
Next, we prove the ``standard'' lower bound $e^{-t\Lambda}(x,y) \geq C_1 e^{-t(-\Delta)^\frac{\alpha}{2}}(x,y)$ away from the singularity at the origin, and then combine these two bounds using the $3q$ argument similar to the one in \cite[sect.\,4.2]{BGJP} to obtain the lower bound. See Theorem \ref{thm2_}.

In the proof of Theorem \ref{thm2_} we avoid the use of scaling, having in mind possible extensions of the method to more general coefficients.

In contrast to the existing literature (see overview below),
the present paper deals with purely non-local and non-symmetric situation. This requires new techniques. The latter becomes apparent already when one faces the problem of constructing the heat kernel of $\Lambda$ ($\equiv$ integral kernel of the corresponding $C_0$ semigroup $e^{-t\Lambda}$). Namely, let us note first that in the special case $\delta<1$ one can construct $\Lambda$ in $L^2$ as the algebraic sum $(-\Delta)^{\frac{\alpha}{2}} + b \cdot \nabla$, $D(\Lambda)=D((-\Delta)^{\frac{\alpha}{2}})$, in which case $\Lambda$ becomes the generator of a $C_0$ semigroup in $L^2$. Indeed, $b \cdot \nabla$ is Rellich's perturbation of $(-\Delta)^{\frac{\alpha}{2}}$, as follows from the fractional Hardy-Rellich inequality
$$\| (-\Delta)^{\frac{\alpha}{4}}f\|_2^2 \geq c_{d,\alpha}^2 \||x|^{-\frac{\alpha}{2}}f\|_2^2, \quad c_{d,\alpha}:=2^{\frac{\alpha}{2}}\frac{\Gamma(\frac{d}{4}+\frac{\alpha}{4})}{\Gamma(\frac{d}{4}-\frac{\alpha}{4})}$$
(see \cite[Lemma 2.7]{KPS}), and so $e^{-t\Lambda}$ is a holomorphic (contraction) semigroup in $L^2$ (see details in Section \ref{rellich_app}). Moreover, for $\delta<1$, $\Lambda$
possesses the Sobolev embedding property
$$
\Real\langle \Lambda f, f\rangle \geq (1-\sqrt{\delta})c_S \|f\|_{2j}^2, \quad j=\frac{d}{d-\alpha},
$$
which plays important role in the existing techniques of obtaining heat kernel estimates.
The Sobolev embedding  is a consequence of 
the Hardy-Rellich inequality and the Sobolev inequality.
However, these arguments become problematic even for $\delta = 1$. When $1<\delta<4$, the operator $(-\Delta)^{\frac{\alpha}{2}} + b \cdot \nabla$ ceases to be quasi-accretive in $L^2$, and the Sobolev embedding property ceases to hold even for some $1<j<\frac{d}{d-\alpha}$. 
Below we show that an operator realization $\Lambda$ of $(-\Delta)^{\frac{\alpha}{2}} + b \cdot \nabla$ can be constructed in $L^r$, $r \in ]\frac{2}{2-\sqrt{\delta}},\infty[$, for every $0<\delta<4$, such that
$$
e^{-t\Lambda}:=s\mbox{-}L^r\mbox{-}\lim_n e^{-t\Lambda(b_n)}, \quad t>0,
$$
where $\{b_n\}$ is appropriate smooth approximation of $b$. The main difficulty here (especially when $1 \leq \delta<4$) is that $\nabla$ is stronger than $(-\Delta)^{\frac{\alpha}{4}}$.
Further, we develop a proper non-symmetric variant of the desingularization method that works for all $0<\delta<4$. We note that the necessity to work not only in $L^2$, but within the entire scale of $L^r$ spaces, $1 \leq r \leq \infty$, is characteristic to the non-symmetric situation, and makes the method of the present paper largely different from the existing ones applicable in symmetric situation.

The method of this paper admits generalization to other operators.

\medskip

Let us now comment on the existing literature.

\smallskip

1.~Sharp two-sided weighted bounds on the heat kernel of the Schr\"{o}dinger operator $-\Delta - V$, $V(x)=\delta(\frac{d-2}{2})^2|x|^{-2}$, $0<\delta \leq 1$, constructed in $L^2$, were obtained in \cite{MS0, MS1,MS}. 
(It is well known that $\delta=1$ is the borderline case for the Schr\"{o}dinger operator, i.e.\,for $\delta>1$  solutions to the corresponding parabolic equation blow up instantly at each point in $\mathbb R^d$,
see \cite{BG,GZ}.) 

In \cite{MS}, the crucial $(L^1,L^1)$ bound is proved  for $-\Delta - V$
by means of the theory of $m$-sectorial operators and the Stampacchia criterion in $L^2$. 
However, attempts to apply that argument to $(-\Delta)^{\frac{\alpha}{2}} + b \cdot \nabla$, $\alpha<2$,  are problematic since $(-\Delta)^{\frac{\alpha}{2}}$ lacks the local properties of $-\Delta$ which makes the corresponding approximation and calculational techniques unusable. 
Below we develop a new approach to the proof of the $(L^1,L^1)$ bound appealing to the Lumer-Phillips Theorem applied to specially constructed $C_0$ semigroups in $L^1$ which approximate $\varphi e^{-t\Lambda}\varphi^{-1}$. The construction of the approximating semigroups is the key observation of this paper.

Thus, in contrast to \cite{MS}, where the $(L^1,L^1)$ bound is proved using the $L^2$ theory, here we prove it while staying within the $L^1$ theory.

Concerning the proof of the lower bound, we note that since $-\Delta$ is a local operator the authors in \cite{MS0, MS1,MS} obtain the lower bound from the a priori bounds corresponding to smooth approximations of the potential and the weights. In the present paper, in the case $\alpha<2$,  in order to obtain the lower bound one has to work directly with singular drift and the weights.

\smallskip

2.~The sharp two-sided weighted bounds on the heat kernel of the fractional Schr\"{o}dinger operator $H=(-\Delta)^{\frac{\alpha}{2}} - \delta  c^2_{d,\alpha} |x|^{-\alpha}$, $0<\alpha<2$, $0<\delta \leq 1$, were obtained in \cite{BGJP}.
In Appendix \ref{app_MS} we discuss how the argument of the present paper can be specified to $H$
to yield a purely operator-theoretic proof of the result of \cite{BGJP}. (Concerning the heat kernel estimates for the operator $(-\Delta)^{\frac{\alpha}{2}} + c|x|^{-\alpha}$, $c>0$, see \cite{CKSV, JW}.)

Concerning operator $H=(-\Delta)^{\frac{\alpha}{2}} - \delta  c^2_{d,\alpha}|x|^{-\alpha}$ in the borderline case $\delta=1$, by \cite[Cor.\,2.5]{FLS} (the fractional variant of the Brezis-Vasquez inequality \cite{BV}) one has $\langle (H+1) f, f \rangle \geq C_d\|f\|_{2j}^2$, $f \in C_c^\infty$, only with  $j \in [1,\frac{d}{d-\alpha}[
$. The latter yields a sub-optimal diagonal estimate ($\equiv$ weighted Nash initial estimate). Nevertheless, as was observed in \cite{BGJP}, this estimate is optimal for $t=1$, and so the optimal weighted Nash initial estimate follows for all $t>0$ due to the scaling properties of $e^{-tH}$. 
This is in contrast to the case $\delta=1$, $\alpha=2$: the scaling properties of $e^{t\Delta}$ are different, so one needs an additional argument in order to obtain the optimal upper bound (i.e.\,to pass to a space of higher dimension where one can appeal to the V.\,P.\,Il'in-Sobolev inequality, see \cite{MS}).

\medskip

Let us comment on the optimality of the constraint $\delta<4$ in \eqref{kappa_}.
In the local case $\alpha=2$, $\delta>4$ destroys the uniqueness of (appropriately defined) weak solution to the corresponding parabolic equation, see \cite[sect.\,4, remark 3]{KiS1} for details; in \cite{FL}, it is demonstrated in dimension $d=3$ that already for $\delta=4$ the properties of the corresponding semigroup are drastically different from the properties of $e^{t\Delta}$ and $e^{-t\Lambda_r}$, $\Lambda_r \supset -\Delta + b \cdot \nabla$, $\delta < 4$. 
We believe that 
the constraint $\delta<4$ (or at most $\delta \leq 4$) is optimal in the non-local case $1<\alpha<2$ for existence of the corresponding Markov semigroup, however, we do not address this issue in this paper.

%

\smallskip

3.~The proof of the weighted Nash initial estimate \eqref{nie_intro} also works for $\alpha=2$ (cf.\,Appendix \ref{app}). The corresponding result, however, is known, see  \cite{MSS}, \cite{MSS2}, where the authors obtain sharp upper and lower bounds on the heat kernel of the operator $-\nabla \cdot (1+c|x|^{-2}x^t x) \cdot \nabla + \delta_1 \frac{d-2}{2}|x|^{-2}x \cdot \nabla + \delta_2\frac{(d-2)^2}{4}|x|^{-2}$, $c>-1$, by considering it in the space $L^2(\mathbb R^d,|x|^\gamma dx)$ where it becomes symmetric (for appropriate constant $\gamma$). This approach, however, does not work for the operator $(-\Delta)^{\frac{\alpha}{2}} + \kappa|x|^{-\alpha}x \cdot \nabla$, $1<\alpha< 2$, cf.\,\cite{BDK}.

\smallskip

We conclude this introduction by emphasizing the following fact ensuing from the previous discussion and well known in the case $\alpha=2$:  the analogy between operators $(-\Delta)^{\frac{\alpha}{2}} + b \cdot \nabla$ with singular $b$ and $(-\Delta)^{\frac{\alpha}{2}} - V$ is superficial.

\medskip


\medskip

\subsection{Notations}

We denote by $\mathcal B(X,Y)$ the space of bounded linear operators between Banach spaces $X \rightarrow Y$, endowed with the operator norm $\|\cdot\|_{X \rightarrow Y}$. Set  $\mathcal B(X):=\mathcal B(X,X)$.

We write $T=s\mbox{-} X \mbox{-}\lim_n T_n$ for $T$, $T_n \in \mathcal B(X)$ if $Tf=\lim_n T_nf$ in $X$ for every $f \in X$. We also write $T_n \overset{s}{\rightarrow} T$ if $X=L^2$.

Denote $\|\cdot\|_{p \rightarrow q}:=\|\cdot\|_{L^p \rightarrow L^q}$.

$L^p_+:=\{f \in L^p \mid f \geq 0 \text{ a.e.}\}$.

$\mathcal S$ denotes the L.\,Schwartz space of test functions.

$C_{u}:=\{f \in C(\mathbb R^d)\mid f \text{ are uniformly continuous and bounded}\}$ (with the $\sup$-norm).

We write $c \neq c(\varepsilon)$ to emphasize that $c$ is independent of $\varepsilon$.


\setcounter{tocdepth}{1}

\tableofcontents

\section{Desingularizing weights}

\label{desing_sect}

Let $X$ be a locally compact topological space, and $\mu$ a $\sigma$-finite Borel measure on $X$. Set $L^p=L^p(X,\mu)$, $p \in [1,\infty]$. 
 
Let $-\Lambda \equiv -\Lambda_r$ be the generator of a $C_0$ semigroup $e^{-t\Lambda}$, $t>0$, in $L^r$ for some $r>1$.
Assume that 
\begin{equation}
\label{S1}
\tag{$S_1$}
\|e^{-t\Lambda}\|_{r \rightarrow \infty} \leq ct^{-\frac{j'}{r}}, \quad j'>0,
\end{equation}
but $e^{-t\Lambda}$ is not an ultra-contraction. In this case we will be assuming that there exist a family of real valued weights $\varphi=\{\varphi_s\}_{s>0}$ in $X$ such that 
\begin{equation}
\label{S2}
\tag{$S_2$}
0 \leq \varphi_s, \frac{1}{\varphi_s} \in L^1_{\loc}(X,\mu) \quad \text{for all $s>0$},
\end{equation}
\begin{equation}
\label{S4}
\tag{$S_3$}
\inf_{s>0,x \in X} \varphi_s(x)  \geq c_0>0,
\end{equation}
and there exists constant $c_1$, independent of $s$, such that, for all $0<t \leq s$,
\begin{equation}
\label{S3}
\tag{$S_4$}
\|\varphi_s e^{-t\Lambda}\varphi_s^{-1}f\|_{1} \leq c_1\|f\|_{1}, \quad f \in L^1 \cap L^\infty.
\end{equation}

The following general theorem is the point of departure for the desingularization method in the
non-selfadjoint setting:

\begin{theoremA}
Assume that  {\rm($S_1$)}\,-\,{\rm($S_4$)} hold.
Then, for each $t>0$, $e^{-t\Lambda}$ is an integral operator, and there is a constant $C=C(j,c_1,c_0)$ such that, up to change of $e^{-t\Lambda}(x,y)$ on a measure zero set, the weighted Nash initial estimate
\[
\label{nie}
\tag{$NIE_w$}
|e^{-t\Lambda}(x,y)|\leq Ct^{-j^\prime}  \varphi_t(y)
\]
is valid for $\mu$ a.e. $x,y \in X$.
\end{theoremA}

The proof of Theorem A uses a weighted variant of the Coulhon-Raynaud Extrapolation Theorem \cite[Prop.\,II.2.1, Prop.\,II.2.2]{VSC}.

\begin{theorem}
\label{thmE2}
Let $U^{t,\theta}$ be a two-parameter family of operators
\[U^{t,\theta}f = U^{t,\tau}U^{\tau,\theta}f,  \qquad f \in L^1 \cap L^\infty, \quad 0 \leq \theta < \tau < t \leq \infty.
\]
Suppose that for some $1 \leq p < q < r \leq \infty$, $\nu>0$
\begin{align*}
\| U^{t,\theta} f \|_p & \leq M_1 \| f \|_{p,\sqrt{\psi}}, \quad 0 \leq \psi \in L^1+L^\infty,  \quad \|f\|_{p,\sqrt{\psi}}:=\langle |f|^p \psi \rangle^{1/p},\\
 \| U^{t,\theta} f \|_r & \leq M_2 (t-\theta)^{-\nu} \|  f \|_q
\end{align*}
for all $(t,\theta)$ and $f \in L^1 \cap L^\infty.$ Then
\[
\| U^{t,\theta} f \|_r \leq M (t-\theta)^{-\nu/(1-\beta)} \| f \|_{p,\sqrt{\psi}} ,
\]
where $\beta = \frac{r}{q}\frac{q-p}{r-p}$ and $M = 2^{\nu/(1-\beta)^2} M_1 M_2^{1/(1-\beta)}.$
\end{theorem}

\begin{proof}[Proof of Theorem \ref{thmE2}]

We have ($t_\theta:=\frac{t+\theta}{2}$)
\begin{align*}
\| U^{t, \theta} f \|_r & \leq M_2 (t-t_\theta)^{-\nu} \| U^{t_\theta,\theta} f \|_q \\
& \leq M_2 (t-t_\theta)^{-\nu} \| U^{t_\theta,\theta} f \|_r^\beta \;\| U^{t_\theta,\theta} f \|_p^{1-\beta} \\
& \leq M_2 M_1^{1-\beta} (t-t_\theta)^{-\nu} \| U^{t_\theta,\theta} f \|_r^\beta \;\| f \|_{p,\sqrt{\psi}}^{1-\beta},
\end{align*}
and hence
\[
(t-\theta)^{\nu/(1-\beta)} \| U^{t,\theta} f \|_r/\| f \|_{p,\sqrt{\psi}} \leq M_2 M_1^{1-\beta} 2^{\nu/(1-\beta)} \big [(t - \theta)^{\nu/(1-\beta)} \| U^{t_\theta,\theta} f \|_r\;/\| f \|_{p,\sqrt{\psi}} \big ]^\beta.
\]
Setting $R_{2 T}: = \sup_{t-\theta \in ]0,T]} \big [ (t-\theta)^{\nu/(1-\beta)} \| U^{t,\theta} f \|_r/\| f \|_{p,\sqrt{\psi}} \big ],$ we obtain from the last inequality that $R_{2 T} \leq M^{1-\beta} (R_T)^\beta.$ But $R_T \leq R_{2T}$, and so $R_{2T} \leq M.$
The proof of Theorem \ref{thmE2} is completed.
\end{proof}

\begin{proof}[Proof of Theorem A]
By \eqref{S4} and \eqref{S3},
\begin{align*}
\|e^{-t\Lambda}h\|_{1} &\leq c_0^{-1} \|\varphi_s e^{-t\Lambda} \varphi_s^{-1} \varphi_s h \|_1  \\
& \leq c_0^{-1}c_1 \|h\|_{1,\sqrt{\varphi_s}}, \qquad  h \in L^\infty_{\rm com}.
\end{align*}
The latter, \eqref{S1} and Theorem \ref{thmE2} with $\psi:=\varphi_s$ yield
$$
\|e^{-t\Lambda}f\|_{\infty} \leq Mt^{-j'}\|\varphi_s f\|_1, \quad 0<t \leq s, \quad f \in  L^\infty_{\rm com}.
$$

Note that ($S_1$) verifies assumptions of Gelfand's Theorem, which then yields that $e^{-t\Lambda}$ is an integral operator.
Therefore, taking $s=t$ in the previous estimate, we obtain \eqref{nie}.
\end{proof}

\begin{remarks}

1.~Although in the end of the proof we take $s=t$, having $s$ and $t$ ``decoupled'', here and in other proofs below, allows us to treat the weight as time independent.

2.~The proof of Theorem A, as well as the proofs of Theorems \ref{thm2} and \ref{thm2_} below, are based on ideas of
J.Nash \cite{N}.

3.~\eqref{S1} can be viewed as a variant of the Sobolev embedding property of $\Lambda$.

4.~In applications of Theorem A to concrete operators the main difficulty consists in verification of the assumption ($S_4$).
\end{remarks}

\bigskip

\section{Heat kernel $e^{-t\Lambda}(x,y)$ of $\Lambda \supset (-\Delta)^{\frac{\alpha}{2}} + \kappa|x|^{-\alpha}x \cdot \nabla$, $1<\alpha<2$}

\label{heat_sect}

We now discuss in detail our main result concerning $(-\Delta)^{\frac{\alpha}{2}} + b \cdot \nabla$,
$b(x):=\kappa |x|^{-\alpha}x$, where, recall,
$$
\kappa=\sqrt{\delta} \frac{2^{\alpha+1}}{d-\alpha}\frac{\Gamma^2(\frac{d}{4}+\frac{\alpha}{4})}{\Gamma^2(\frac{d}{4}-\frac{\alpha}{4})}, \quad 0<\delta<4.
$$

\smallskip

\textbf{1.~}Let us first construct an operator realization $-\Lambda$ of $-(-\Delta)^{\frac{\alpha}{2}} - b \cdot \nabla$ as the generator of a $C_0$ semigroup in an appropriate $L^r$. The heat kernel of $\Lambda$ will be defined as the integral kernel of the semigroup.

Throughout this paper, $(-\Delta)^{\frac{\alpha}{2}}$ is the fractional power of $-\Delta$ in the sense of Balakrishnan (see e.g.\,\cite[Sect.\,IX.11]{Yos}), where $-\Delta$ is defined as the generator of heat semigroup
in $L^p$, $1 \leq p<\infty$ or in $C_{u}$; one has $D((-\Delta)^{\frac{\alpha}{2}}) \supset D(-\Delta)$. In what follows, we write $(-\Delta)^{\frac{\alpha}{2}}_p$ or $(-\Delta)^{\frac{\alpha}{2}}_{C_u}$ when we want to emphasize that 
$(-\Delta)^{\frac{\alpha}{2}}$ is defined in $L^p$ on in $C_u$, respectively.

\smallskip

In $L^p$, $1 \leq p<\infty$, and $C_{u}$ define the approximating operators
$$
P^\varepsilon:= (-\Delta)^\frac{\alpha}{2} + b_\varepsilon \cdot \nabla+U_\varepsilon,$$
$$
 D(P^\varepsilon)=D((-\Delta)_p^\frac{\alpha}{2})=\mathcal W^{\alpha,p}:= \bigl(1+(-\Delta)^{\frac{\alpha}{2}}\bigr)^{-1}L^p,$$ 
$$D(P^\varepsilon)=D((-\Delta)^{\frac{\alpha}{2}}_{C_u})=\bigl(1+(-\Delta)^{\frac{\alpha}{2}}\bigr)^{-1}C_u, \quad \text{respectively,}
$$
where $\varepsilon>0$, 
$$b_\varepsilon(x)=\kappa|x|_\varepsilon^{-\alpha}x, \quad |x|_\varepsilon:=\sqrt{|x|^2+\varepsilon}, \qquad
U_\varepsilon(x):= \alpha \kappa \varepsilon |x|_\varepsilon^{-\alpha-2}.
$$
(The auxiliary potentials $U_\varepsilon$ are needed to carry out the estimates in the proof of Proposition \ref{main_prop4}.)

By the Hille Perturbation Theorem \cite[Ch.\,IX, sect.\,2.2]{Ka}, $-P^\varepsilon$ is the generator of a holomorphic semigroup in $L^p$, $1 \leq p<\infty$, and $C_{u}$. 
Similarly, $-\Lambda^\varepsilon$, where $\Lambda^\varepsilon:=(-\Delta)^{\frac{\alpha}{2}}+b_\varepsilon \cdot \nabla$, generates a holomorphic semigroup in $L^p$ and $C_u$ (for details, if needed, see Remark \ref{rem_hille} below). 
It is well known that $$e^{-t\Lambda^\varepsilon}L^p_+\subset L^p_+,$$ and so $e^{-tP^\varepsilon}L^p_+\subset L^p_+$. Also, $$\|e^{-tP^\varepsilon}f\|_\infty \leq \|e^{-t\Lambda^\varepsilon}f\|_\infty \leq \|f\|_\infty, \quad f \in L^p \cap L^\infty.$$

We now consider two cases:

\medskip

(a) First, let $0<\delta<1$. Since $b \cdot \nabla$ is a Rellich's perturbation of $(-\Delta)^{\frac{\alpha}{2}}$ in $L^2$ (for details, if needed, see Proposition \ref{prop_2} below), the algebraic 
sum $\Lambda:=(-\Delta)^\frac{\alpha}{2}+b\cdot\nabla$, $D(\Lambda)=\mathcal W^{\alpha,2}$, is the (minus) generator of a holomorphic semigroup in $L^2$, with the property
$$ 
e^{-t\Lambda}=s\mbox{-}L^2\mbox{-}\lim_{\varepsilon \downarrow 0}e^{-tP^\varepsilon}=s\mbox{-}L^2\mbox{-}\lim_{\varepsilon \downarrow 0}e^{-t\Lambda^\varepsilon} \quad (\text{locally uniformly in $t \geq 0$}),
$$
see Propositions \ref{prop_2} and \ref{prop_2_conv} below. Since $e^{-t\Lambda}$ is an $L^\infty$ contraction, $e^{-t\Lambda}$ extends by continuity to $C_0$ semigroup on $L^r$ for every $r \in [2,\infty[$:
$$e^{-t\Lambda_r}:=\big[e^{-t\Lambda} \upharpoonright L^2 \cap L^r\big]^{\rm clos}_{L^r \rightarrow L^r},$$
such that
$e^{-t\Lambda_r}=s\mbox{-}L^r\mbox{-}\lim_{\varepsilon \downarrow 0}e^{-tP^\varepsilon}$ 
(see, if needed, Remark \ref{rem_conv_2} below).


\smallskip

(b) For $1 \leq \delta<4$, we prove in Proposition \ref{prop4}(\textit{i}) below that there exists a sequence $\varepsilon_i \downarrow 0$ such that, for every $r \in ]r_c,\infty[$, $r_c:=\frac{2}{2-\sqrt{\delta}}$, the limit
\begin{equation}
\label{conv_def}
s\mbox{-}L^r\mbox{-}\lim_{i} e^{-tP^{\varepsilon_i}} \quad (\text{locally uniformly in $t \geq 0$}),
\end{equation}
 exists and determines a $L^\infty$ contraction, contraction $C_0$ semigroup $e^{-t\Lambda_r}$ in $L^r$; $\Lambda_r$ is an operator realization of $(-\Delta)^{\frac{\alpha}{2}} + b \cdot \nabla$ in $L^r$. 

Thus, 
while for $0<\delta<1$ \eqref{conv_def} can be viewed as a (fundamental) property of the semigroup $e^{-t\Lambda_r}$, for $1 \leq \delta < 4$ \eqref{conv_def} becomes the principal means of construction of $e^{-t\Lambda_r}$.


The semigroups $e^{-t\Lambda_r}$ are consistent: 
\begin{equation}
\label{consist}
e^{-t\Lambda_r} \upharpoonright L^r \cap L^p = e^{-t\Lambda_p} \upharpoonright L^r \cap L^p.
\end{equation}
Since $e^{-tP^{\varepsilon_i}}$ are positivity preserving, so are $e^{-t\Lambda_r}$, $r \in ]r_c,\infty[$.

By consistency, if $e^{-t\Lambda_r}$ is an integral operator for some $r$, then $e^{-t\Lambda_r}$ are integral operators for all $r \in ]r_c,\infty[$ having the \textit{same}, up to change on measure zero set in $\mathbb R^d \times \mathbb R^d$, integral kernel $e^{-t\Lambda}(x,y) \geq 0$. This integral kernel is called the \textit{heat kernel} of the operators $\Lambda_r$.

Let us also note that, by Proposition \ref{prop4}(\textit{iii}), $v=(\mu+\Lambda_r)^{-1}f$, $\mu>0$, $f \in L^r$, is a weak solution to the equation $\mu v + (-\Delta)^{\frac{\alpha}{2}} v+b \cdot \nabla v=f$.

\medskip

\textbf{2.~}We now define the desingularizing weights for $e^{-t\Lambda_r}$. 
Define constant $\beta$ by \eqref{beta_eq}.
Direct calculations show that such $\beta \in ]\alpha,d[$ exists and is unique for all $0<\delta\leq 4$. See Figure \ref{fig1}.
 This choice of $\beta$ entails that $|x|^{-d+\beta}$ is a Lyapunov function to the formal operator $\Lambda^*=(-\Delta)^{\frac{\alpha}{2}}-\nabla \cdot b$, i.e.\,$\Lambda^*|x|^{-d+\beta}=0$ (for details, see Remark \ref{lyapunov_rem} in the next section).

\begin{figure}
\begin{center}
\includegraphics[width=0.6\textwidth]{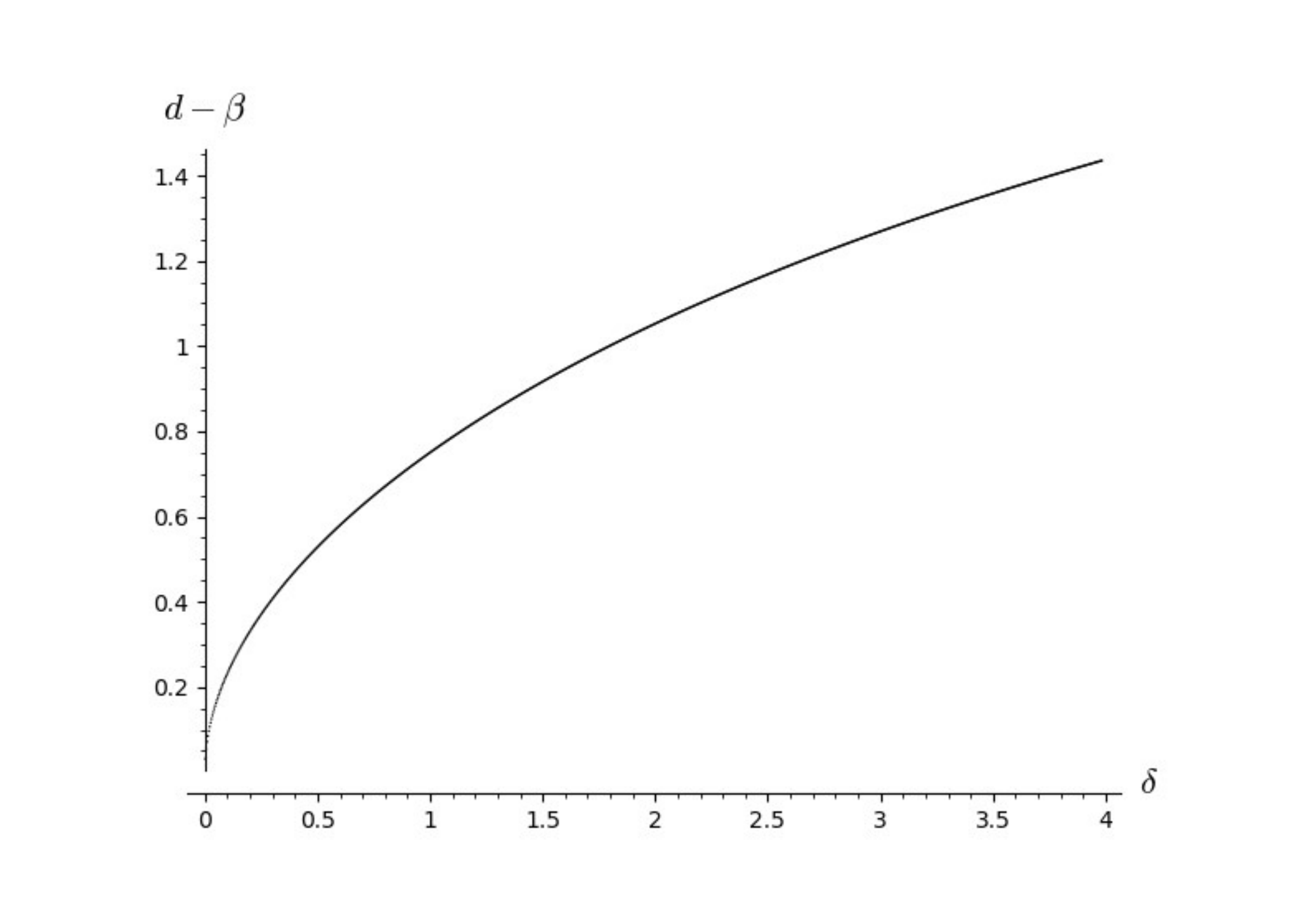}
\end{center}
\vspace*{-5mm}
\caption{The function $\delta \mapsto d-\beta$ for $d=3$ and $\alpha=\frac{3}{2}$. \label{fig1}}
\end{figure}

We define the weight:
$$
\varphi_t(y)=\eta(t^{-\frac{1}{\alpha}}|y|), \quad y \in \mathbb R^d, t>0,
$$
where $\frac{1}{2} \leq \eta$ is a $C^2(]0,\infty[)$ function such that
\[
\eta(r)=\left \{
\begin{array}{ll}
r^{-d+\beta}, & 0<r< 1,\\
\frac{1}{2}, & r\geq 2.
\end{array}
\right.
\]
Note that
$$
\varphi_t(y) \approx (t^{-\frac{1}{\alpha}}|y| \wedge 1)^{-d+\beta}.
$$

\begin{theorem}
\label{thm2}Let $0<\delta<4$. 
Then $e^{-t\Lambda_r}$, $r \in ]\frac{2}{2-\sqrt{\delta}},\infty[$ is an integral operator for each $t > 0$. Let $e^{-t\Lambda}(x,y)$ be the (common) integral kernel of $e^{-t\Lambda_r}$. Then there exists a constant $C$ such that, up to change of $e^{-t\Lambda}(x,y)$ on a measure zero set, the
weighted Nash initial estimate
\begin{equation}
\tag{$\text{\rm NIE}_w$}
e^{-t\Lambda}(x,y)\leq Ct^{-\frac{d}{\alpha}}\varphi_t(y), 
\end{equation}
is valid for all $x,y\in \mathbb R^d$, $y\neq 0$ and $t>0$.
\end{theorem}

Having at hand Theorem \ref{thm2}, we obtain below  the following

\begin{theorem}
\label{thm2_}
Let $0<\delta<4$. Then, up to change of $e^{-t\Lambda}(x,y)$ on a measure zero set,
$$e^{-t\Lambda}(x,y)\approx e^{-t(-\Delta)^\frac{\alpha}{2}}(x,y)\varphi_t(y) $$
for all $t>0$, $x,y \in \mathbb R^d$, $y \neq 0$.
\end{theorem}


\begin{remark}
\label{rem_hille}
1.~In the proof that $-P^\varepsilon$ is the generator of a holomorphic semigroup in $L^p$, $1 \leq p<\infty$, and $C_{u}$, we use a well known estimate (set $A \equiv (-\Delta)^{\frac{\alpha}{2}}$)
$$
|\nabla \big(\zeta + A\big)^{-1}(x,y)| \leq C \bigl(\Real \zeta + A\bigr)^{-1+\frac{1}{\alpha}}(x,y), \quad \Real \zeta>0.
$$
Then (for $Y=L^p$ or $C_u$) $$\|b_\varepsilon \cdot \nabla \big(\zeta + A\big)^{-1}\|_{Y \rightarrow Y} \leq C\|b_\varepsilon\|_\infty  \|\big(\Real \zeta + A\big)^{-1+\frac{1}{\alpha}})\|_{Y \rightarrow Y} \leq C\|b_\varepsilon\|_\infty (\Real \zeta)^{-1+\frac{1}{\alpha}},$$
and so $\|b_\varepsilon \cdot \nabla \big(\zeta + A\big)^{-1}\|_{Y \rightarrow Y}$, $\Real \zeta \geq c_\varepsilon$, can be made arbitrarily small by selecting $c_\varepsilon$ sufficiently large.
Similar argument applies to $\|U_\varepsilon (\zeta + A)^{-1}\|_{Y \rightarrow Y}$.
It follows that the Neumann series for $$(\zeta + P^\varepsilon)^{-1}=(\zeta + A)^{-1}(1+T)^{-1}, \quad T:=(b_\varepsilon \cdot \nabla + U_\varepsilon) (\zeta + A)^{-1},$$converges in $L^p$ and $C_u$ and satisfies $\|(\zeta + P^\varepsilon)^{-1}\|_{Y \rightarrow Y} \leq C|\zeta|^{-1}$, $\Real \zeta \geq c_\varepsilon$, i.e.\,$-P^\varepsilon$ is the generator of a holomorphic semigroup.

\smallskip

2.~The same argument yields: in $L^p$, $1 \leq p<\infty$, and $C_{u}$, $$(P^\varepsilon)^*:=(-\Delta)^\frac{\alpha}{2} - \nabla \cdot b_\varepsilon + U_\varepsilon=(-\Delta)^\frac{\alpha}{2} - b_\varepsilon \cdot \nabla - W_\varepsilon,$$
$$
D((P^\varepsilon)^*)=D((-\Delta)_p^{\frac{\alpha}{2}}) \quad \text{ or } \quad D((P^\varepsilon)^*)=D((-\Delta)_{C_u}^{\frac{\alpha}{2}}), \text{ respectively},$$
where $W_\varepsilon(x)=(d-\alpha)\kappa|x|_\varepsilon^{-\alpha}$, is the (minus) generator of a holomorphic semigroup.

Then for all $\varepsilon>0$, $p \in ]1,\infty[$,
$$
\langle (\mu+P^\varepsilon)^{-1}f,g \rangle = \langle f,(\mu+(P^\varepsilon)^*)^{-1}g \rangle, \quad \mu \geq c_\varepsilon,
$$
$$
\langle e^{-tP^\varepsilon}f,g\rangle=\langle f,e^{-t(P^\varepsilon)^*}g\rangle, \quad t>0, \quad f \in L^{p}, \quad g \in L^{p'},
$$
and $e^{-t(P^\varepsilon)^*}(x,y)=e^{-tP^\varepsilon}(y,x)$, $t>0$, $x,y \in \mathbb R^d$.
\end{remark}

\begin{remark}
\label{rem_conv_2} Above we used the following elementary result:
\begin{quote}
Let $S_k : L^{p_1} \cap L^{p_2} \rightarrow L^{p_1} \cap L^{p_2}, \; 1 \leq p_1 < p_2 \leq \infty, \; k=0, 1, 2, \dots,$ be such that $\|S_k f\|_{p_i} \leq M \|f\|_{p_i}, \; i = 1, 2,$ for all $f \in L^{p_1}\cap L^{p_2}, \; k$ and some $M < \infty.$ If $\|S_k f - S_0 f \|_{p_0} \rightarrow 0$ for some $p_0 \in ]p_1, p_2[,$ then $\|S_k f - S_0 f \|_p \rightarrow 0$ for every $p \in ]p_1, p_2[.$   
\end{quote}
\end{remark}


\section{Proof of Theorem \ref{thm2}}

\label{proof_thm2_sect}

In this proof, and in other proofs below, we will be using the fractional Hardy-Rellich inequality:
\begin{equation}
\label{hardy_ineq}
\| (-\Delta)^{\frac{\alpha}{4}}f\|_2^2 \geq c^2_{d,\alpha} \||x|^{-\frac{\alpha}{2}}f\|_2^2, \quad f \in \mathcal W^{\frac{\alpha}{2},2},
\end{equation}
where, recall, 
$$
c_{d,\alpha}:=2^{\frac{\alpha}{2}}\frac{\Gamma(\frac{d}{4}+\frac{\alpha}{4})}{\Gamma(\frac{d}{4}-\frac{\alpha}{4})}
$$
(see \cite[Lemma 2.7]{KPS}).

Let us write the coefficient $\kappa$ in $b(x)=\kappa|x|^{-\alpha}x$ (cf.\,\eqref{kappa_}) in the form 
\[
\kappa=\sqrt{\delta} (d-\alpha)^{-1}2c^2_{d,\alpha}.
\]

\medskip

\noindent\textbf{1.~}We are going to verify the assumptions of Theorem A for $\Lambda_r$.

\medskip

($S_1$):  $\|e^{-t\Lambda_r}\|_{r \rightarrow \infty} \leq ct^{-\frac{j'}{r}}$, $j'=\frac{d}{\alpha}$, $r \in ]r_c,\infty[$, $r_c:=\frac{2}{2-\sqrt{\delta}}$.
 
\begin{proof}[Proof of ($S_1$)](a) First, we are going to prove ($S_1$) for $e^{-tP^\varepsilon}$ for all $r \in [r_c,\infty[$, where, recall,
$$
P^\varepsilon:= (-\Delta)^\frac{\alpha}{2} + b_\varepsilon \cdot \nabla+U_\varepsilon,
$$
with $b(x)=\kappa|x|_\varepsilon^{-\alpha}x$ and  $U_\varepsilon(x)= \alpha \kappa \varepsilon |x|_\varepsilon^{-\alpha-2}$, $|x|_\varepsilon:=\sqrt{|x|^2+\varepsilon}$.

Set $A \equiv (-\Delta)^\frac{\alpha}{2}$. Set $0 \leq u=e^{-tP^\varepsilon}f$, $f \in L^1_+ \cap L^\infty$. Then, multiplying the equation $(\partial_t + P^\varepsilon)u=0$ by $u^{r-1}$ and integrating, we obtain 
\begin{equation}
\label{main_id}
-\frac{1}{r}\frac{d}{dt}\|u\|_r^r = \langle Au, u^{r-1}\rangle + \frac{2}{r}\kappa \langle |x|_\varepsilon^{-\alpha}x \cdot \nabla u^{\frac{r}{2}},u^{\frac{r}{2}} \rangle + \langle U_\varepsilon u^{\frac{r}{2}}, u^\frac{r}{2}\rangle.
\end{equation}

In the RHS of \eqref{main_id}, we estimate
$$\langle Au, u^{r-1}\rangle \geq \frac{4}{rr'} \|A^{\frac{1}{2}}u^{\frac{r}{2}}\|_2^2 \qquad \text{ by \cite[Theorem 2.1]{LS}},$$
and, integrating by parts, find
$$\kappa \langle |x|_\varepsilon^{-\alpha}x \cdot \nabla u^{\frac{r}{2}},u^{\frac{r}{2}} \rangle=-\kappa \frac{d-\alpha}{2} \langle |x|_\varepsilon^{-\alpha}u^{\frac{r}{2}},u^{\frac{r}{2}}\rangle-\frac{1}{2}\langle U_\varepsilon u^{\frac{r}{2}}, u^\frac{r}{2}\rangle.$$
Applying the last two estimates in \eqref{main_id}, we obtain
\begin{align*}
-\frac{1}{r}\frac{d}{dt}\|u\|_r^r & \geq \frac{4}{rr'} \|A^{\frac{1}{2}}u^{\frac{r}{2}}\|_2^2-\frac{2}{r}\kappa \frac{d-\alpha}{2} \langle |x|_\varepsilon^{-\alpha}u^{\frac{r}{2}},u^{\frac{r}{2}}\rangle + \bigl(1-\frac{1}{r}\bigr)\langle U_\varepsilon u^{\frac{r}{2}}, u^\frac{r}{2}\rangle \\
& \geq \frac{4}{rr'} \|A^{\frac{1}{2}}u^{\frac{r}{2}}\|_2^2-\frac{2}{r}\kappa \frac{d-\alpha}{2} \langle |x|^{-\alpha}u^{\frac{r}{2}},u^{\frac{r}{2}}\rangle.
\end{align*}
Therefore, using the Hardy-Rellich inequality \eqref{hardy_ineq}, we obtain
\begin{equation}
\label{est_n}
\tag{$\ast$}
-\frac{d}{dt}\|u\|_r^r \geq \biggl(\frac{4}{r'} - 2\sqrt{\delta} \biggr)\|A^{\frac{1}{2}}u^{\frac{r}{2}}\|_2^2,
\end{equation}
where $\frac{2}{r'}-\sqrt{\delta}\geq 0$ (which is our assumption $r \geq \frac{2}{2-\sqrt{\delta}}$).
In particular, integrating in $t$, we obtain
\begin{equation}
\label{est_c}
\tag{$\ast\ast$}
\|u(t)\|_{r} \leq \|f\|_{r}, \quad r \in [r_c,\infty],  \quad t>0.
\end{equation}

Now, let $r=2r_c$ ($\in ]r_c,\infty[$). Using the Nash inequality $\|A^{\frac{1}{2}}h\|_2^2 \geq C_N\|h\|_2^{2+\frac{2\alpha}{d}}\|h\|_1^{-\frac{2\alpha}{d}}$,
we obtain from \eqref{est_n} (put $w:=\|u\|_r^r$)
$$
\frac{d}{dt}w^{-\frac{\alpha}{d}} \geq c_2\|f\|^{-\frac{r\alpha}{d}}_{\frac{r}{2}}, \quad c_2=C_N\frac{\alpha}{d}\bigl(\frac{4}{r'} - 2\sqrt{\delta} \bigr).
$$
Integrating this inequality, we obtain
$$
\|e^{-tP^\varepsilon}\|_{r_c \rightarrow 2r_c} \leq c_2^{-\frac{d}{2\alpha r_c}}t^{-\frac{d}{\alpha}(\frac{1}{r_c}-\frac{1}{2r_c})}, \quad t>0.
$$

Using the latter and $\|e^{-tP^\varepsilon}f\|_{\infty} \leq \|f\|_\infty$, we pass to the adjoint semigroup:
$$
\|e^{-t(P^\varepsilon)^*}\|_{(2r_c)' \rightarrow r_c'} \leq c_2^{-\frac{d}{2\alpha r_c}}t^{-\frac{d}{\alpha}(\frac{1}{r_c}-\frac{1}{2r_c})}, \qquad \|e^{-t(P^\varepsilon)^*}\|_{1 \rightarrow 1} \leq 1.
$$
Applying Theorem \ref{thmE2} (with $\psi \equiv 1$ and $p:=1$, $q:=(2r_c)'$, $r:=r_c'$), we obtain: $\|e^{-t(P^\varepsilon)^*}\|_{1 \rightarrow r_c'} \leq c t^{-\frac{d}{\alpha}\frac{1}{r_c}}$, $c \neq c(\varepsilon)$. By duality, estimate ($S_1$) for $e^{-tP^\varepsilon}$ follows ($j'=\frac{d}{\alpha}$):
\begin{equation}
\label{P_e}
\|e^{-tP^\varepsilon}\|_{r \rightarrow \infty} \leq ct^{-\frac{j'}{r}}
\end{equation}
(for $r=r_c$, and by interpolation for all $r \in [r_c,\infty[$).

\medskip

(b) Let $r \in ]r_c,\infty[$. We now pass to the limit $\varepsilon \downarrow 0$ in \eqref{P_e}. 
We have
$$
|\langle e^{-tP^\varepsilon} f,g\rangle| \leq C t^{-\frac{j'}{r}}\|f\|_r\|g\|_1, \quad f \in L^r,g \in L^1 \cap L^{\infty}.
$$
According to Proposition \ref{prop4}(\textit{i}), 
$e^{-tP^{\varepsilon_i}}f \rightarrow e^{-t\Lambda_r}f$ strongly in $L^r$ for some $\varepsilon_i \downarrow 0$ for every $f \in L^r$. Therefore,
$$
|\langle e^{-t\Lambda_r} f,g\rangle | \leq C t^{-\frac{j'}{r}}\|f\|_r\|g\|_1, \quad f \in L^r,g \in L^1 \cap L^{\infty},
$$
and hence 
$$
\|e^{-t\Lambda_r} f\|_\infty \leq C t^{-\frac{j'}{r}}\|f\|_r,
$$
i.e.\,we have proved ($S_1$).
\end{proof}

By the construction of $\varphi_s$,

$(S_2),(S_3)$: \textit{$\varphi_s^{\pm 1}\in L^1_\loc$ and $\inf_{s>0, \;x\in \mathbb R^d} \varphi_s(x) \geq \frac{1}{2}$} are valid.

\smallskip

($S_4$): \textit{There exists a constant $c_1>0$ such that, for all $r \in ]r_c,\infty[$, $0<t\leq s$ }
\[
 \|\varphi_s e^{-t\Lambda_r}\varphi_s^{-1}h\|_1\leq c_1\|h\|_{1}, \quad h\in L^1 \cap L^\infty.
\]
See the proof of ($S_4$) below.

\medskip

Thus, Theorem A applies and yields that $e^{-t\Lambda_r}$ is an integral operator for every $t>0$,
and
\begin{equation*}
e^{-t\Lambda}(x,y) \leq Ct^{-j^\prime}\varphi_t(y), \quad t>0, x,y \in \mathbb R^d, y \neq 0.
\end{equation*}
Theorem \ref{thm2} follows.

%
%

It remains to prove ($S_4$). This presents the main difficulty.

\bigskip

\noindent\textbf{2.~Proof of {\rm($S_4$)}.}
We are going to first prove ($S_4$) for approximating semigroups $e^{-tP^\varepsilon}$ by showing that the generator of the weighted semigroup $\phi_n e^{-t(\lambda + P^\varepsilon)} \phi_n^{-1}$, for appropriate regularization $\phi_n$ ($n=1,2,\dots$) of $\varphi_s$, is accretive in $L^1$ and has dense range, and so $\phi_n e^{-tP^\varepsilon} \phi_n^{-1}$ is a quasi contraction in $L^1$ ($\Rightarrow$ ($S_4$)). 

In the beginning of Section \ref{heat_sect} we introduced (in $L^1$): $$
P^\varepsilon:= (-\Delta)^\frac{\alpha}{2} + b_\varepsilon \cdot \nabla+U_\varepsilon, \qquad D(P^\varepsilon)=D((-\Delta)_1^\frac{\alpha}{2})\equiv \mathcal W^{\alpha,1}\;(=\bigl(1+(-\Delta)^{\frac{\alpha}{2}}\bigr)^{-1}L^1),$$
$$(P^\varepsilon)^*:=(-\Delta)^\frac{\alpha}{2} - \nabla \cdot b_\varepsilon + U_\varepsilon=(-\Delta)^\frac{\alpha}{2} - b_\varepsilon \cdot \nabla - W_\varepsilon, \qquad D((P^\varepsilon)^*)=D((-\Delta)_1^\frac{\alpha}{2}),$$
where $W_\varepsilon(x)=(d-\alpha)\kappa|x|_\varepsilon^{-\alpha}$.

Recall that, for each $\varepsilon>0$, both $e^{-t P^\varepsilon}$, $e^{-t(P^\varepsilon)^*}$ can be viewed as $C_0$ semigroups in  $L^1$ and $C_{u}$, with $D(P^\varepsilon)=D((P^\varepsilon)^*)$ being the same as the domain of $(-\Delta)^{\frac{\alpha}{2}}$ in $L^1$ and $C_u$, respectively. See Remark \ref{rem_hille}.

Define a regularization  of $\varphi$, $$\phi_n=e^{-\frac{(P^\varepsilon)^*}{n}}\varphi, \qquad  n=1,2,\dots,$$
where, here and in the rest of the proof, we write for brevity
$$
\varphi \equiv \varphi_s.
$$
Since  we can represent $\varphi=\varphi_{(1)} + \varphi_{(u)}$,  
$\varphi_{(1)} \in L^1$, $\varphi_{(u)} \in C_u$, in view of the previous comment 
the weights $\phi_n$ are well defined. It is clear from the definition that $\phi_n \rightarrow \varphi$ in $L^{1}_{\loc}$ as $n \rightarrow \infty$.

Note that $\phi_n \geq \frac{1}{2}$ (indeed, since $e^{-t(P^\varepsilon)^*}$ is positivity preserving and $\varphi \geq \frac{1}{2}$, we have $\phi_n \geq \frac{1}{2} e^{-\frac{(P^\varepsilon)^*}{n}}1 \geq \frac{1}{2} e^{-\frac{1}{n}((-\Delta)^{\frac{\alpha}{2}} - b_\varepsilon \cdot \nabla)} 1=\frac{1}{2}$).

\begin{remark*}
We emphasize that this choice of  the regularization  of $\varphi$ is the key observation that allows us to carry out the method in the case $\alpha<2$, cf.\,\eqref{key} below.
\end{remark*}

In $L^1$ define operators
\[
Q=\phi_n P^\varepsilon \phi_n^{-1}, \quad D(Q)=\phi_n D(P^\varepsilon)=\phi_n D((-\Delta)^\frac{\alpha}{2}), 
\]
where $\phi_n D(P^\varepsilon):=\{\phi_n u \mid u \in D(P^\varepsilon)\}$, 
\[
F_{\varepsilon,n}^t=\phi_n e^{-tP^\varepsilon}\phi_n^{-1}. 
\]
Since $\phi_n, \phi_n^{-1}\in L^\infty$, these operators are well defined.
In particular, $F^t_{\varepsilon,n}$ are bounded $C_0$ semigroups in $L^1$. Write $F^t_{\varepsilon,n}=e^{-tG}$. 

Set 
\begin{align*}
M:=&\,\phi_n(1+(-\Delta)^{\frac{\alpha}{2}})^{-1}[L^1 \cap C_{u}] \\
=&\,\phi_n (\lambda_\varepsilon+P^\varepsilon)^{-1}[L^1 \cap C_{u}], \quad 0<\lambda_\varepsilon\in \rho(-P^\varepsilon).
\end{align*}
Clearly, $M$ is a dense subspace of $L^1$, $M\subset D(Q)$ and $M\subset D(G)$. Moreover, $Q\upharpoonright M\subset G$. Indeed, for $f=\phi_n u\in M$,
\[
Gf=s\mbox{-}L^1\mbox{-}\lim_{t\downarrow 0}t^{-1}(1-e^{-tG})f=\phi_n s\mbox{-}L^1\mbox{-}\lim_{t\downarrow 0}t^{-1}(1-e^{-tP^\varepsilon})u=\phi_nP^\varepsilon u=Qf.
\]
Thus $Q\upharpoonright M$ is closable and $\tilde{Q}:=(Q\upharpoonright M)^{\rm clos}\subset G$.

\begin{proposition}
\label{dense_prop}
The range $R(\lambda_\varepsilon +\tilde{Q})$ is dense in $L^1$.
\end{proposition}

\begin{proof}[Proof of Proposition \ref{dense_prop}]
If $\langle(\lambda_\varepsilon +\tilde{Q})h,v\rangle=0$ for all $h\in D(\tilde{Q})$ and some $v\in L^\infty$, $\|v\|_\infty=1$, then taking $h\in M$ we would have $\langle(\lambda_\varepsilon+Q)\phi_n (\lambda_\varepsilon+P^\varepsilon)^{-1}g,v\rangle=0$, $g\in L^1 \cap C_{u}$, or $\langle\phi_n g,v\rangle=0$. Choosing $g=e^\frac{\Delta}{k}(\chi_m v)$, where $\chi_m\in C^\infty_c$ with $\chi_m(x)=1$ when $x\in B(0,m)$, we would have $\lim_{k\uparrow\infty}\langle \phi_n g,v\rangle=\langle\phi_n\chi_m,|v|^2\rangle=0$, and so $v\equiv 0$. Thus,  $R(\lambda_\varepsilon +\tilde{Q})$ is dense in $L^1$.
\end{proof}

The following is the main step in the proof.

\begin{proposition}
\label{main_prop4}
There is a constant $\hat{c}=\hat{c}(d,\alpha,\delta)$ such that for every $\varepsilon>0$ and all $n \geq n(\varepsilon)$,
 \[
  \lambda+\tilde{Q} \textit{ is accretive whenever } \lambda\geq \hat{c} s^{-1},
 \] 
i.e.
$$
\Real\big\langle (\lambda + \tilde{Q})f,\frac{f}{|f|} \big\rangle \geq 0, \quad f \in D(\tilde{Q}).$$
\end{proposition}


\begin{proof}[Proof of Proposition \ref{main_prop4}]
Below we will be appealing to Appendix \ref{appA__}, where we use convenient notation
$$\gamma(s):=\frac{2^s\pi^\frac{d}{2}\Gamma(\frac{s}{2})}{\Gamma(\frac{d}{2}-\frac{s}{2})}.$$ 

Recall that both $e^{-tP^\varepsilon}$, $e^{-t(P^\varepsilon)^*}$ are holomorphic in $L^1$ and $C_u$, see Remark \ref{rem_hille}. Note that we can represent $$\varphi=\varphi_{(1)} + \varphi_{(u)}, \qquad 0 \leq \varphi_{(1)} \in D((-\Delta)^{\frac{\alpha}{2}}_1), \qquad 0 \leq \varphi_{(u)} \in D((-\Delta)^{\frac{\alpha}{2}}_{C_{u}}).$$
(For example, 
fix $\xi \in C_c^\infty$, $0 \leq \xi \leq 1$, $\xi =1$ on $B(0,r)$, and put
 \begin{align*}
&\varphi_{(1)}:=(1+(-\Delta)^{\frac{\alpha}{2}})^{-1}\xi v, \quad v(x):=\bigl(1 + \frac{\gamma(\beta)}{\gamma(\beta-\alpha)}|x|^{-\alpha} \bigr)\tilde{\varphi}(x), \qquad \tilde{\varphi}(x):=(s^\frac{1}{\alpha}|x|^{-1})^{d-\beta}\\
&\varphi_{(u)}:=(1+(-\Delta)^{\frac{\alpha}{2}})^{-1}\big(1-\xi\big)v + \varphi-\tilde{\varphi} \quad \text{for appropriate $r>0$},
\end{align*}
and then use the result in Appendix \ref{appA__}. Let us note that if $\delta<1$, then one can differentiate $\varphi_{(1)}:=\varphi-\frac{1}{2}$ directly to obtain $\varphi_{(1)} \in \mathcal W^{2,1} \equiv D\big((-\Delta)_1\big)$. Since $D\big((-\Delta)_1\big) \subset D\big((-\Delta)^{\frac{\alpha}{2}}_1\big)$, the latter yields $\varphi_{(1)} \in D\big((-\Delta)^{\frac{\alpha}{2}}_1\big)$.)
Therefore,
$$(P^\varepsilon)^*\varphi\;\; (=(P^\varepsilon)^*_{L^1}\varphi_{(1)} + (P^\varepsilon)^*_{C_u}\varphi_{(u)})$$ is well defined and belongs to $L^1 + C_u=\{w+v\mid w\in L^1, v\in C_u\}$. 

Next, for $f=\phi_n u\in M$, we have
\begin{align*}
\langle Qf,\frac{f}{|f|}\rangle=&\langle \phi_nP^\varepsilon u,\frac{f}{|f|}\rangle=\lim_{t\downarrow 0}t^{-1}\langle\phi_n(1-e^{-tP^\varepsilon})u,\frac{f}{|f|}\rangle,\\
\Real\langle Qf,\frac{f}{|f|}\rangle &\geq\lim_{t\downarrow 0}t^{-1}\langle(1-e^{-tP^\varepsilon})|u|,\phi_n\rangle \tag{$\ast$} \label{key}\\
&=\lim_{t\downarrow 0}t^{-1}\langle (1-e^{-tP^\varepsilon}) e^{-\frac{P^\varepsilon}{n}}|u|,\varphi\rangle\\
&=\lim_{t\downarrow 0}t^{-1}\langle e^{-\frac{P^\varepsilon}{n}}|u|,(1-e^{-t(P^\varepsilon)^*})\varphi\rangle \\
&=\langle e^{-\frac{P^\varepsilon}{n}}|u|, (P^\varepsilon)^*\varphi\rangle.
\end{align*}

Now we are going to estimate $$\langle   e^{-\frac{P^\varepsilon}{n}}|u|,(P^\varepsilon)^*\varphi\rangle=:J$$ from below. For that, we estimate from below $(P^\varepsilon)^*\varphi=(-\Delta)^\frac{\alpha}{2}\varphi - \bigl(b_\varepsilon \cdot \nabla + W_\varepsilon\bigr)$.

\begin{remark}
\label{lyapunov_rem}
We will be using the fact that, by our choice of $\beta$ in \eqref{beta_eq}, 
 $|x|^{-d+\beta}$ is a Lyapunov function to the formal operator $\Lambda^*=(-\Delta)^{\frac{\alpha}{2}}-\nabla \cdot b$:
$$
\text{$(-\Delta)^\frac{\alpha}{2}|x|^{-d+\beta}=\frac{\gamma(\beta)}{\gamma(\beta-\alpha)}|x|^{-d+\beta-\alpha}$,
see Appendix \ref{appA__}; $(-\nabla \cdot b(x)|x|^{-d+\beta})=-(\beta-\alpha)\kappa|x|^{-d+\beta-\alpha}$},
$$
where, by the choice of $\beta$, $(\beta-\alpha)\kappa=\frac{\gamma(\beta)}{\gamma(\beta-\alpha)}$, and so $\Lambda^*|x|^{-d+\beta}=0$ ($=\Lambda^*\tilde{\varphi}$, $\tilde{\varphi}(x):=(s^\frac{1}{\alpha}|x|^{-1})^{d-\beta}$).
\end{remark}

We will be using the representation 
\begin{equation}
\label{repr7}
\tag{$\star\star\star$}
(-\Delta)^{\frac{\alpha}{2}}h = -\Delta I_{2-\alpha}h =-I_{2-\alpha}\Delta h, \quad h \in C_c^\infty,
\end{equation}
where $I_\nu \equiv (-\Delta)^{-\frac{\nu}{2}}$.
 
We have (in the sense of distributions):
\begin{align*}
\big\langle (-\Delta)^\frac{\alpha}{2}\varphi, h \rangle\;\;  & \big(= \big\langle (-\Delta)_1^\frac{\alpha}{2}\varphi_{(1)}, h \rangle  + \big\langle (-\Delta)_{C_u}^\frac{\alpha}{2}\varphi_{(u)}, h \rangle \big) \qquad 0 \leq h \in C_c^\infty \\
& =  \big\langle \varphi, (-\Delta)^\frac{\alpha}{2} h \rangle \qquad (\text{see, if needed, Remark \ref{rem_int} below})\\
& = \big\langle \tilde{\varphi}, (-\Delta)^\frac{\alpha}{2} h \rangle + \big\langle \varphi-\tilde{\varphi}, (-\Delta)^\frac{\alpha}{2} h \rangle \\
& (\text{we are using \eqref{repr7}}) \\
&= \big\langle I_{2-\alpha}\tilde{\varphi}, -\Delta h \rangle + \big\langle -\Delta(\varphi-\tilde{\varphi}), I_{2-\alpha} h \rangle \\ 
&= \big\langle -\Delta I_{2-\alpha}\tilde{\varphi}, h \rangle + \big\langle -I_{2-\alpha}\Delta(\varphi-\tilde{\varphi}), h \rangle.
\end{align*}
where (see Appendix \ref{appA__})
$
-\Delta I_{2-\alpha}\tilde{\varphi}=V\tilde{\varphi},
$ $V(x):=(\beta-\alpha)\kappa|x|^{-\alpha}=\frac{\gamma(\beta)}{\gamma(\beta-\alpha)}|x|^{-\alpha}$,
and, routine calculation shows, $$-I_{2-\alpha}\Delta(\varphi - \tilde{\varphi}) \;\bigl(\,\equiv -I_{2-\alpha}\mathbf 1_{B^c(0,s^\frac{1}{\alpha})}\Delta(\varphi - \tilde{\varphi})\,\bigr) \geq - c_0 s^{-1}$$ 
for a constant $c_0$.
Also, direct calculation shows, $$-\bigl(b_\varepsilon \cdot \nabla + W_\varepsilon\bigr)\varphi \geq -V\tilde{\varphi} - c_1 s^{-1}$$ for a constant $c_1 \neq c_1(\varepsilon)$.

Therefore,
\begin{equation}
\label{P}
(P^\varepsilon)^*\varphi=(-\Delta)^\frac{\alpha}{2}\varphi - \bigl(b_\varepsilon \cdot \nabla + W_\varepsilon\bigr)\varphi \geq - Cs^{-1}, \qquad C:= c_0 + c_1,
\end{equation}
and so
\[
J=\langle e^{-\frac{P^\varepsilon}{n}}|u|,(P^\varepsilon)^* \varphi\rangle \geq - Cs^{-1} \|e^{-\frac{P^\varepsilon}{n}}|u|\|_1\geq -Cs^{-1} \|e^{-\frac{P^\varepsilon}{n}}\|_{1\to 1}\|\phi_n^{-1}f\|_1,
\]
or due to $\phi_n \geq \frac{1}{2}$,
\[
J\geq -2Cs^{-1}\|e^{-\frac{P^\varepsilon}{n}}\|_{1\to 1}\|f\|_1.
\]
Noticing that $\|W_\varepsilon\|_\infty\leq c  \varepsilon^{-\frac{\alpha}{2}}$, $c:=\kappa(d-\alpha)$, we have $\|e^{-\frac{P^\varepsilon}{n}}\|_{1\to 1}\leq e^{ c \varepsilon^{-\frac{\alpha}{2}}n^{-1}}=1+o(n)$. Taking $\lambda =3Cs^{-1}$ we see that, for every $\varepsilon>0$ and for all $n$ larger than some $n(\varepsilon)$,
\[ 
\Real\langle(\lambda+ Q)f,\frac{f}{|f|}\rangle\geq 0 \quad \quad f\in M.
\] 
Now it is easily seen that the latter holds for $\tilde{Q}$ and all $f \in D(\tilde{Q})$. The proof of Proposition \ref{main_prop4} is completed.
\end{proof}

\begin{remark}
\label{rem_int}
Above we integrated by parts: since $ t^{-1}(1-e^{-t(-\Delta)_1^\frac{\alpha}{2}})\varphi_{(1)} \rightarrow (-\Delta)_1^\frac{\alpha}{2}\varphi_{(1)}$ strongly in $L^1$ as $t \downarrow 0$, we have ($h \in C_c^\infty$)
\begin{align*}
\big\langle (-\Delta)_1^\frac{\alpha}{2}\varphi_{(1)}, h \rangle &  = \lim_{t \downarrow 0}\langle t^{-1}(1-e^{-t(-\Delta)_1^\frac{\alpha}{2}})\varphi_{(1)},h\rangle \\
& = \lim_{t \downarrow 0}\langle \varphi_{(1)},t^{-1}(1-e^{-t(-\Delta)^\frac{\alpha}{2}}) h\rangle=\langle \varphi_{(1)},(-\Delta)^\frac{\alpha}{2}h\rangle.
\end{align*}
Arguing in the same way, we obtain $\big\langle (-\Delta)_{C_u}^\frac{\alpha}{2}\varphi_{(u)}, h \rangle=\big\langle \varphi_{(u)}, (-\Delta)^\frac{\alpha}{2}h \rangle$.
\end{remark}


We now end the proof of $(S_4)$.
The fact that $\tilde{Q}$ is closed together with Proposition \ref{dense_prop} and Proposition \ref{main_prop4}  imply $R(\lambda_\varepsilon +\tilde{Q})=L^1$ (Appendix \ref{appC}). But then, by the Lumer-Phillips Theorem (see e.g.\,\cite[Sect.\,IX.8]{Yos}), $\lambda+\tilde{Q}$ is the (minus) generator of a contraction semigroup, and $\tilde{Q}=G$ due to $\tilde{Q}\subset G$. Thus, it follows that
 \[
\label{star2}
 \|e^{-tG}\|_{1\to 1}\equiv\|\phi_ne^{-t P^\varepsilon}\phi_n^{-1}\|_{1\to 1}\leq e^{\hat{c} s^{-1}t}. \tag{$\star\star$}
 \]
Now, taking $n \rightarrow \infty$ in
$
\|\phi_ne^{-t P^\varepsilon}h\|_{1\to 1}\leq e^{\hat{c} s^{-1} t}\|\phi_n h\|_1$, $ h \in L^1 \cap L^\infty,
$
we obtain
\begin{equation}
\tag{$S_{4,\varepsilon}$}
 \|\varphi_s e^{-tP^\varepsilon}\varphi_s^{-1}h\|_1\leq e^{\hat{c} s^{-1} t}\|h\|_{1}, \quad h\in L^1 \cap L^\infty.
\end{equation}
According to Proposition \ref{prop4}(\textit{i}), 
$e^{-tP^{\varepsilon_i}}f \rightarrow e^{-t\Lambda_r}f$, $f \in L^r$ strongly in $L^r$ for some $\varepsilon_i \downarrow 0$, and hence a.e. Thus, appealing to Fatou's Lemma, we obtain ($S_4$).

\smallskip

The proof of ($S_4$) is completed. This completes the proof of Theorem \ref{thm2}. \hfill \qed

\medskip

In the course of the proof of Theorem \ref{thm2} we have verified the assumptions of Theorem A for $e^{-tP^\varepsilon}$, obtaining

\begin{corollary}
\label{cor2}
Let $0<\delta<4$. Then, for every $\varepsilon>0$,
$$
e^{-tP^\varepsilon}(x,y) \leq C t^{-j'}\varphi_t(y), \quad j'=\frac{d}{\alpha}, \quad C \neq C(\varepsilon),
$$
for all $t>0$, $x,y \in \mathbb R^d$, $y \neq 0$.
\end{corollary}

The following inequalities will be needed in the proof of Theorem \ref{thm2_}.


\begin{corollary}
\label{cor_rem4}
\[
\langle e^{-t(P^\varepsilon)^*}(x,\cdot)\varphi_t(\cdot) \rangle \leq c_1 \varphi_t(x),
\]
\[
\langle e^{-t(P^\varepsilon)^*}(x,\cdot)\rangle \leq 2c_1\varphi_t(x), \qquad x \neq 0, \quad t>0, \quad c_1 \neq c_1(\varepsilon).
\]
\end{corollary}
\begin{proof}
The second inequality follows from the first inequality and $\varphi_t \geq \frac{1}{2}$.

Let us prove the first inequality. From ($S_{4,\varepsilon}$): $\|\varphi_t e^{-tP^\varepsilon}\varphi_t^{-1}\|_{1 \rightarrow 1} \leq c_1$, $c_1 \neq c_1(\varepsilon)$, it follows
$$
\big\langle \varphi_{t} \wedge n, e^{-tP^\varepsilon} h \rangle \leq \langle \varphi_t e^{-tP^\varepsilon} h \big\rangle \leq c_1 \langle \varphi_t h \rangle, \quad 0 \leq h \in C_c^\infty, \quad n=1,2,\dots
$$
so
$$
\big\langle e^{-t(P^\varepsilon)^*}(\varphi_t \wedge n), h \rangle  \leq c_1 \langle \varphi_t h \rangle.
$$
Thus, upon taking $h \rightarrow \delta_x$, $x \neq 0$
$$
\big\langle e^{-t(P^\varepsilon)^*}(x,\cdot)(\varphi_t(\cdot) \wedge n) \big\rangle = e^{-t(P^\varepsilon)^*}(\varphi_t \wedge n)(x) \leq c_1 \varphi_t(x). 
$$
Taking $n \rightarrow \infty$ using the Beppo Levi Lemma, we obtain the required.
\end{proof}


\bigskip

\section{Proof of Theorem \ref{thm2_}: The upper bound $e^{-t\Lambda}(x,y)\leq C e^{-t(-\Delta)^\frac{\alpha}{2}}(x,y)\varphi_t(y)$} 

\label{ub_sect}

Set $A \equiv (-\Delta)^\frac{\alpha}{2}$.
The upper bound will follow from a priori upper bound
\begin{equation}
\label{apr_ub}
e^{-tP^\varepsilon}(x,y)\leq Ce^{-tA}(x,y)\varphi_t(y), \quad C\neq C(\varepsilon)
\end{equation}
for all $t>0$, $x,y \in \mathbb R^d$, $y \neq 0$.

Indeed, taking \eqref{apr_ub} for granted, and using Proposition \ref{prop4}(\textit{i}): $e^{-tP^{\varepsilon_i}}f \rightarrow e^{-t\Lambda_r}f$ strongly in $L^r$ for some $\varepsilon_i \downarrow 0$ for every $f \in L^r$, $r \in ]r_c,\infty[$, we obtain
$$
\langle \mathbf{1}_{S_1},e^{-t\Lambda_r}\mathbf{1}_{S_2} \rangle \leq C\langle \mathbf{1}_{S_1},e^{-tA}\varphi_t\mathbf{1}_{S_2} \rangle
$$
for all bounded measurable $S_1$, $S_2 \subset \mathbb R^d$. Since, by Theorem \ref{thm2}, $e^{-t\Lambda_r}$ is an integral operator for every $t>0$ with integral kernel $e^{-t\Lambda(x,y)}$, we obtain by the Lebesgue Differentiation Theorem, possibly after changing $e^{-t\Lambda}(x,y)$ on a measure zero set in $\mathbb R^d \times \mathbb R^d$, that
$$
e^{-t\Lambda}(x,y) \leq Ce^{-tA}(x,y)\varphi_t(y)
$$
for all $t>0$, $x,y \in \mathbb R^d$, $y \neq 0$, and so Theorem \ref{thm2_} follows. 

\medskip

To establish \eqref{apr_ub}, by duality it suffices to prove 
\begin{equation*}
e^{-t(P^\varepsilon)^*}(x,y)\leq Ce^{-tA}(x,y)\varphi_t(x).
\end{equation*}

Let $R>1$ to be chosen later. 

\medskip

\textit{\textbf{Case 1:} $|x|, |y| \leq 2Rt^{\frac{1}{\alpha}}$.}

 Since 
\begin{equation*}
k_0^{-1} t\bigl(|x-y|^{-d-\alpha} \wedge t^{-\frac{d+\alpha}{\alpha}}\bigr) \leq e^{-tA}(x,y) \leq k_0 t\bigl(|x-y|^{-d-\alpha} \wedge t^{-\frac{d+\alpha}{\alpha}}\bigr)
\end{equation*}
for all $x,y \in \mathbb R^d$, $x \neq y$, $t>0$, for a constant $k_0=k_0(d,\alpha)>1$,
the Nash initial estimate
$e^{-t(P^\varepsilon)^*}(x,y) \leq C t^{-\frac{d}{\alpha}}\varphi_t(x)$ (see Corollary \ref{cor2}) yields
\[
e^{-t(P^\varepsilon)^*}(x,y)\leq C_R e^{-tA}(x,y)\varphi_t(x), \quad C_R\neq C_R(\varepsilon).
\]

\bigskip

To consider the other cases
we use the Duhamel formula
\begin{align*}
e^{-t(P^\varepsilon)^*} &=e^{-tA} + \int_0^t e^{-\tau(P^\varepsilon)^*}(B^t_{\varepsilon,R} + B_{\varepsilon,R}^{t,c})e^{-(t-\tau)A}d\tau \\
& =: e^{-A} + K^t_R + K_R^{t,c},
\end{align*}
where $$B^t_{\varepsilon,R}:=\mathbf{1}_{B(0,Rt^{\frac{1}{\alpha}})}B_\varepsilon, \quad B^{t,c}_{\varepsilon,R}:=\mathbf{1}_{B^c(0,Rt^{\frac{1}{\alpha}})}B_\varepsilon$$ and $$B_\varepsilon:=b_\varepsilon \cdot \nabla + W_\varepsilon$$ (recall, 
$W_\varepsilon(x)=\kappa(d-\alpha)|x|_\varepsilon^{-\alpha}$, $b_\varepsilon(x)=\kappa|x|_\varepsilon^{-\alpha}x$).

Below we prove that $K^t_R(x,y)$, $K_R^{t,c}(x,y) \leq C'_R  e^{-tA}(x,y)\varphi_t(x)$, which would yield the upper bound. We will need the following.

\begin{lemma}
\label{claim_lem}

Set
$
E^t(x,y)=t\bigl( |x-y|^{-d-\alpha-1} \wedge t^{-\frac{d+\alpha+1}{\alpha}}\bigr),
$
$E^tf(x):=\langle E^t(x,\cdot)f(\cdot)\rangle$.

Then there exist constants $k_i$ {\rm($i=1,2,3$)} such that for all $0<t<\infty$, $x$, $y \in \mathbb R^d$

{\rm (\textit{i})} $|\nabla_x e^{-tA}(x,y)| \leq k_1 E^t(x,y)$;

{\rm (\textit{ii})} $\int_0^t \langle e^{-(t-\tau)A}(x,\cdot)E^\tau(\cdot,y)\rangle d\tau \leq k_2 t^\frac{\alpha-1}{\alpha} e^{-tA}(x,y)$;

{\rm (\textit{iii})} $\int_0^t \langle E^{t-\tau}(x,\cdot) E^\tau(\cdot,y)\rangle d\tau \leq k_3 t^\frac{\alpha-1}{\alpha} E^t(x,y).$

\end{lemma}

\begin{proof} For the proof of (\textit{i}) see e.g.\,\cite[Lemma 5]{BJ}, for the proof of (\textit{ii}) see \cite[proof of Lemma 13]{BJ}. Essentially the same arguments yield (\textit{iii}). For the sake of completeness, we provide the details: 
\begin{align*}
E^t(x,z) \wedge E^\tau(z,y)  & = (t|x-z|^{-d-\alpha-1} \wedge t^{-\frac{d+1}{\alpha}}) \wedge (\tau|z-y|^{-d-\alpha-1} \wedge \tau^{-\frac{d+1}{\alpha}}) \\
& \leq C_0\left(\frac{t+\tau}{2}\right)^{-\frac{d+1}{\alpha}} \wedge \biggl[(t+\tau) \left(\frac{|x-z|+|z-y|}{2}\right)^{-d-\alpha-1}\biggr]\qquad  (\text{$C_0>1$}) \\
& \leq C(t+\tau)^{-\frac{d+1}{\alpha}} \wedge \bigl[(t+\tau) (|x-y|)^{-d-\alpha-1}\bigr]= C E^{t+\tau}(x,y).
\end{align*}
Thus, using inequality $ac = (a \wedge c) (a \vee c) \leq (a \wedge c)(a+c)$ ($a,c \geq 0$), we obtain
$$
\int_0^t \langle E^{t-\tau}(x,\cdot) E^\tau(\cdot,y)\rangle d\tau \leq  C E^{t}(x,y) \int_0^t \langle E^{t-\tau}(x,\cdot) + E^{\tau}(\cdot,y) \rangle d\tau,
$$
where, routine calculation shows, $\int_0^t \langle E^{t-\tau}(x,\cdot) + E^{\tau}(\cdot,y) \rangle d\tau \leq C_1 t^\frac{\alpha-1}{\alpha}$. Put $k_3:=CC_1$.
\end{proof}

\textit{\textbf{Case 2:} $|y| > 2Rt^{\frac{1}{\alpha}}$, $0<|x| \leq |y|$.}

\begin{claim}
\label{claim_n1}
If $|y| > 2Rt^{\frac{1}{\alpha}}$, $0<|x| \leq |y|$, then
$$
K^t_R(x,y) \equiv \int_0^t \big\langle e^{-\tau(P^\varepsilon)^*}(x,\cdot)B_{\varepsilon,R}(\cdot)e^{-(t-\tau)A}(\cdot,y)\big\rangle d\tau
 \leq \hat{C} e^{-tA}(x,y)\varphi_t(x), \quad \hat{C}\neq \hat{C}(\varepsilon).
$$
\end{claim}
\begin{proof}
Claim \ref{claim_n1} clearly follows from

{\rm(\textit{j})} $
\int_0^t  \langle e^{-\tau(P^\varepsilon)^*}(x,\cdot)\mathbf{1}_{B(0,Rt^{\frac{1}{\alpha}})}(\cdot)W_\varepsilon(\cdot) e^{-(t-\tau)A}(\cdot,y)\rangle d\tau \leq c_4 e^{-tA}(x,y) \varphi_t(x),
$

\noindent and, in view of Lemma \ref{claim_lem}(\textit{i}), from

{\rm(\textit{jj})} $\int_0^t \langle e^{-\tau(P^\varepsilon)^*}(x,\cdot)\mathbf{1}_{B(0,Rt^{\frac{1}{\alpha}})}(\cdot)Z_\varepsilon(\cdot) E^{t-\tau}(\cdot,y)\rangle d\tau \leq c_3 e^{-tA}(x,y) \varphi_t(x)$, where $Z_\varepsilon(x):=|x|_\varepsilon^{-\alpha}|x|$.

Let us prove (\textit{jj}):
\begin{align*}
&\int_0^t \langle e^{-\tau(P^\varepsilon)^*}(x,\cdot)\mathbf{1}_{B(0,Rt^{\frac{1}{\alpha}})}(\cdot)Z_\varepsilon(\cdot)E^{t-\tau}(\cdot,y)\rangle d\tau \\
& (\text{we are using $E^{t-\tau}(\cdot,y) \leq Ce^{-(t-\tau)A}(\cdot,y)|\cdot-y|^{-1}$}) \\
&\leq C \int_0^t   \langle e^{-\tau(P^\varepsilon)^*}(x,\cdot)\mathbf{1}_{B(0,Rt^{\frac{1}{\alpha}})}(\cdot)Z_\varepsilon(\cdot)e^{-(t-\tau)A}(\cdot,y)|\cdot-y|^{-1}\rangle d\tau \\
& (\text{we are using $\mathbf{1}_{B(0,Rt^{\frac{1}{\alpha}})}(\cdot)|\cdot-y|^{-1} \leq |\cdot|^{-1}$}) \\
& \leq C'  \int_0^t   \langle e^{-\tau(P^\varepsilon)^*}(x,\cdot)\mathbf{1}_{B(0,Rt^{\frac{1}{\alpha}})}(\cdot)W_\varepsilon(\cdot)e^{-(t-\tau)A}(\cdot,y)\rangle d\tau \\
& (\text{we are using $\mathbf{1}_{B(0,Rt^{\frac{1}{\alpha}})}(\cdot)e^{-(t-\tau)A}(\cdot,y) \leq \tilde{c}e^{-tA}(x,y)$}) \\
& \leq C''e^{-tA}(x,y) \int_0^t   \langle e^{-\tau(P^\varepsilon)^*}(x,\cdot)\mathbf{1}_{B(0,Rt^{\frac{1}{\alpha}})}(\cdot)W_\varepsilon(\cdot)\rangle d\tau.
\end{align*}
According to the Duhamel formula $e^{-t(P^\varepsilon)^*}=e^{-tA} + \int_0^t e^{-\tau(P^\varepsilon)^*}(b_\varepsilon \cdot \nabla + W_\varepsilon)e^{-(t-\tau)A}d\tau$,
$$1+\int_0^t \langle e^{-\tau(P^\varepsilon)^*}(x,\cdot) W_\varepsilon(\cdot)\rangle d\tau=\langle e^{-t(P^\varepsilon)^*}(x,\cdot)\rangle.$$
Using the inequality $\langle e^{-t(P^\varepsilon)^*}(x,\cdot)\rangle\leq 2c_1 \varphi_t(x)$ from Corollary \ref{cor_rem4}, it is seen that 
\begin{equation*}
\int_0^t \langle e^{-\tau(P^\varepsilon)^*}(x,\cdot) W_\varepsilon(\cdot)\rangle d\tau  \leq 2c_1 \varphi_t(x).
\end{equation*}
The latter and the previous estimate yield (\textit{jj}).
Incidentally, we also have proved (\textit{j}).
\end{proof}

\begin{claim}
\label{claim_n2}
If $|y| > 2Rt^{\frac{1}{\alpha}}$, $0<|x| \leq |y|$, then
$$
K_R^{t,c}(x,y) \equiv \int_0^t \langle e^{-\tau(P^\varepsilon)^*}(x,\cdot)B^c_{\varepsilon,R}(\cdot)e^{-(t-\tau)A}(\cdot,y)\rangle d\tau
 \leq C e^{-tA}(x,y)\varphi_t(x), \qquad C \neq C(\varepsilon).$$
\end{claim}
\begin{proof}
Lemma \ref{claim_lem}(\textit{i}) yields
\begin{equation}
\label{Best}
\tag{$\ast$}
|B_{\varepsilon,R}^{t,c}(\cdot) e^{-(t-\tau)A}(\cdot,y)| \leq C_0\big(R^{-\alpha}t^{-1} e^{-(t-\tau)A}(\cdot,y) + R^{-\alpha+1}t^{-\frac{\alpha-1}{\alpha}} E^{t-\tau}(\cdot,y)\big),
\end{equation}
so
\begin{align}
K_R^{t,c}(x,y) & \leq C_0 R^{-\alpha}t^{-1} \int_0^t \big\langle e^{-\tau(P^\varepsilon)^*}(x,\cdot) e^{-(t-\tau)A}(\cdot,y)\big\rangle d\tau \notag \\
&+ C_0 R^{-\alpha+1}t^{-\frac{\alpha-1}{\alpha}}\int_0^t \big\langle e^{-\tau(P^\varepsilon)^*}(x,\cdot) E^{t-\tau}(\cdot,y)\big\rangle d\tau. \label{n2} \tag{$\ast\ast$}
  \end{align}

$\mathbf 1$.~Let us estimate the first term in the RHS of \eqref{n2}. By the Duhamel formula,
\begin{align*}
&\int_0^t e^{-\tau(P^\varepsilon)^*}e^{-(t-\tau)A}d\tau \\
& = \int_0^t e^{-\tau A}e^{-(t-\tau)A}d\tau + \int_0^t \int_0^\tau e^{-\tau'(P^\varepsilon)^*}(B^t_{\varepsilon,R} + B^{t,c}_{\varepsilon,R})e^{-(\tau-\tau')A}d\tau' e^{-(t-\tau)A} d\tau \\
&\equiv t e^{-tA} + I_R + I_R^{c}.
\end{align*}

Applying Lemma \ref{claim_lem}(\textit{i}), we obtain
\begin{align*}
I_R(x,y) & \leq k_1 \int_0^t \int_0^\tau \big(e^{-\tau'(P^\varepsilon)^*}\mathbf{1}_{B(0,Rt^{\frac{1}{\alpha}})}|b_\varepsilon|E^{\tau-\tau'}d\tau' e^{-(t-\tau)A}\big)(x,y) d\tau \\
&+ \int_0^t \int_0^\tau \big(e^{-\tau'(P^\varepsilon)^*}\mathbf{1}_{B(0,Rt^{\frac{1}{\alpha}})}W_\varepsilon e^{-(\tau-\tau')A}d\tau' e^{-(t-\tau)A}\big)(x,y) d\tau \\
& (\text{we are changing the order of integration and applying Lemma \ref{claim_lem}}(\textit{ii})) \\
& \leq k_1 k_2\int_0^t \big( e^{-\tau'(P^\varepsilon)^*}\mathbf{1}_{B(0,Rt^{\frac{1}{\alpha}})}|b_\varepsilon|(t-\tau')^{\frac{\alpha-1}{\alpha}}e^{-(t-\tau')A}\big)(x,y)d\tau' \\
&+ \int_0^t \big(e^{-\tau'(P^\varepsilon)^*}\mathbf{1}_{B(0,Rt^{\frac{1}{\alpha}})}W_\varepsilon (t-\tau') e^{-(t-\tau')A}\big)(x,y)d\tau' \\
& \leq k_1 k_2t^{\frac{\alpha-1}{\alpha}}\int_0^t \big( e^{-\tau'(P^\varepsilon)^*}\mathbf{1}_{B(0,Rt^{\frac{1}{\alpha}})}|b_\varepsilon|e^{-(t-\tau')A}\big)(x,y)d\tau' \\
& + t\int_0^t \big(e^{-\tau'(P^\varepsilon)^*}\mathbf{1}_{B(0,Rt^{\frac{1}{\alpha}})}W_\varepsilon e^{-(t-\tau')A}\big)(x,y)d\tau' \\
& (\text{we are using obvious estimate $t^{-\frac{1}{\alpha}}\mathbf{1}_{B(0,Rt^{\frac{1}{\alpha}})}|b_\varepsilon| \leq RC\mathbf{1}_{B(0,Rt^{\frac{1}{\alpha}})}W_\varepsilon$} \\
& \text{and then applying to both terms estimate (\textit{j}) in the proof of Claim \ref{claim_n1}}) \\
& \leq \hat{C}t e^{-tA}(x,y)\varphi_t(x).
\end{align*}

In turn, $I_R^{c}=\int_0^t (I_R^c)^\tau e^{-(t-\tau)A}d\tau$, where
$
(I_R^c)^\tau :=  \int_0^\tau e^{-\tau'(P^\varepsilon)^*} B_{\varepsilon,R}^{t,c} e^{-(\tau-\tau')A}d\tau',
$
so
\begin{align*}
|(I_R^c)^\tau(x,y)| &\leq C_0 R^{-\alpha}t^{-1}\int_0^\tau \big\langle e^{-\tau'(P^\varepsilon)^*}(x,\cdot) e^{-(\tau-\tau')A}(\cdot,y)\big\rangle d\tau'\\
& + C_0 R^{-\alpha+1}t^{-\frac{\alpha-1}{\alpha}}\int_0^\tau \big\langle e^{-\tau'(P^\varepsilon)^*}(x,\cdot) E^{\tau-\tau'}(\cdot,y)\big\rangle d\tau'.
  \end{align*}
Then
\begin{align*}
 |I_R^c(x,y)| &\leq C_0 R^{-\alpha}t^{-1}
\int_0^t \int_{0}^\tau \big( e^{-\tau'(P^\varepsilon)^*} e^{-(\tau-\tau')A} e^{-(t-\tau)A}\big)(x,y) d\tau' d\tau
 \\
&+ C_0 R^{-\alpha+1}t^{-\frac{\alpha-1}{\alpha}} \int_0^t \int_{0}^\tau \big( e^{-\tau'(P^\varepsilon)^*} E^{\tau-\tau'} e^{-(t-\tau)A}\big)(x,y) d\tau' d\tau,
\end{align*}
where we estimate the first and second integrals as follows.
\begin{align*}
&\int_0^t \int_{0}^\tau \big( e^{-\tau'(P^\varepsilon)^*} e^{-(t-\tau')A}\big)(x,y) d\tau' d\tau \\
&\leq \int_0^t \int_{0}^t \big( e^{-\tau'(P^\varepsilon)^*} e^{-(t-\tau')A}\big)(x,y) d\tau' d\tau = t \int_0^t \big\langle e^{-\tau'(P^\varepsilon)^*}(x,\cdot)e^{-(t-\tau')A}(\cdot,y)\big\rangle d\tau',
\end{align*}
\begin{align*}
&\int_0^t \int_{0}^\tau \big( e^{-\tau'(P^\varepsilon)^*} E^{\tau-\tau'} e^{-(t-\tau)A}\big)(x,y) d\tau' d\tau \\
& (\text{we are changing the order of integration in $\tau$ and $\tau'$}) \\
&=\int_0^t \int_{\tau'}^t \big( e^{-\tau'(P^\varepsilon)^*} E^{\tau-\tau'} e^{-(t-\tau)A}\big)(x,y) d\tau d\tau' \\
& (\text{by Lemma \ref{claim_lem}(\textit{ii}), $\int_{\tau'}^t (E^{\tau-\tau'} e^{-(t-\tau)A})(\cdot,y) d \tau\leq k_2(t-\tau')^{\frac{\alpha-1}{\alpha}}e^{-(t-\tau')A}(\cdot,y)$}) \\
& \leq k_2t^{\frac{\alpha-1}{\alpha}}\int_0^t \big\langle e^{-\tau'(P^\varepsilon)^*}(x,\cdot)e^{-(t-\tau')A}(\cdot,y)\big\rangle d\tau'.
\end{align*}
Thus,
$$
 |I_R^c(x,y)| \leq C_0 (R^{-\alpha}+k_2 R^{-\alpha+1}) \int_0^t \big\langle e^{-\tau(P^\varepsilon)^*}(x,\cdot)e^{-(t-\tau)A}(\cdot,y)\big\rangle d\tau.
$$

Therefore, for $R>1$ such that $C_0 (R^{-\alpha}+k_2 R^{-\alpha+1})\leq \frac{1}{2}$,
\begin{align*}
&\int_0^t \langle e^{-\tau(P^\varepsilon)^*}(x,\cdot)e^{-(t-\tau)A}(\cdot,y)\rangle d\tau \notag \\
&\leq t e^{-tA}(x,y) + \frac{1}{2}\int_0^t \langle e^{-\tau(P^\varepsilon)^*}(x,\cdot)e^{-(t-\tau)A}(\cdot,y)\rangle d\tau
+ \hat{C} t e^{-tA}(x,y)\varphi_t(x), 
\end{align*}
i.e.
$
\int_0^t \langle e^{-\tau(P^\varepsilon)^*}(x,\cdot)e^{-(t-\tau)A}(\cdot,y)\rangle d\tau \leq 2(2+\hat{C}) t e^{-tA}(x,y) \varphi_t(x)$.	

\smallskip

$\mathbf 2$.~Let us estimate the second term in the RHS of \eqref{n2}. By the Duhamel formula
\begin{align*}
&\int_0^t e^{-\tau(P^\varepsilon)^*}E^{t-\tau}d\tau \\
& = \int_0^t e^{-\tau A}E^{t-\tau}d\tau + \int_0^t \int_0^\tau e^{-\tau'(P^\varepsilon)^*}(B^t_{\varepsilon,R} + B^{t,c}_{\varepsilon,R})e^{-(\tau-\tau')A}d\tau' E^{t-\tau} d\tau  \\
&  \equiv \int_0^t e^{-\tau A}E^{1-\tau}d\tau + J_R + J_R^c,
\end{align*}
where, by Lemma \ref{claim_lem}(\textit{ii}),
$
\int_0^t \langle e^{-\tau A}(x,\cdot)E^{t-\tau}(\cdot,y)\rangle d\tau \leq k_2 t^{\frac{\alpha-1}{\alpha}}e^{-tA}(x,y).
$
Let us estimate $J_R(x,y)$ and $J_R^c(x,y)$.

Using Lemma \ref{claim_lem}(\textit{i}), we obtain
\begin{align*}
J_R(x,y) & \leq k_1 \int_0^t \big(\int_0^\tau e^{-\tau'(P^\varepsilon)^*}\mathbf{1}_{B(0,Rt^{\frac{1}{\alpha}})}|b_\varepsilon|E^{\tau-\tau'}d\tau' E^{t-\tau} \big)(x,y)d\tau \\
&+ \int_0^t \int_0^\tau \big(e^{-\tau'(P^\varepsilon)^*}\mathbf{1}_{B(0,Rt^{\frac{1}{\alpha}})}W_\varepsilon e^{-(\tau-\tau')A}d\tau' E^{t-\tau}\big)(x,y) d\tau \\
& (\text{we are changing the order of integration and applying Lemma \ref{claim_lem}}(\textit{ii}),(\textit{iii})) \\
& \leq k_1 k_3\int_0^t \big( e^{-\tau'(P^\varepsilon)^*}\mathbf{1}_{B(0,Rt^{\frac{1}{\alpha}})}|b_\varepsilon|(t-\tau')^{\frac{\alpha-1}{\alpha}}E^{t-\tau'}\big)(x,y)d\tau' \\
&+ k_2\int_0^t \big(e^{-\tau'(P^\varepsilon)^*}\mathbf{1}_{B(0,R)}W_\varepsilon (t-\tau')^{\frac{\alpha-1}{\alpha}} e^{-(t-\tau')A}\big)(x,y)d\tau' \\
& \leq k_1 k_3 t^{\frac{\alpha-1}{\alpha}} \int_0^t \big( e^{-\tau'(P^\varepsilon)^*}\mathbf{1}_{B(0,Rt^{\frac{1}{\alpha}})}|b_\varepsilon|E^{t-\tau'}\big)(x,y)d\tau' \\
& + k_2t^{\frac{\alpha-1}{\alpha}}\int_0^t \big(e^{-\tau'(P^\varepsilon)^*}\mathbf{1}_{B(0,Rt^{\frac{1}{\alpha}})}W_\varepsilon  e^{-(t-\tau')A}\big)(x,y)d\tau' \\
& (\text{we are applying estimates (\textit{j}), (\textit{jj}) in the proof of Lemma \ref{claim_n1}}) \\
& \leq C_1 t^{\frac{\alpha-1}{\alpha}} e^{-tA}(x,y)\varphi_t(x).
\end{align*}

In turn, $J_R^c=\int_0^t (J_R^c)^\tau E^{t-\tau}d\tau$, where
$
(J_R^c)^\tau :=  \int_0^\tau e^{-\tau'(P^\varepsilon)^*} B_{\varepsilon,R}^{t,c} e^{-(\tau-\tau')A}d\tau'.
$
By \eqref{Best} and Lemma \ref{claim_lem}(\textit{ii}), 
\begin{align*}
|(J_R^c)^\tau(x,y)| &\leq C_0 R^{-\alpha}t^{-1}\int_0^\tau \big(e^{-\tau'(P^\varepsilon)^*} e^{-(\tau-\tau')A}\big)(x,y)d\tau' \\
&+ C_0t^{-\frac{\alpha-1}{\alpha}} R^{-\alpha+1}\int_0^\tau \big(e^{-\tau'(P^\varepsilon)^*} E^{\tau-\tau'}\big)(x,y)d\tau'.
\end{align*}
Thus, changing the order of integration and applying Lemma \ref{claim_lem}(\textit{ii}),(\textit{iii}) as above, we obtain
\begin{align*}
 |J_R^c(x,y)| & \leq C_0 k_2 R^{-\alpha}t^{-\frac{1}{\alpha}} \int_0^t \langle e^{-\tau'(P^\varepsilon)^*}(x,\cdot)e^{-(t-\tau')A}(\cdot,y)\rangle d\tau' \\
&+ C_0 k_3 R^{-\alpha + 1} \int_0^t \langle e^{-\tau'(P^\varepsilon)^*}(x,\cdot)E^{t-\tau'}(\cdot,y)\rangle d\tau'.
\end{align*}

Thus, for $R>1$ such that $C_0k_2R^{-\alpha}, C_0 k_3 R^{-\alpha+1} \leq \frac{1}{2}$,
\begin{align*}
&\int_0^t \langle e^{-\tau(P^\varepsilon)^*}(x,\cdot)E^{t-\tau}(\cdot,y)\rangle d\tau \notag \leq  k_2 t^{\frac{\alpha-1}{\alpha}}e^{-tA}(x,y) + \frac{1}{2}t^{-\frac{1}{\alpha}}\int_0^t \langle e^{-\tau(P^\varepsilon)^*}(x,\cdot) e^{-(t-\tau)A}(\cdot,y)\rangle d\tau \\
&+ \frac{1}{2}\int_0^t \langle e^{-\tau(P^\varepsilon)^*}(x,\cdot)E^{t-\tau}(\cdot,y)\rangle d\tau + C_1t^{\frac{\alpha-1}{\alpha}} e^{-tA}(x,y)\varphi_t(x). 
\end{align*}
Applying $\mathbf 1$, we obtain $\int_0^t \langle e^{-\tau(P^\varepsilon)^*}(x,\cdot)E^{t-\tau}(\cdot,y)\rangle d\tau \notag \leq 2(2k_2+2+\hat{C}+C_1)t^{\frac{\alpha-1}{\alpha}}  e^{-tA}(x,y)\varphi_t(x)$.

Now $\mathbf 1$ and $\mathbf 2$ applied in \eqref{n2} yield Claim \ref{claim_n2}.	
\end{proof}

\bigskip

\textit{\textbf{Case 3:} $|x| > 2Rt^{\frac{1}{\alpha}}$, $|y| \leq |x|$. }
Using the Duhamel formula (where we discard the term containing $-U_\varepsilon$) and applying Lemma \ref{claim_lem}(\textit{i}), we have

\begin{align}
&e^{-t(P^\varepsilon)^*}(x,y)  \leq e^{-tA}(x,y) + k_1\int_0^t \big(E^\tau |b_\varepsilon| e^{-(t-\tau)(P^\varepsilon)^*}\big)(x,y)d\tau \notag \\
& \leq e^{-tA}(x,y) + k_1\int_0^t \big(E^\tau \mathbf{1}_{B(0,Rt^{\frac{1}{\alpha}})}|b_\varepsilon| e^{-(t-\tau)(P^\varepsilon)^*} + E^\tau \mathbf{1}_{B^c(0,Rt^{\frac{1}{\alpha}})}|b_\varepsilon| e^{-(1-\tau)(P^\varepsilon)^*}\big)(x,y)d\tau \notag \\
& =: e^{-tA}(x,y)+k_1 L_{R}(x,y)+ k_1 L^{c}_{R}(x,y). \label{e_est}
\end{align}

We first estimate $L_{R}(x,y)$. 
Recalling that $E^\tau(x,z) =\tau\bigl( |x-z|^{-d-\alpha-1} \wedge \tau^{-\frac{d+\alpha+1}{\alpha}}\bigr)$ and taking into account that $|x| \geq 2Rt^{\frac{1}{\alpha}}$, $|z| \leq Rt^{\frac{1}{\alpha}}$, we obtain, for all $0 \leq \tau \leq t$,
$
E^\tau(x,z) \leq t|x-z|^{-d-\alpha-1} \leq t^{\frac{\alpha-1}{\alpha}}|x-z|^{-d-\alpha} R^{-1}.
$
Therefore,
\begin{align*}
L_{R}(x,y) & \leq R^{-1}t^{\frac{\alpha-1}{\alpha}}  \int_0^t \langle |x-\cdot|^{-\alpha-d} \mathbf{1}_{B(0,Rt^{\frac{1}{\alpha}})}(\cdot)|b_\varepsilon(\cdot)| e^{-(t-\tau)(P^\varepsilon)^*}(\cdot,y)\rangle d\tau  \\
& (\text{we are using that $|x| >2Rt^{\frac{1}{\alpha}}$,  $|\cdot| \leq Rt^{\frac{1}{\alpha}}$}) \\
& \leq C_Rt^{\frac{\alpha-1}{\alpha}}|x|^{-\alpha-d} \int_0^t \langle \mathbf{1}_{B(0,Rt^{\frac{1}{\alpha}})}(\cdot)|b_\varepsilon(\cdot)| e^{-(t-\tau)(P^\varepsilon)^*}(\cdot,y)\rangle d\tau\\
& (\text{we are using that $|y| \leq |x|$ and re-denoting $t-\tau$ by $\tau$}) \\
& \leq C'_R t^{-\frac{1}{\alpha}}e^{-tA}(x,y) \int_0^t \|e^{-\tau P^\varepsilon}\mathbf{1}_{B(0,Rt^{\frac{1}{\alpha}})}|b|\|_{\infty} d\tau  \\
& \leq C'_R  t^{-\frac{1}{\alpha}} e^{-tA}(x,y) \int_0^t\bigl(  \|e^{-\tau  P^\varepsilon}\mathbf{1}_{B(0,Rt^{\frac{1}{\alpha}})}\mathbf{1}_{B^c(0,\tau^{\frac{1}{\alpha}})}|b|\|_{\infty}  + \|e^{-\tau P^\varepsilon}\mathbf{1}_{B(0,\tau^{\frac{1}{\alpha}})}|b|\|_{\infty}  \bigr) d\tau  \\
& (\text{in the first term, we are using that $e^{-\tau P^\varepsilon}$ is $L^\infty$ contraction and  $\mathbf{1}_{B^c(0,\tau^{\frac{1}{\alpha}})}|b| \leq \kappa \tau^{-\frac{\alpha-1}{\alpha}}$,} \\
& \text{in the second term we are applying Corollary \ref{cor2})} \\
& \leq C'_R t^{-\frac{1}{\alpha}} e^{-tA}(x,y)\biggl( \kappa \int_0^t \tau^{-\frac{\alpha-1}{\alpha}}d\tau  +  C  \int_0^t \tau^{-\frac{d}{\alpha}} \|\mathbf{1}_{B(0,\tau^{\frac{1}{\alpha}})}|b|\|_{1,\sqrt{\varphi_{\tau}}} d\tau\biggr)
\end{align*}
(where, recall, $\|h\|_{1,\sqrt{\varphi_\tau}}:=\langle |h|\varphi_\tau\rangle$).
In turn, by the definition of $\varphi_\tau$,
\begin{align*}
\|\mathbf{1}_{B(0,\tau^{\frac{1}{\alpha}})}|b|\|_{1,\sqrt{\varphi_{\tau }}} & = \langle \mathbf{1}_{B(0,\tau^{\frac{1}{\alpha}})}(\cdot)\big(\tau^{-\frac{1}{\alpha}}|\cdot|\big)^{-d+\beta}|\cdot|^{-\alpha+1 }\rangle \\
& \leq  \tau^{\frac{d-\beta}{\alpha}} \langle \mathbf{1}_{B(0,\tau^{\frac{1}{\alpha}})}(\cdot)|\cdot|^{-d+\beta-\alpha+1 }\rangle \\
 & = \tau^{\frac{d-\beta}{\alpha}} \tau^{\frac{d}{\alpha}}\tau^{-\frac{d-\beta+\alpha-1}{\alpha}} = \tau^{\frac{d-\alpha+1}{\alpha}}.
\end{align*}
It follows that
\begin{equation}
\label{l_est}
L_{R}(x,y) \leq C' e^{-tA}(x,y).
\end{equation}

Next,
$$
L^{c}_{R}(x,y) \leq \kappa R^{1-\alpha}t^{-\frac{\alpha-1}{\alpha}}\int_0^t E^\tau  e^{-(t-\tau)(P^\varepsilon)^*}d\tau.
$$
Let us estimate the integral in the RHS. Using the Duhamel formula  (where we discard the term containing $-U_\varepsilon $), we obtain
\begin{align*}
\int_0^t\big( E^\tau  e^{-(t-\tau)(P^\varepsilon)^*}\big)(x,y)d\tau & \leq  \int_0^t \big(E^\tau  e^{-(t-\tau)A}\big)(x,y)d\tau + \int_0^t \big(E^\tau \int_0^{t-\tau} E^{t-\tau-s}|b_\varepsilon|e^{-s(P^\varepsilon)^*}ds\big)(x,y) d\tau \\
& (\text{we are applying Lemma \ref{claim_lem}(\textit{ii}) and changing the order of integration}) \\
& \leq k_2 t^{\frac{\alpha-1}{\alpha}}e^{-tA}(x,y) + \int_0^t \int_0^{t-s} \big(E^\tau E^{t-s-\tau}|b_\varepsilon|e^{-s(P^\varepsilon)^*}\big)(x,y)d\tau ds \\
& (\text{we are applying Lemma \ref{claim_lem}(\textit{iii})}) \\
& \leq k_2 t^{\frac{\alpha-1}{\alpha}} e^{-tA}(x,y) + k_3\int_0^t (t-s)^\frac{\alpha-1}{\alpha} \big(E^{t-s}|b_\varepsilon|e^{-s(P^\varepsilon)^*}\big)(x,y) ds \\
& \leq k_2  t^{\frac{\alpha-1}{\alpha}} e^{-tA}(x,y) + k_3 t^{\frac{\alpha-1}{\alpha}} \int_0^t \big( E^{t-s}\mathbf{1}_{B(0,Rt^{\frac{1}{\alpha}})}|b_\varepsilon|e^{-s(P^\varepsilon)^*}\big)(x,y)d\tau ds \\
& + k_3 t^{\frac{\alpha-1}{\alpha}} \int_0^t \big( E^{t-s}\mathbf{1}_{B^c(0,Rt^{\frac{1}{\alpha}})}|b|e^{-s(P^\varepsilon)^*}\big)(x,y) ds \\
& \leq k_2 t^{\frac{\alpha-1}{\alpha}} e^{-tA}(x,y) + \\
& k_3 t^{\frac{\alpha-1}{\alpha}} L_{R}(x,y) + k_3\kappa R^{1-\alpha}\int_0^t \big( E^{t-s}e^{-s(P^\varepsilon)^*}\big)(x,y) ds \\
& (\text{we are applying \eqref{l_est} to the second term} \\
& \text{and enlarge $R$, if needed, to have $k_3 \kappa R^{1-\alpha}\leq \frac{1}{2}$}) \\
& \leq (k_2 + k_3C') t^{\frac{\alpha-1}{\alpha}}  e^{-tA}(x,y) + \frac{1}{2} \int_0^t \big( E^{t-s}e^{-s(P^\varepsilon)^*}\big)(x,y) ds.
\end{align*}
Therefore,
$$
\int_0^t\big( E^\tau  e^{-(t-\tau)(P^\varepsilon)^*}\big)(x,y)d\tau \leq 2(k_2+k_3C')t^{\frac{\alpha-1}{\alpha}}  e^{-tA}(x,y),
$$
and so 
\begin{equation}
\label{l_est2}
L^{c}_{R}(x,y) \leq 2\kappa (k_2+k_3C')R^{1-\alpha}e^{-tA}(x,y).
\end{equation}

We now apply \eqref{l_est} and \eqref{l_est2} in \eqref{e_est} to obtain for all $|x| >2Rt^{\frac{1}{\alpha}}$, $|y| \leq |x|$ the desired bound
$$
e^{-t(P^\varepsilon)^*}(x,y) \leq C e^{-tA}(x,y) \quad \bigg(\leq 2Ce^{-tA}(x,y)\varphi_t(x)\bigg).
$$

\medskip

 The proof of the upper bound is completed.


\bigskip

\section{Proof of Theorem \ref{thm2_}: The lower bound $e^{-t\Lambda}(x,y)\geq C e^{-t(-\Delta)^\frac{\alpha}{2}}(x,y)\varphi_t(y)$}

\begin{proposition}
\label{claim1_lb}
There exists a constant $\hat{\mu}>0$ such that, for all $0<t\leq s$,
\[
\langle\varphi_s e^{-t\Lambda}\varphi_s^{-1} g\rangle \geq e^{-\frac{\hat{\mu}}{s}t}\langle g\rangle.
\]
for every $g=\varphi_s h$, $0\leq h\in \mathcal S$.
\end{proposition}

\begin{proof}

1.~For brevity, write $\varphi\equiv \varphi_s$. Set $g=\varphi h$ where, recall $\varphi=\varphi_{(1)}+\varphi_{(u)}$, $\varphi_{(1)}\in D((-\Delta)^{\frac{\alpha}{2}}_1)$, $\varphi_{(u)}\in D((-\Delta)^{\frac{\alpha}{2}}_{C_u})$. Then $(P^\varepsilon)^*\varphi=(P^\varepsilon)^*\varphi_{(1)}+(P^\varepsilon)^*\varphi_{(u)} \in L^1 + C_u$ (cf.\,proof of Theorem \ref{thm2} and Remark \ref{rem_hille}) and hence
\begin{align*}
\langle g\rangle-\langle \varphi e^{-t(P^\varepsilon-\mu)}h\rangle & =\langle (1-e^{-t((P^\varepsilon)^*-\mu)})\varphi,h\rangle= \\
& -\mu\int_0^t\langle e^{-\tau((P^\varepsilon)^*-\mu)}\varphi,h\rangle d\tau + \int_0^t\langle (P^\varepsilon)^* e^{-\tau((P^\varepsilon)^*-\mu)}\varphi,h\rangle d\tau
\end{align*}
or
\[
\langle g\rangle-\langle \varphi e^{-t(P^\varepsilon-\mu)}h\rangle= -\mu\int_0^t\langle\varphi,e^{-\tau(P^\varepsilon-\mu)}h\rangle d\tau + \int_0^t\langle (P^\varepsilon)^*\varphi, e^{-\tau(P^\varepsilon-\mu)}h\rangle d\tau,
\]
where $0 \leq e^{-\tau(P^\varepsilon-\mu)}h \in L^1 \cap C_u$ since $0 \leq h \in \mathcal S$ (see beginning of Section \ref{heat_sect}).
Write $$(P^\varepsilon)^*\varphi=(P^\varepsilon)^*\tilde{\varphi} + (P^\varepsilon)^*(\varphi-\tilde{\varphi}) =\mathbf{1}_{B(0,s^\frac{1}{\alpha})}(V-V_{\varepsilon})\varphi + v_\varepsilon,$$
where, recall, $\tilde{\varphi}(x)=(s^{-\frac{1}{\alpha}}|x|)^{-d+\beta}$ and $$V(x) \equiv V(|x|)=(\beta-\alpha)\kappa|x|^{-\alpha} \geq V_{\varepsilon}(x):=V(|x|_\varepsilon).$$
Routine calculation (cf.\,the proof of Proposition \ref{main_prop4}) shows that $\|v_\varepsilon\|_\infty \leq \frac{\mu_1}{s}$ for a $\mu_1\neq\mu_1(\varepsilon)$, so
 \[
\int_0^t\langle v_\varepsilon, e^{-\tau(P^\varepsilon-\mu)}h\rangle d\tau\leq \frac{\mu_1}{s}\int_0^t\langle e^{-\tau(P^\varepsilon-\mu)}h\rangle d\tau\leq \frac{2\mu_1}{s}\int_0^t\langle\varphi, e^{-\tau(P^\varepsilon-\mu)}h\rangle d\tau.
\]
Thus, taking $\mu=\frac{\hat \mu}{s}$ and $\hat{\mu}=2\mu_1$, we have
\begin{equation}
\label{ineq77}
\tag{$\diamond$}
\langle g\rangle-e^{\frac{\hat{\mu}}{s}t}\langle \varphi e^{-tP^\varepsilon}h\rangle\leq e^{\hat{\mu}}\int_0^t\langle \mathbf 1_{B(0,s^\frac{1}{\alpha})}(V-V_\varepsilon)\varphi,e^{-\tau P^\varepsilon}h\rangle d\tau.
\end{equation}

2.~Now, we take $\varepsilon \downarrow 0$ in \eqref{ineq77}.

The  RHS of \eqref{ineq77} tends to $0$ as $\varepsilon \downarrow 0$ due to the Dominated Convergence Theorem. Indeed, $V-V_{\varepsilon} \rightarrow 0$ a.e.\,on $\mathbb R^d$, while $\|e^{-\tau P^\varepsilon}h\|_\infty \leq \|h\|_\infty$ and so
\begin{align*}
\mathbf{1}_{B(0,s^\frac{1}{\alpha})}(V-V_{\varepsilon})\varphi  & \leq 2 \varphi \mathbf{1}_{B(0,s^\frac{1}{\alpha})}V \\
& \leq C  \mathbf{1}_{B(0,s^\frac{1}{\alpha})}(s^{-\frac{1}{\alpha}}|x|)^{-d+\beta}|x|^{-\alpha} \in L^1 \quad \text{ since } d-\beta+\alpha<d. 
\end{align*}
Next, by Proposition \ref{prop4}(\textit{i}), $e^{-tP^{\varepsilon_i}}h \rightarrow e^{-t\Lambda}h$ a.e.\,on $\mathbb R^d$ for some $\varepsilon_i \downarrow 0$. The upper bound $e^{-tP^\varepsilon}(x,y) \leq Ce^{-tA}(x,y)\varphi_t(y)$, see \eqref{apr_ub}, yields $\varphi e^{-tP^\varepsilon}h \leq C\varphi e^{-tA} g \in L^1$, so we can apply the Dominated Convergence Theorem in the LHS of \eqref{ineq77} to obtain that it
converges to $\langle g\rangle-e^{\frac{\hat{\mu}}{s}t}\langle \varphi e^{-t\Lambda}h\rangle$ as $\varepsilon_i \downarrow 0$.

The proof of Proposition \ref{claim1_lb} is completed.
\end{proof}

In what follows, as before, $A \equiv (-\Delta)^\frac{\alpha}{2}$.

\begin{proposition}
\label{lem4_lb}
$\langle h \rangle = \langle e^{-t\Lambda^*}h \rangle$ for every $h \in \mathcal S$, $t>0$.
\end{proposition}
\begin{proof}
We have, for $h \in \mathcal S$,
\begin{align*}
\langle h\rangle-\langle e^{-t(P^\varepsilon)^*}h\rangle &= \int_0^t\langle (P^\varepsilon)^* e^{-\tau(P^\varepsilon)^*}h,1\rangle d\tau  = \int_0^t\langle U_\varepsilon e^{-\tau(P^\varepsilon)^*}h\rangle d\tau \\
& = \int_0^t\langle \mathbf{1}_{B^c(0,1)}U_\varepsilon e^{-\tau(P^\varepsilon)^*}h\rangle d\tau + \int_0^t\langle \mathbf{1}_{B(0,1)}U_\varepsilon e^{-\tau(P^\varepsilon)^*}h\rangle d\tau.
\end{align*}
Since $e^{-t(P^\varepsilon)^*}$ is $L^1$ contraction, we have
$
\langle \mathbf{1}_{B^c(0,1)}U_\varepsilon e^{-\tau(P^\varepsilon)^*}h\rangle \leq \|\mathbf{1}_{B^c(0,1)}U_\varepsilon\|_\infty \|h\|_1 \rightarrow 0$ as $\varepsilon \downarrow 0,
$
and so the first integral converges to $0$.
Let us estimate the second integral:
\begin{align*}
& \int_0^t \langle \mathbf{1}_{B(0,1)}U_\varepsilon e^{-\tau(P^\varepsilon)^*}h\rangle d\tau  = \int_0^t \langle e^{-\tau P^\varepsilon}\mathbf{1}_{B(0,1)}U_\varepsilon,h\rangle d\tau \\
&\text{(we are using the upper bound } e^{-\tau P^\varepsilon}(x,y) \leq C e^{-\tau A}(x,y)\varphi_t(y), \text{ see \eqref{apr_ub}})\\
&\leq C\int_0^t \langle e^{-\tau A} \varphi\mathbf{1}_{B(0,1)}U_\varepsilon,|h|\rangle d\tau \\
&\leq Ct \|h\|_\infty\|\varphi_t\mathbf{1}_{B(0,1)}U_\varepsilon\|_1 \rightarrow 0 \text{ as } \varepsilon \downarrow 0 \quad \text{ due to } d-\beta+\alpha <d
\end{align*}
(recall that $U_\varepsilon(x)=\alpha \kappa \varepsilon|x|_\varepsilon^{-\alpha-2} \leq \alpha \kappa |x|_\varepsilon^{-\alpha}$).

Thus, $\langle h\rangle= \lim_{\varepsilon \downarrow 0}\langle e^{-t(P^\varepsilon)^*}h\rangle$. 

By Proposition \ref{prop5}, we have $e^{-t(P^{\varepsilon_i})^*}h \rightarrow e^{-t\Lambda^*}h$ a.e.\,on $\mathbb R^d$. The upper bound  $e^{-t(P^\varepsilon)^*}(x,y) \leq C e^{-tA}(x,y)\varphi_t(x)$ yields $|e^{-t(P^\varepsilon)^*}h| \leq C  \varphi_t e^{-tA}|h| \in L^1$, and so $\lim_i\langle e^{-t(P^{\varepsilon_i})^*}h \rangle = \langle e^{-t\Lambda^*}h \rangle$ by the Dominated Convergence Theorem. Thus, equality $\langle h \rangle = \langle e^{-t\Lambda^*}h \rangle$ holds for every $h \in \mathcal S$.
\end{proof}

\medskip

\textit{In the rest of the proof, all pointwise inequalities are understood in the sense a.e.}.

\medskip
   
Given constants $0<r<R_0<R$, we denote
$$
r_t:=rt^{\frac{1}{\alpha}}, \quad R_{0,t}:=R_0t^{\frac{1}{\alpha}}, \quad R_t:=Rt^{\frac{1}{\alpha}},
$$
and put $\mathbf 1_{b,a}:=\mathbf 1_{B(0,b)}-\mathbf 1_{B(0,a)}$ for any $0<a<b$.

We need the following consequence of Proposition \ref{claim1_lb} and Proposition \ref{lem4_lb}.

\begin{proposition}
\label{ANcorol2} 
For every $R_0>0$ there exist constants $0<r<R_0<R$ such that

(\textit{i}) \begin{equation*}
e^{-\hat{\mu}-1}\varphi_t(x)\leq e^{-t\Lambda^*}\varphi_t \mathbf 1_{R_t,r_t}(x) \;\; \text{ for all }x \in B(0,R_{0,t}), \quad x \neq 0,
\end{equation*}
where $\hat{\mu}>0$ is from Proposition \ref{claim1_lb};

(\textit{ii})
$$
\frac{1}{2} \leq e^{-t\Lambda}\mathbf{1}_{R_t,r_t}(x) \quad \text{ for all $x \in B(0,R_{0,t})$}.
$$

\end{proposition}

\begin{proof}(\textit{i}) By duality,
it suffices to prove 
that for all $g:=\varphi_t h$, $0\leq h\in \mathcal S$ with $\sprt h\subset B(0,R_{0,t})$ we have
\[
e^{-\hat{\mu}-1}\langle g\rangle\leq\langle \mathbf 1_{R_t,r_t}\varphi_t e^{-t\Lambda}\varphi_t^{-1}g\rangle, 
\]
By the upper bound $e^{-t\Lambda}(x,y) \leq Ce^{-tA}(x,y)\varphi_t(y)$ proved in the previous section,
\begin{align*}
\langle\mathbf 1_{B(0,r_t)}\varphi_t e^{-t\Lambda}\varphi_t^{-1}g\rangle &\leq C\langle\mathbf 1_{B(0,r_t)}\varphi_t,e^{-tA}g\rangle\\
&\leq CC_1t^{-\frac{d}{\alpha}}\|\mathbf 1_{B(0,r_t)}\varphi_t\|_1\|g\|_1\\
&= CC_1\|\mathbf 1_{B(0,r)}\varphi_1\|_1\|g\|_1, \quad \|\mathbf 1_{B(0,r)}\varphi_1\|_1 \rightarrow 0 \text{ as } r \downarrow 0.
\end{align*}
\begin{align*}
\langle\mathbf 1_{B^c(0,R_t)}\varphi_t e^{-t\Lambda}\varphi_t^{-1}g\rangle &\leq C\langle\mathbf 1_{B^c(0,R_t)}\varphi_t,e^{-tA}g\rangle\\
&\leq C\langle e^{-tA}\mathbf 1_{B^c(0,R_t)},g\mathbf 1_{B(0,R_{0,t})}\rangle \\
&\leq C\sup_{x\in B(0,R_{0,t})}e^{-tA}\mathbf 1_{B^c(0,R_t)}(x)\|g\|_1\\
&\leq C(R_0,R)\|g\|_1, \quad C(R_0,R)\rightarrow 0 \text{ as } R-R_0 \uparrow \infty,
\end{align*}
where at the last step we have used, for $x \in B(0,R_{0,t})$, $y \in B^c(0,R_{t})$ and $\tilde{x}=R_0^{-1}t^{-\frac{1}{\alpha}}x \in B(0,1)$, $\tilde{y}=R^{-1}t^{-\frac{1}{\alpha}}y \in B^c(0,1)$, 
\begin{align*}
e^{-tA}(x,y) & \leq k_0 t |x-y|^{-d-\alpha} \leq k_0 t |R_0 t^{\frac{1}{\alpha}}\tilde{x}-R t^{\frac{1}{\alpha}}\tilde{y}|^{-d-\alpha} < 2k_0 t^{-\frac{d}{\alpha}}(R-R_0)^{-d-\alpha}|\tilde{y}|^{-d-\alpha}.
\end{align*}
It remains to apply Proposition \ref{claim1_lb}.

\medskip

(\textit{ii}) We follow the argument in (\textit{i}). By duality, it suffices to prove 
that for all $0\leq h\in \mathcal S$ with $\sprt h\subset B(0,R_{0,t})$ we have
\[
\frac{1}{2}\langle h \rangle \leq \langle \mathbf{1}_{R_t,r_t} e^{-t\Lambda^*}h\rangle.
\]
By the upper bound $e^{-t\Lambda}(x,y) \leq Ce^{-tA}(x,y)\varphi_t(y)$,
\begin{align*}
\langle\mathbf 1_{B(0,r_t)} e^{-t\Lambda^*} h\rangle &\leq C\langle\mathbf 1_{B(0,r_t)}\varphi_t,e^{-tA}h\rangle\\
&\leq CC_1t^{-\frac{d}{\alpha}}\|\mathbf 1_{B(0,r_t)}\varphi_t\|_1\|h\|_1\\
&= o(r)\|h\|_1, \quad o(r)\rightarrow 0 \text{ as } r\downarrow 0;
\end{align*}
\begin{align*}
\langle\mathbf 1_{B^c(0,R_t)}e^{-t\Lambda^*} h\rangle &\leq C\langle\mathbf 1_{B^c(0,R_t)}\varphi_t,e^{-tA}h\rangle\\
&\leq C\langle e^{-tA}\mathbf 1_{B^c(0,R_t)},h \mathbf 1_{B(0,R_{0,t})}\rangle \\
&\leq C\sup_{x\in B(0,R_{0,t})}e^{-tA}\mathbf 1_{B^c(0,R_t)}(x)\|h\|_1\\
&= C(R_0,R)\|h\|_1, \quad C(R_0,R)\rightarrow 0 \text{ as } R-R_0 \uparrow \infty.
\end{align*}
The last two estimates and Proposition \ref{lem4_lb} yield $\frac{1}{2}\langle h \rangle \leq \langle \mathbf{1}_{R_t,r_t} e^{-t\Lambda^*}h\rangle$.  
\end{proof}

%
%
%

Next, we show that away from the singularity the ``standard lower bound'' on the heat kernel $e^{-t\Lambda^*}(x,y)$ is valid.

\begin{proposition}
\label{claim2_lb}
There exists a constant $r \geq 2$ such that for all $t>0$
we have
\[
e^{-t\Lambda^*}(x,y) \geq \frac{1}{2} e^{-tA}(x,y) \quad \text{ for all }|x| \geq rt^{\frac{1}{\alpha}},\; |y| \geq rt^{\frac{1}{\alpha}}. 
\]
\end{proposition}
\begin{proof}
By the Duhamel formula (where we discard positive term containing $W_\varepsilon$ in the RHS)
\[
e^{-t(P^\varepsilon)^*}(x,y) \geq e^{-tA}(x,y)  - |M_t(x,y)|, \qquad M_t(x,y) \equiv \int_0^t \langle e^{-(t-\tau)(P^\varepsilon)^*}(x,\cdot) b_\varepsilon (\cdot)\cdot\nabla_\cdot e^{-\tau A}(\cdot,y)\rangle d\tau.
\]
By Lemma \ref{claim_lem}(\textit{i}),
\begin{align*}
|M_t(x,y)| & \leq k_2 \int_0^t \big\langle  e^{-(t-\tau)(P^\varepsilon)^*}(x,\cdot)|\cdot|^{1-\alpha}E^\tau(\cdot,y)\big\rangle d\tau  \\
& (\text{we apply the upper bound}) \\
& \leq k_2 C \int_0^t \varphi_{t-\tau}(x) \big\langle  e^{-(t-\tau)A}(x,\cdot)|\cdot|^{1-\alpha}E^\tau(\cdot,y)\big\rangle d\tau \\
& (\text{since $|x| \geq rt^{\frac{1}{\alpha}}$, where $r \geq 2$, we have $\varphi_{t-\tau}(x)=\frac{1}{2}$}) \\
& \leq  \frac{k_2 C}{2} \int_0^t \big\langle  e^{-(t-\tau)A}(x,\cdot)|\cdot|^{1-\alpha}E^\tau(\cdot,y)\big\rangle d\tau \\
& =: J(|\cdot|^{1-\alpha}).
\end{align*}

Let $0<\gamma<\frac{1}{2}$, to be chosen later. We have for all $0<\tau<t$, $|x|$, $|y| \geq rt^{\frac{1}{\alpha}}$,
$$
\mathbf{1}_{B(0,\gamma r t^{\frac{1}{\alpha}})}(\cdot) e^{-(t-\tau)A}(x,\cdot)  \leq C_5e^{-tA}(x,0),
$$
$$
\mathbf{1}_{B(0,\gamma r t^{\frac{1}{\alpha}})}(\cdot)E^{\tau }(\cdot,y) \leq C_7t^{-\frac{1}{\alpha}}e^{-tA}(0,y),
$$
where $C_5 \neq C_5(\gamma)$, $C_7 \neq C_7(\gamma)$.
Therefore, using the inequality 
\[
e^{-tA}(x,z)e^{-\tau A}(z,y) \leq K e^{-(t+\tau )A}(x,y)\bigl(e^{-tA}(x,z) + e^{-\tau A}(z,y) \bigr), 
\]
which holds for a constant $K=K(d,\alpha)$, all $x,z,y \in \mathbb R^d$ and $t,\tau>0$ (see e.g. \cite[(9)]{BJ}), we have 
\begin{align}
J(\mathbf{1}_{B(0,\gamma r t^{\frac{1}{\alpha}})}|\cdot|^{1-\alpha}) &\leq ct^{-\frac{1}{\alpha}}\int_0^t \langle \mathbf{1}_{B(0,\gamma r t^{\frac{1}{\alpha}})}(\cdot)|\cdot|^{1-\alpha}\rangle d\tau (e^{-tA}(x,0)+e^{-tA}(0,y)) e^{-2tA}(x,y) \notag \\
& (\text{we are using that $e^{-tA}(x,0)$, $e^{-tA}(0,y) \leq k_0 t^{-\frac{d}{\alpha}}r^{-d+\alpha}$ }) \notag \\
& \leq  2ck_0 t^{-\frac{d+1}{\alpha}} r^{-d+\alpha}  e^{-tA}(x,y) \int_0^t \langle \mathbf{1}_{B(0,\gamma r t^{\frac{1}{\alpha}})}(\cdot)|\cdot|^{1-\alpha} \rangle d \tau \notag \\
& \leq c'  r \gamma^{d-\alpha+1} e^{-tA}(x,y). \label{J}
\tag{$\ast$}
\end{align}

In turn, using Lemma \ref{claim_lem}(\textit{iii}), we obtain
\begin{equation}
\label{J_}
\tag{$\ast\ast$}
J(\mathbf{1}_{B^c(0,\gamma r t^{\frac{1}{\alpha}})}|\cdot|^{1-\alpha}) \leq c'' (\gamma r)^{1-\alpha} e^{-tA}(x,y), \quad c''=\frac{k_2 k_3C}{2}.
\end{equation}

Without loss of generality suppose that $1\leq c'\leq c''$. Take any $0<\gamma<\frac{1}{2}$ and $r>2$ such that $4c''\leq \gamma^{-\frac{d-\alpha}{\alpha}\frac{1} {\alpha-1}}$ and $r\in](4c'')^{\alpha-1}\gamma^{-1},\gamma^{-\frac{d}{\alpha}}[$. Then $c'r \gamma^{d-\alpha+1}\leq c''(\gamma r)^{1-\alpha}\leq\frac{1}{4}$, and so by \eqref{J}, \eqref{J_},
$$
|M_t(x,y)| \leq \frac{1}{2} e^{-tA}(x,y).
$$
Thus,
$$
e^{-t(P^\varepsilon)^*}(x,y) \geq \frac{1}{2} e^{-tA}(x,y), \quad |x| \geq rt^{\frac{1}{\alpha}},\; |y| \geq r t^{\frac{1}{\alpha}}.
$$
Finally, to take $\varepsilon \downarrow 0$, we repeat the argument in the beginning of Section \ref{ub_sect} but use the convergence $e^{-t(P^{\varepsilon_i})^*}h \rightarrow e^{-t\Lambda^*}h$ strongly in $L^r$, $r \in ]1,r_c'[$ for some $\varepsilon_i \downarrow 0$ for every $h \in L^r$ (Proposition \ref{prop5}). This completes the proof of Proposition \ref{claim2_lb}.
\end{proof}

Next, we show that the constant $r$ in Proposition \ref{claim2_lb} can, in fact, be made as small as needed at expense of decreasing the constant in front of $e^{-tA}(x,y)$.

\begin{proposition}
\label{claim3_lb}
For every $r>0$, there is a constant $c(r)>0$ such that
\[
e^{-t\Lambda^*}(x,y) \geq c(r) e^{-tA}(x,y)
\]
whenever $|x| \geq rt^\frac{1}{\alpha}$, $|y| \geq rt^\frac{1}{\alpha}$, $t>0$.
\end{proposition}

\begin{proof}
In Proposition \ref{claim2_lb}, there is an $r$ for which
\begin{equation}
\label{e11}
e^{-t\Lambda^*}(x,y) \geq 2^{-1} e^{-tA}(x,y), \quad |x| \geq rt^\frac{1}{\alpha}, \quad |y| \geq rt^\frac{1}{\alpha},
\end{equation}
\begin{equation}
\label{e2}
e^{-t\frac{1}{2}\Lambda^*}(x,y) \geq 2^{-1} e^{-\frac{t}{2}A}(x,y), \quad |x| \geq r\bigg(\frac{t}{2}\bigg)^\frac{1}{\alpha}, \quad |y| \geq r\bigg(\frac{t}{2}\bigg)^\frac{1}{\alpha}.
\end{equation}
We now extend \eqref{e11}, i.e.\,we prove existence of a constant $0<c_1<2^{-1}$ such that
\begin{equation}
\label{e3}
\tag{$\ref{e11}'$}
e^{-t\Lambda^*}(x,y) \geq c_1 e^{-tA}(x,y), \quad |x| \geq r\bigg(\frac{t}{2}\bigg)^\frac{1}{\alpha}, \quad |y| \geq r\bigg(\frac{t}{2}\bigg)^\frac{1}{\alpha}.
\end{equation}
We need to consider only the case $rt^\frac{1}{\alpha} \geq |x| \geq r\bigg(\frac{t}{2}\bigg)^\frac{1}{\alpha}$, $r \geq |y| \geq r\bigg(\frac{t}{2}\bigg)^\frac{1}{\alpha}$. Recall that
\begin{equation}
\label{st_bd}
k_0^{-1} t\bigl(|x-y|^{-d-\alpha} \wedge t^{-\frac{d+\alpha}{\alpha}}\bigr) \leq e^{-tA}(x,y) \leq k_0 t\bigl(|x-y|^{-d-\alpha} \wedge t^{-\frac{d+\alpha}{\alpha}}\bigr), \quad x,y \in \mathbb R^d,\;t>0.
\end{equation}
By the reproduction property,
\begin{align*}
e^{-t\Lambda^*}(x,y) &\geq \langle e^{-\frac{1}{2}t\Lambda^*}(x,\cdot)\mathbf{1}_{B^c\big(0,r\big(\frac{t}{2}\big)^\frac{1}{\alpha}\big)}(\cdot)e^{-\frac{1}{2}t\Lambda^*}(\cdot,y)\rangle \\
& (\text{we are applying \eqref{e2}}) \\
& \geq 2^{-2}\langle  e^{-\frac{1}{2}tA}(x,\cdot)\mathbf{1}_{B^c\big(0,r\big(\frac{t}{2}\big)^\frac{1}{\alpha}\big)}(\cdot)e^{-\frac{1}{2}tA}(\cdot,y)\rangle \\
& > 2^{-2}\langle  e^{-\frac{1}{2}tA}(x,\cdot)\mathbf{1}_{B\big(0,(r+1)\big(\frac{t}{2}\big)^\frac{1}{\alpha}\big) - B\big(0,r\big(\frac{t}{2}\big)^\frac{1}{\alpha}\big)}(\cdot)e^{-\frac{1}{2}tA}(\cdot,y)\rangle \\
& (\text{we are using the lower bound in \eqref{st_bd}}) \\
&\geq 2^{-2} \tilde{c}t^{-\frac{d}{\alpha}} \qquad \text{for appropriate }\tilde{c}=\tilde{c}(r)>0 \\
& (\text{we are using the upper bound in \eqref{st_bd}}) \\
&\geq c_1 e^{-tA}(x,y) \qquad \text{ for appropriate } 0<c_1=c_1(r)<2^{-1},
\end{align*}
i.e.\,we have proved \eqref{e3}.

The same argument yields
\begin{equation}
\label{e4}
\tag{$\ref{e2}'$}
e^{-\frac{1}{2}t\Lambda^*}(x,y) \geq c_1 e^{-\frac{1}{2}tA}(x,y), \quad |x| \geq r\bigg(\frac{t}{2^2}\bigg)^\frac{1}{\alpha}, \quad |y| \geq r\bigg(\frac{t}{2^2}\bigg)^\frac{1}{\alpha}.
\end{equation}

We repeat the above procedure $m-1$ times, obtaining
$$
e^{-t\Lambda^*}(x,y) \geq c_m e^{-tA}(x,y), \quad |x| \geq r\bigg(\frac{t}{2^m}\bigg)^\frac{1}{\alpha}, \quad |y| \geq r\bigg(\frac{t}{2^m}\bigg)^\frac{1}{\alpha}
$$
for appropriate $c_m>0$, from which the assertion of Proposition \ref{claim3_lb} follows.
\end{proof}

We are in position to complete the proof of the lower bound using the so-called $3q$ argument.
Set $q_t(x,y):=\varphi_t^{-1}(x)e^{-t\Lambda^*}(x,y)$.

\smallskip

(a)~Let $x,y \in B^c(0,t^{\frac{1}{\alpha}})$, $x \neq y$. Then, using that $\varphi_{3t}^{-1}\geq \tilde{c}>0$ on $B^c(0,t^{\frac{1}{\alpha}})$, we have by Proposition \ref{claim3_lb}
$$
q_{3t}(x,y) \geq \tilde{c} e^{-3t\Lambda^*}(x,y) \geq c e^{-3tA}(x,y), \quad c= \tilde{c} c(1).
$$

\medskip

Now, let $r_t=rt^{\frac{1}{\alpha}}$, $R_t=Rt^{\frac{1}{\alpha}}$ be as in Proposition \ref{ANcorol2}, where we fix $R_0=1$ (so $r<1<R$). Recall that $\mathbf{1}_{b,a}=\mathbf{1}_{B(0,b)} - \mathbf{1}_{B(0,a)}$. 

\smallskip

(b)~Let $x \in B(0,t^{\frac{1}{\alpha}})$, $|y| \geq rt^{\frac{1}{\alpha}}$, $x \neq y$. By the reproduction property,
\begin{align*}
q_{2t}(x,y) &\geq \varphi_{2t}^{-1}(x)\langle e^{-t\Lambda^*}(x,\cdot) \varphi_t^{-1}(\cdot)\varphi_t(\cdot) e^{-t\Lambda^*}(\cdot,y)\mathbf{1}_{R_t,r_t}(\cdot) \rangle \notag \\
& \geq \varphi_{2t}^{-1}(x)\varphi_t^{-1}(r_t) \langle e^{-t\Lambda^*}(x,\cdot)\varphi_t(\cdot) e^{-t\Lambda^*}(\cdot,y)\mathbf{1}_{R_t,r_t}(\cdot) \rangle \notag \\
& \text{(we are applying Proposition \ref{ANcorol2}(\textit{i}) and using $\varphi_t^{-1}(r_t)=r^{d-\beta}$)} \notag  \\
&\geq e^{-\hat{\mu}-1}r^{d-\beta}\varphi_{2t}^{-1}(x) \varphi_t (x)\inf_{r_t\leq |z|\leq R_t} e^{-t\Lambda^*}(z,y)\notag\\
& (\text{we are using $\varphi_{2t}^{-1} \varphi_t=2^{-\frac{d-\beta}{\alpha}}$ and applying Proposition \ref{claim3_lb}}) \notag \\
&\geq e^{-\hat{\mu}-1}r^{d-\beta} 2^{-\frac{d-\beta}{\alpha}} c(r) e^{-tA}(x,y)\notag \\
& \geq C_1(r) e^{-tA}(x,y). \notag 
\end{align*}

(b') Let $x \in B(0,t^{\frac{1}{\alpha}}), |y| \geq t^{\frac{1}{\alpha}}$, $x \neq y$. Arguing as in (b), we obtain
$$
q_{3t}(x,y) \geq C_2(r) e^{-3tA}(x,y).
$$

(c)~Let $|x| \geq rt^{\frac{1}{\alpha}}$, $y \in B(0,t^{\frac{1}{\alpha}})$, $x \neq y$. We have
\begin{align*}
q_{2t}(x,y) &\geq \varphi^{-1}_{2t}(x) \langle e^{-t\Lambda^*}(x,\cdot)e^{-t\Lambda^*}(\cdot,y)\mathbf{1}_{R_t,r_t}(\cdot) \rangle \notag \\ 
&= \varphi_{2t}^{-1}(x)\langle  e^{-t\Lambda^*}(x,\cdot) e^{-t\Lambda}(y,\cdot)\mathbf{1}_{R_t,r_t}(\cdot)\rangle \notag \\ 
& \geq \varphi_{2t}^{-1}(r_t)\langle  e^{-t\Lambda^*}(x,\cdot) e^{-t\Lambda}(y,\cdot)\mathbf{1}_{R_t,r_t}(\cdot)\rangle \notag \\ 
& (\text{we are using $\varphi_{2t}^{-1}(r_t)=2^{-\frac{d-\beta}{\alpha}}r^{d-\beta}$ and applying Proposition \ref{claim3_lb}}) \notag \\
&\geq 2^{-\frac{d-\beta}{\alpha}}r^{d-\beta}c(r)\langle e^{-tA}(x,\cdot) e^{-t\Lambda}(y,\cdot)\mathbf{1}_{R_t,r_t}(\cdot) \rangle \notag \\ 
& \geq C_3(r)t(Rt^{\frac{1}{\alpha}}+|x|)^{-d-\alpha}\langle e^{-t\Lambda}(y,\cdot)\mathbf{1}_{R_t,r_t}(\cdot)\rangle \notag \\ 
& \text{(we are applying Proposition \ref{ANcorol2}(\textit{ii}))} \notag \\
& \geq C_3(r)2^{-1}t(Rt^{\frac{1}{\alpha}}+|x|)^{-d-\alpha} \geq C_4(r,R) e^{-2tA}(x,y). \notag
\end{align*}

(c') Let $|x| \geq t^{\frac{1}{\alpha}}$, $y \in B(0,t^{\frac{1}{\alpha}})$, $x \neq y$. Arguing as in (c), we obtain
$$
q_{3t}(x,y) \geq C_5(r) e^{-3tA}(x,y).
$$

(d)~Let $x,y \in B(0,t^{\frac{1}{\alpha}})$, $x \neq y$. By the reproduction property,
\begin{align*}
q_{3t}(x,y) & =  \varphi_{3t}^{-1}(x)\langle e^{-t\Lambda^*}(x,\cdot)e^{-2t\Lambda^*}(\cdot,y)\rangle 
 \\
& \geq \varphi_{3t}^{-1}(x)\langle e^{-t\Lambda^*}(x,\cdot)e^{-2t\Lambda^*}(\cdot,y)\mathbf{1}_{R_t,r_t}(\cdot)\rangle \\
& (\text{by (c), we have $e^{-2t\Lambda^*}(\cdot,y)\mathbf{1}_{R_t,r_t}(\cdot) \geq \varphi_{2t}(\cdot)e^{-2tA}(\cdot,y)\mathbf{1}_{R_t,r_t}(\cdot)$}) \\
& \geq \varphi_{3t}^{-1}(x)C_4(r)\langle e^{-t\Lambda^*}(x,\cdot)\varphi_{2t}(\cdot)e^{-2tA}(\cdot,y)\mathbf{1}_{R_t,r_t}(\cdot) \rangle\\
& (\text{we are using  $e^{-2tA}(z,y) \geq c_{r,R}t^{-\frac{d}{\alpha}}>0$ for $r_t \leq |z| \leq R_t$, $|y| \leq t^{\frac{1}{\alpha}}$}) \\
& \geq C_4(r)c_{r,R} t^{-\frac{d}{\alpha}}\varphi_{3t}^{-1}(x) \langle e^{-t\Lambda^*}(x,\cdot)\mathbf{1}_{R_t,r_t}(\cdot) \varphi_{2t}(\cdot)\rangle \\
&\text{(we are applying Proposition \ref{ANcorol2}(\textit{i}))} \\
& \geq C_4(r)c_{r,R}e^{-\hat{\mu}-1}t^{-\frac{d}{\alpha}} \varphi_{3t}^{-1}(x) \varphi_{2t}(x) \\
& (\text{we are using $\varphi_{3t}^{-1} \varphi_{2t}=(3/2)^{-\frac{d-\beta}{\alpha}}$}) \\ 
& \geq C_4(r)c_{r,R}e^{-\hat{\mu}-1}(3/2)^{-\frac{d-\beta}{\alpha}} t^{-\frac{d}{\alpha}} \geq C_5(r,R) e^{-3tA}(x,y).
\end{align*}

By (a), (b'), (c'), (d),
$
q_{3t}(x,y) \geq C e^{-3tA}(x,y)$ for all $x,y \in \mathbb R^d$, $x \neq y$,
and so
$
e^{-3t\Lambda^*}(x,y) \geq C e^{-3tA}(x,y)\varphi_{3t}(x),
$ $x \neq 0$.

\section{Operator realization of $(-\Delta)^{\frac{\alpha}{2}} + b \cdot \nabla$ in $L^2$}

\label{rellich_app}

Recall 
$$
b(x)=\kappa |x|^{-\alpha}x, \quad \kappa=\sqrt{\delta} \frac{2^{\alpha+1}}{d-\alpha}\frac{\Gamma^2(\frac{d}{4}+\frac{\alpha}{4})}{\Gamma^2(\frac{d}{4}-\frac{\alpha}{4})} \equiv \sqrt{\delta} (d-\alpha)^{-1}2c^{2}_{d,\alpha},
$$
where
$c_{d,\alpha}=\frac{\gamma(\frac{d}{2})}{\gamma(\frac{d}{2}-\frac{\alpha}{2})}$, $ \gamma(s):=\frac{2^s\pi^\frac{d}{2}\Gamma(\frac{s}{2})}{\Gamma(\frac{d}{2}-\frac{s}{2})}$.


\begin{proposition}
\label{prop_2}
\label{prop_reg}Let $0<\delta<1$. Then the algebraic sum
$\Lambda:=(-\Delta)^\frac{\alpha}{2}+b\cdot\nabla$, $D(\Lambda)=\mathcal W^{\alpha,2}\;(=\bigl(1+(-\Delta)^{\frac{\alpha}{2}}\bigr)^{-1}L^2)$, is the (minus) generator of a holomorphic semigroup in $L^2$.
\end{proposition}

\begin{proof}[Proof of Proposition \ref{prop_reg}]
We show that $b \cdot \nabla$ is Rellich's perturbation of $(-\Delta)^{\frac{\alpha}{2}}$.

For brevity, write $\|\cdot\| \equiv \|\cdot\|_{2\rightarrow 2}$ and $A\equiv(-\Delta)^{\frac{\alpha}{2}}$ in $L^2$.

Define $T=b\cdot\nabla(\zeta+A)^{-1}$, $\Real \zeta>0$, and note that 
\begin{align*}
\|T\| &\leq \||b|(\zeta+A)^{-1+\frac{1}{\alpha}}\|\|\nabla(\zeta+A)^{-\frac{1}{\alpha}}\| \\
& (\text{we are using $\|\nabla g\|_2=\|(-\Delta)^{\frac{1}{2}}g\|_2$}) \\
& \leq \||b|(\Real \zeta+A)^{-1+\frac{1}{\alpha}}\|\|A^{\frac{1}{\alpha}}(\zeta+A)^{-\frac{1}{\alpha}}\| \\
& (\text{by the Spectral Theorem, $\|A^\frac{1}{\alpha}(\zeta+A)^{-\frac{1}{\alpha}}\|\leq 1$}) \\
& \leq \||b|(-\Delta)^{-\frac{\alpha-1}{2}}\|  = \kappa\||x|^{-\alpha+1}(-\Delta)^{-\frac{\alpha-1}{2}}\| \\
& (\text{we are using \cite[Lemma 2.7]{KPS}}) \\
&=\kappa c(\alpha-1,d)<\sqrt{\delta}, 
\end{align*}
where $c(\alpha-1,d):=\gamma(\frac{d}{2}-\alpha+1)/\gamma(\frac{d}{2})$ and, recall,  $\kappa=\sqrt{\delta} (d-\alpha)^{-1}2c^{2}_{d,\alpha}$. The last inequality holds
 because $c(\alpha-1,2)<(d-\alpha)2^{-1}c_{d,\alpha}^{-2}$ or, equivalently,
\[
F(\alpha)\equiv(d-\alpha)\Gamma\big(\frac{d-2+2\alpha}{4}\big)\big[\Gamma\big(\frac{d-\alpha}{4}\big)\big]^2-4 \Gamma\big(\frac{d+2-2\alpha}{4}\big)\big[\Gamma\big(\frac{d+\alpha}{4}\big)\big]^2>0
\]
(the latter is due to $\frac{d^2}{dt^2}\log\Gamma(t)\geq 0$ and $F(2)=0$ $\big((d-2)\Gamma(\frac{d-2}{4})=4 \Gamma(\frac{d+2}{4})\big)$).

Thus, the Neumann series for $(\zeta+\Lambda)^{-1}=(\zeta+A)^{-1}(1+T)^{-1}$ converges, and
$$
\|(\zeta+\Lambda)^{-1}\| \leq (1-\sqrt{\delta})^{-1}|\zeta|^{-1}, \quad \Real \zeta>0,
$$
i.e.\,$-\Lambda$ is the generator of a holomorphic semigroup. 
\end{proof}

\begin{proposition} 
\label{prop_2_conv}
In the assumptions of Proposition \ref{prop_2}, $e^{-tP^\varepsilon}\overset{s}\rightarrow e^{-t\Lambda}$ in $L^2$.
\end{proposition}

\begin{proof}
It suffices to show that $(\mu+P^\varepsilon)^{-1}\overset{s}\rightarrow (\mu+\Lambda)^{-1}$ (in $L^2$) for a $\mu>0$. 

First, we show that $(\mu+\Lambda^\varepsilon)^{-1}\overset{s}\rightarrow (\mu+\Lambda)^{-1}$.
We will use notation introduced in the proof of Proposition \ref{prop_reg} above. Recall: $(\mu+\Lambda)^{-1}=(\mu+A)^{-1}(1+T)^{-1}$, $\|(\mu+\Lambda)^{-1}\|\leq (1-\sqrt{\delta})^{-1}\mu^{-1}$. Since $\|(T-T_\varepsilon)f\|_2\leq \||b-b_\varepsilon|(\mu+A)^{-1}|\nabla f|\|_2\rightarrow 0$ for every $f\in C_c^\infty$ by the Dominated Convergence Theorem, we have $T_\varepsilon \overset{s}\rightarrow T$. Therefore, $(\mu+\Lambda^\varepsilon)^{-1}\overset{s}\rightarrow (\mu+\Lambda)^{-1}$.

We show that $(\mu+P^\varepsilon)^{-1}-(\mu+\Lambda^\varepsilon)^{-1}\overset{s}\rightarrow 0$. 
 Set $S=(\mu+A)^{-1+\frac{1}{\alpha}}b\cdot \nabla(\mu+A)^{-\frac{1}{\alpha}}$ and $S_\varepsilon=(\mu+A)^{-1+\frac{1}{\alpha}}b_\varepsilon \cdot \nabla(\mu+A)^{-\frac{1}{\alpha}}$. Then $\sup_\varepsilon\|S_\varepsilon\|<1$, $\|S\|<1$ and
\begin{equation}
\label{re_1}
\tag{$\star$}
(\mu+\Lambda^\varepsilon)^{-1}=(\mu+A)^{-\frac{1}{\alpha}}(1+S_\varepsilon)^{-1}(\mu+A)^{-1+\frac{1}{\alpha}}, \quad \mu>0.
\end{equation}

Now, let $h\in L^2\cap L^\infty$. Then $$\|(\mu+P^\varepsilon)^{-1}h-(\mu+\Lambda^\varepsilon)^{-1}h\|_2=\|(\mu+\Lambda^\varepsilon)^{-1}U_\varepsilon(\mu+P^\varepsilon)^{-1}h\|_2\leq K_1+K_2,$$ 
\begin{align*}
K_1&=\|(\mu+\Lambda^\varepsilon)^{-1}U_\varepsilon\mathbf 1_{B(0,1)}(\mu+P^\varepsilon)^{-1}h\|_2 \\
&
\leq \|(\mu+\Lambda^\varepsilon)^{-1}|x|^{-\alpha+1}\|\||x|^{\alpha-1}U_\varepsilon\mathbf 1_{B(0,1)}\|_2\mu^{-1}\|h\|_\infty \\
& \text{(we are using \eqref{re_1})}\\
&\leq C\mu^{-1}\|h\|_\infty\|\varepsilon|x|_\varepsilon^{-2}|x|^{-1}\mathbf 1_{B(0,1)}\|_2\rightarrow 0,
\end{align*} 
$$K_2=\|(\mu+\Lambda^\varepsilon)^{-1}U_\varepsilon\mathbf 1_{B^c(0,1)}(\mu+P^\varepsilon)^{-1}h\|_2\leq \kappa\alpha\varepsilon(1-\sqrt{\delta})^{-1}\mu^{-2}\|h\|_2\rightarrow 0.$$
The convergence $e^{-tP^\varepsilon}\overset{s}\rightarrow e^{-t\Lambda}$ is established.
\end{proof}

Similar arguments show that $e^{-t(P^\varepsilon)^*}\overset{s}\rightarrow e^{-t\Lambda^*}$ in $L^2$, $\Lambda^* \equiv \Lambda^*_2$ (cf.\,Section \ref{appA_}).

\begin{remark}
Above we could have constructed an operator realization $\Lambda$ of $(-\Delta)^{\frac{\alpha}{2}} +  b \cdot \nabla$ on $L^2$ for $b(x):=\sqrt{\delta_2} c^{-2}(\frac{\alpha-1}{2},d) |x|^{-\alpha}x$, $0<\delta_2<1$, where $c^{-2}(\frac{\alpha-1}{2},d)=\gamma(\frac{d}{2}-\frac{\alpha-1}{2})/\gamma(\frac{d}{2})$ following the arguments in \cite[sect.\,4]{KiS1}. Note that
$$
c^{-1}(\alpha-1,d)<c^{-2}(\frac{\alpha-1}{2},d)
$$ 
(indeed, $\Gamma(\frac{d+2-2\alpha}{4})[\Gamma(\frac{d-1+\alpha}{4})]^2 - \Gamma(\frac{d-2+2\alpha}{4})[\Gamma(\frac{d+1-\alpha}{4})]^2>0$), i.e.\,these assumptions are less restrictive than the ones in the proof of Proposition \ref{prop_reg}. 

Then, in particular, 
$$
\|e^{-t\Lambda}f\|_{q} \leq c_rt^{-j'\left(\frac{1}{r}-\frac{1}{q}\right)}\|f\|_r, \quad f \in L^r \cap L^q, \quad 2 \leq r < q \leq \infty
$$
(arguing as in the proof of \cite[Theorem 4.3]{KiS1}).
\end{remark}

\bigskip

\section{Operator realizations of $(-\Delta)^{\frac{\alpha}{2}} + b \cdot \nabla$ in $L^r$, $r \in ]r_c,\infty[$, and in $L^1_{\scriptscriptstyle \sqrt{\varphi}}$}

\label{appA}

We are using notation introduced in the beginning of Section \ref{rellich_app}.

\begin{proposition}
\label{prop4}
Let $0< \delta<4$. The following is true:

{\rm(\textit{i})} There exists an operator realization $\Lambda_{r}(b)$ of $(-\Delta)^{\frac{\alpha}{2}} + b \cdot \nabla$ in $L^r$, $r \in ]r_c,\infty[$, $r_c=\frac{2}{2-\sqrt{\delta}}$, as the (minus) generator of a contraction $C_0$ semigroup 
$$
e^{-t\Lambda_r(b)}:=s\mbox{-}L^r\mbox{-}\lim_{n} e^{-tP^{\varepsilon_n}} \quad (\text{loc.\,uniformly in $t \geq 0$}), \quad r \in ]r_c,\infty[,
$$
for a sequence $\{\varepsilon_n\} \downarrow 0$.

{\rm(\textit{ii})} There exists an operator realization $\Lambda_{\scriptsize{1,\sqrt{\varphi}}}(b)$ of $(-\Delta)^{\frac{\alpha}{2}} + b \cdot \nabla$ in $L^1_{\scriptscriptstyle \sqrt{\varphi}}$, $\varphi \equiv \varphi_s$, as the (minus) generator of a quasi contraction $C_0$ semigroup
$$
e^{-t\Lambda_{\scriptsize{1,\sqrt{\varphi}}}(b)}:=s\mbox{-}L^1_{\scriptsize{\scriptscriptstyle \sqrt{\varphi}}}\mbox{-}\lim_{n} e^{-tP^{\varepsilon_n}} \quad (\text{loc.\,uniformly in $t \geq 0$})
$$
for a sequence $\{\varepsilon_n\} \downarrow 0$.

\smallskip

The semigroups in {\rm(\textit{i})}, {\rm(\textit{ii})} are consistent: $e^{-t\Lambda_r(b)} \upharpoonright L^r \cap L^1_{\scriptscriptstyle \sqrt{\varphi}}=e^{-t\Lambda_{\scriptsize{1,\sqrt{\varphi}}}(b)} \upharpoonright L^r \cap L^1_{\scriptscriptstyle \sqrt{\varphi}}$.

\smallskip

{\rm(\textit{iii})} For every $u \in D(\Lambda_r(b))$, $r \in ]r_c,\infty[$,
$$
\langle \Lambda_r(b)u,h\rangle=\langle u, (-\Delta)^\frac{\alpha}{2}h\rangle - \langle u,b\cdot \nabla h\rangle - \langle u, ({\rm div\,}b)h\rangle, \qquad h \in C_c^\infty.
$$

\end{proposition}


\begin{proof}[Proof of Proposition \ref{prop4}]
\textit{Proof of {\rm(\textit{i}), (\textit{ii})}.}

\smallskip

1.~Set $0 \leq v \equiv v_\varepsilon:=(\mu+P^\varepsilon)^{-1}f$, $\mu>0$, $f \in L_+^r$. 

Multiplying the equation $(\mu+P^\varepsilon)v=f$ by $v^{r-1}$, integrating, and then arguing as in the proof of ($S_1$) in Section \ref{proof_thm2_sect}, we obtain, for $r \in ]r_c,\infty[$,
\begin{equation*}
\mu\|v\|_r^r + \frac{2}{r}\left(\frac{2}{r'}-\sqrt{\delta}\right)\|A^{\frac{1}{2}}v^{\frac{r}{2}}\|_2^2 \leq \|f\|_r \|v\|_r^{r-1},
\end{equation*}
where $\frac{2}{r'}-\sqrt{\delta}>0$ ($\Leftrightarrow r \in ]r_c,\infty[$), and so
\begin{equation}
\label{r_est}
\tag{$\ast$}
\mu^{r-1}\frac{2}{r}\left(\frac{2}{r'}-\sqrt{\delta}\right)\|A^{\frac{1}{2}}v^{\frac{r}{2}}\|_2^2 \leq\|f\|^r_r,
\end{equation}
\begin{equation}
\label{r_est2}
\tag{$\ast\ast$}
\|\mu(\mu+P^\varepsilon)^{-1}f\|_r \leq \|f\|_r.
\end{equation}
Since $(\mu+P^\varepsilon)^{-1}$ preserves positivity, \eqref{r_est2} is valid for all $f \in L^r$.

Now, let $0 \leq f \in L^1_{\sqrt{\varphi}}$. Multiplying $(\mu+P^\varepsilon)v=f$ by $\varphi$ and integrating, we have
$$
\mu\|v\|_{1,\sqrt{\varphi}} + \big\langle P^\varepsilon v, \varphi  \big\rangle = \|f\|_{1,\sqrt{\varphi}}.
$$
Using $(P^\varepsilon)^*\varphi \geq -cs^{-1}$
(see \eqref{P} in the proof of Proposition \ref{main_prop4}), we obtain
$$
\mu\|v\|_{1,\sqrt{\varphi}}-cs^{-1}\|v\|_1 \leq \|f\|_{1,\sqrt{\varphi}}.
$$
Since $\varphi \geq \frac{1}{2}$, $\|v\|_1 \leq 2 \|v\|_{1,\sqrt{\varphi}}$,
$$
\|v\|_{1,\sqrt{\varphi}} \leq (\mu-2cs^{-1})^{-1} \|f\|_{1,\sqrt{\varphi}},
$$
i.e.
\begin{equation*}
\|\varphi(\mu+P^\varepsilon)^{-1}\varphi^{-1} h\|_{1} \leq (\mu-2cs^{-1})^{-1}\|h\|_{1}, \quad h \in L^1_+.
\end{equation*}
Since $(\mu+P^\varepsilon)^{-1}$ preserves positivity, we have
\begin{equation}
\label{tri_ast}
\tag{$\ast\ast\ast$}
\|\varphi(\mu+P^\varepsilon)^{-1}\varphi^{-1} \|_{1 \rightarrow 1} \leq (\mu-2cs^{-1})^{-1}.
\end{equation}

\smallskip

2.~We will need

\begin{claim}
\label{claim_t1}
{\rm(1)} $\mu(\mu+P^\varepsilon)^{-1} \rightarrow 1$ strongly in $L^1_{\scriptscriptstyle \sqrt{\varphi}}$ as $\mu \uparrow \infty$ uniformly in $\varepsilon>0$.

{\rm(2)} $\mu(\mu+P^\varepsilon)^{-1} \rightarrow 1$ strongly in $L^r$ as $\mu \uparrow \infty$ uniformly in $\varepsilon>0$, for every $r \in ]r_c,\infty[$.
\end{claim}
\begin{proof}[Proof of Claim \ref{claim_t1}]
Proof of Claim \ref{claim_t1}(1). In view of \eqref{tri_ast}, it suffices to prove the convergence only for $f \in C_c^\infty$. We write
$$
(\mu+P^\varepsilon)^{-1}f=(\mu+A)^{-1}f-(\mu+P^\varepsilon)^{-1}(b_\varepsilon \cdot \nabla + U_\varepsilon)(\mu+A)^{-1}f.
$$
We have $\mu(\mu+A)^{-1}f \rightarrow f$ strongly in $L^1_{\scriptscriptstyle \sqrt{\varphi}}$ as $\mu \uparrow \infty$. Indeed,
\begin{align*}
\big\langle \varphi|\mu(\mu+A)^{-1}f-f|\big\rangle & \leq \big\langle \mathbf{1}_{B(0,s^{\frac{1}{\alpha}})}\varphi|\mu(\mu+A)^{-1}f-f|\big\rangle  \\
& + \big\langle \mathbf{1}_{B^c(0,s^{\frac{1}{\alpha}})}|\mu(\mu+A)^{-1}f-f|\big\rangle \rightarrow 0,
\end{align*}
where the first term tends to $0$ since $\mu(\mu+A)^{-1}f \rightarrow f$ strongly in $C_u$, while the 
second term tends to $0$ since $\mu(\mu+A)^{-1}f \rightarrow f$ strongly in $L^1$.

Thus, it remains to show that $\|\mu (\mu+P^\varepsilon)^{-1}(b_\varepsilon \cdot \nabla + U_\varepsilon)(\mu+A)^{-1}f\|_{1,\sqrt{\varphi}} \rightarrow 0$ as $\mu \uparrow \infty$ uniformly in $\varepsilon>0$.
First, we consider $I_{1,\varepsilon}:=\|\mu (\mu+P^\varepsilon)^{-1} b_\varepsilon \cdot \nabla(\mu+A)^{-1}f\|_{1,\sqrt{\varphi}}$\,:
\begin{align*}
I_{1,\varepsilon} & \leq \|(\mu+P^\varepsilon)^{-1}\mathbf{1}_{B(0,s^{\frac{1}{\alpha}})}|b_\varepsilon|\|_{1,\sqrt{\varphi}}\|\mu(\mu+A)^{-1}|\nabla f|\|_\infty \\
& + \|(\mu+P^\varepsilon)^{-1}\mathbf{1}_{B^c(0,s^{\frac{1}{\alpha}})}C_1 \mu(\mu+A)^{-1}|\nabla f|\|_{1,\sqrt{\varphi}} \qquad (C_1:=\kappa s^{-\frac{\alpha-1}{\alpha}}<\infty) \\
&=:J_1 + C_1 J_2.
\end{align*}
We have
\begin{align*}
J_1 &\leq \|\varphi(\mu+P^\varepsilon)^{-1}\varphi^{-1}\|_{1 \rightarrow 1} \|\varphi \mathbf{1}_{B(0,s^{\frac{1}{\alpha}})}|b|\|_1\|\nabla f\|_\infty \\
& (\text{we are applying \eqref{tri_ast}}) \\
&\leq  (\mu-2cs^{-1})^{-1} \|\varphi \mathbf{1}_{B(0,s^{\frac{1}{\alpha}})}|b|\|_1 \|\nabla f\|_\infty<\infty.
\end{align*}
(Indeed, $\varphi \mathbf{1}_{B(0,s^{\frac{1}{\alpha}})}|b| \in L^1$ since on $B(0,s^{\frac{1}{\alpha}})$ $\varphi |b|=\kappa|s^{-\frac{1}{\alpha}}x|^{-d+\beta}|x|^{-\alpha+1}$ and $\beta>\alpha$.)
\begin{align*}
J_2 &\leq \|\varphi (\mu+P^\varepsilon)^{-1} \varphi^{-1}\|_{1 \rightarrow 1} \|\varphi \mathbf{1}_{B^c(0,s^{\frac{1}{\alpha}})}\mu(\mu+A)^{-1}|\nabla f|\|_1 \\
& (\text{we are applying \eqref{tri_ast}}) \\
& \leq (\mu-2cs^{-1})^{-1} \|\varphi \mathbf{1}_{B^c(0,s^{\frac{1}{\alpha}})}\|_\infty \|\nabla f\|_1.
\end{align*}
Thus, $I_{1,\varepsilon} \rightarrow 0$ as $\mu \uparrow \infty$ uniformly in $\varepsilon>0$.

Similar argument shows that $I_{2,\varepsilon}:=\|\mu (\mu+P^\varepsilon)^{-1} U_\varepsilon (\mu+A)^{-1}f\|_{1,\sqrt{\varphi}} \rightarrow 0$ as $\mu \uparrow \infty$ uniformly in $\varepsilon>0$.
We only note that $\|\mathbf{1}_{B(0,s^{\frac{1}{\alpha}})}\varphi U_\varepsilon\|_1 \leq C<\infty$, $C \neq C(\varepsilon)$ due to
$$
\mathbf{1}_{B(0,s^{\frac{1}{\alpha}})}\varphi U_\varepsilon \leq \alpha \kappa s^{\frac{d-\beta}{\alpha}}|x|^{-d+\beta-\alpha}.
$$


The proof of Claim \ref{claim_t1}(1) is completed.

\smallskip

Proof of Claim \ref{claim_t1}(2). In view of \eqref{r_est2}, it suffices to prove the convergence in Claim \ref{claim_t1}(2) for every $f \in C_c^\infty$. The latter follows from $\varphi \geq \frac{1}{2}$ and Claim \ref{claim_t1}(1):
\begin{align*}
\|\mu(\mu+P^\varepsilon)^{-1}f-f\|_r^r & \leq \|\mu(\mu+P^\varepsilon)^{-1}f-f\|_1 \|\mu(\mu+P^\varepsilon)^{-1}f-f\|_\infty^{r-1}\\
& \leq 2\|\mu(\mu+P^\varepsilon)^{-1}f-f\|_{1,\sqrt{\varphi}}(2\|f\|_\infty)^{r-1} \rightarrow 0
\end{align*}
as $\mu \rightarrow \infty$ uniformly in $\varepsilon>0$.

The proof of Claim \ref{claim_t1} is completed.
\end{proof}

\begin{claim}
\label{claim_t2}
Then there exists a sequence $\varepsilon_k \downarrow 0$ such that the following is true:

{\rm(1)} $\|(\mu+P^{\varepsilon_n})^{-1}f-(\mu+P^{\varepsilon_k})^{-1}f\|_{1,\sqrt{\varphi}} \rightarrow 0$ as $n,k \rightarrow \infty$ for every $f \in L^1_{\scriptscriptstyle \sqrt{\varphi}}$, $\mu>cs^{-1}$

{\rm(2)} $\|(\mu+P^{\varepsilon_n})^{-1}f-(\mu+P^{\varepsilon_k})^{-1}f\|_r \rightarrow 0$ as $n,k \rightarrow \infty$ for every $f \in L^r$, $\mu>0$, $r \in ]r_c,\infty[$.

\end{claim}
\begin{proof}[Proof of Claim \ref{claim_t2}]

By \eqref{r_est2}, \eqref{tri_ast}, it suffices to prove (1) and (2) for every $f$ from a countable subset $F$ of $C_c^\infty$-functions  such that $F$ is dense in $L^1_{\scriptscriptstyle \sqrt{\varphi}}$ and $L^r$, respectively.

Without loss of generality, we prove (1), (2) for $0 \leq f \in F$. Set $0 \leq v_n:=(\mu+P^{\varepsilon_n})^{-1}f$, for some $\varepsilon_n \downarrow 0$.

\medskip

Proof of Claim \ref{claim_t2}(2). We emphasize that, according to Remark \ref{rem_conv_2}, it suffices to carry out the proof for some $r \in ]r_c,\infty[$. 

First, we establish the inequality
\begin{equation}
\label{vub}
v_n(x) \leq C_f |x|^{-d-\alpha} \quad \text{ $\forall\,x \in B^c(0,R)$}
\end{equation}
which holds all $R$ sufficiently large,
with constant $C_f$ depending on $\|f\|_\infty$ and its compact support. Indeed,
let $\sprt f \subset B(0,R_f)$, $R \geq 2R_f \wedge 1$.
Using the upper bound $e^{-tP^\varepsilon}(x,y) \leq Ce^{-t(-\Delta)^{\frac{\alpha}{2}}}\varphi_t(y)$, $C \neq C(\varepsilon)$, and recalling that $e^{-t(-\Delta)^\frac{\alpha}{2}}(x,y) \approx t^{-\frac{d}{\alpha}}\wedge\frac{t}{|x-y|^{d+\alpha}}$ and $\varphi_t(y) \approx t^{\frac{d-\beta}{\alpha}}|y|^{-d+\beta} \vee 1$, we have
($x \in B^c(0,R)$, $y \in B(0,R_f)$)
\begin{align*}
(\mu+P^{\varepsilon})^{-1}(x,y) & \leq C_1 \int_0^\infty e^{-\mu t} \big(t^{-\frac{d}{\alpha}}\wedge\frac{t}{|x-y|^{d+\alpha}} \big) \big(t^{\frac{d-\beta}{\alpha}}|y|^{-d+\beta} \vee 1 \big)dt \\
& \leq C_2(R_f) \biggl[\int_{0}^{|y|^\alpha} e^{-\mu t} t|x|^{-d-\alpha}dt + \int_{|y|^\alpha}^{|x|^\alpha} e^{-\mu t} t|x|^{-d-\alpha}t^{\frac{d-\beta}{\alpha}}|y|^{-d+\beta} dt \\
& + \int_{|x|^\alpha}^\infty e^{-\mu t} t^{-\frac{d}{\alpha}}t^{\frac{d-\beta}{\alpha}}|y|^{-d+\beta} dt \bigg] \\
& \leq C_3(\mu,R_f) |x|^{-d-\alpha}\bigg[|y|^{2\alpha} + e^{-\frac{\mu}{2}|x|^\alpha}|y|^{-d+\beta} + |y|^{-d+\beta}\bigg],
\end{align*}
where we have used inequalities $e^{-\mu t} \leq 1$ in the first integral, $e^{-\mu t} t^{1+\frac{d-\beta}{\alpha}} \leq c_1(\mu) e^{-\frac{\mu}{2}t}$ in the second integral, and $e^{-\mu t} t^{-\frac{\beta}{\alpha}} \leq c_2(\mu) t^{-\frac{d-\alpha}{\alpha}-1}$ (for $t \geq |x|^\alpha \geq 1$) in the third integral. 
Therefore, for every $x \in B^c(0,R)$,
$
v_n(x)=\langle (\mu+P^{\varepsilon_n})^{-1}(x,\cdot)\mathbf{1}_{B(0,R_f)}(\cdot)f(\cdot)\rangle$ satisfies \eqref{vub} with constant $C_f:=C_3(\mu,R_f) \langle (|\cdot|^{2\alpha} + 2|\cdot|^{-d+\beta})\mathbf{1}_{B(0,R_f)}\rangle \|f\|_\infty <\infty
$ (recall $\beta \in ]\alpha,d[$).

%

By \eqref{vub},
\begin{equation*}
\|\mathbf{1}_{B^c(0,R)}v_n\|_r \downarrow 0 \quad \text{ as } R \uparrow \infty \text{ uniformly in $n$},
\end{equation*}
and so, for a given $\gamma>0$, there exists $R=R_\gamma$ such that
\begin{equation}
\label{c_diff}
\tag{$\bullet$}
\|\mathbf{1}_{B^c(0,R)}(v_n-v_k)\|_r < \gamma \quad \text{ for all } n,k,
\end{equation}
i.e.\,we can control $\{v_n-v_k\}$ outside of the finite ball $B(0,R)$.

 By \eqref{r_est},\eqref{r_est2}
\begin{equation*}
\|v_n^\frac{r}{2}\|_2^2+\|A^{\frac{1}{2}}v_n^{\frac{r}{2}}\|_2^2 \leq c\|f\|_r^r, \quad \text{ for all } n, \quad c \neq c(n).
\end{equation*}
Therefore, by the fractional Rellich-Kondrashov Theorem, there are a subsequence $\{v_{n_l}\}$ and a function $v \geq 0$ such that $v^{r/2} \in L^2(B(0,R))$,
\begin{align*}
 \|\mathbf{1}_{B(0,R)}v^{r/2}_{n_l}\|_2 &\rightarrow \|\mathbf{1}_{B(0,R)}v^{r/2}\|_2, \\
v_{n_l} &\rightarrow v \quad \text{a.e. on $B(0,R)$}.
\end{align*}
The latter yields that $v_{n_l} \rightarrow v$ strongly in $L^r(B(0,R))$. Thus,
\begin{equation}
\label{c_diff2}
\tag{$\bullet\bullet$}
\{\mathbf{1}_{B(0,R)}v_{n_l}\} \text{ is a Cauchy sequence in $L^r$.}
\end{equation}

Now, we consider \eqref{c_diff} with $\gamma=\gamma_m \downarrow 0$. Combining  it with \eqref{c_diff2} we  obtain a subsequence of $\{v_{n}\}$ (which we again denote by $\{v_{n_m}\}$) that constitutes a Cauchy sequence in $L^r$. 

A priori, the choice of $\{v_{n_m}=(\mu+P^{\varepsilon_{n_m}})^{-1}f\}$ depends on $0 \leq f \in F$. Since $F$ is countable, we can use the standard diagonal argument to pass to a subsequence $\{\varepsilon_{n_{m_k}}\}$ such that $\{(\mu+P^{\varepsilon_{n_{m_k}}})^{-1}f\}$ is a Cauchy sequence for every $0 \leq f \in F$. Thus, re-denoting $\{\varepsilon_{n_{m_k}}\}$ as $\{\varepsilon_k\}$, we arrive at Claim \ref{claim_t2}(2).

\smallskip

Proof of Claim \ref{claim_t2}(1). Let us show that $\{v_n\}$ is a Cauchy sequence in $L^1_{\scriptscriptstyle \sqrt{\varphi}}$.
We have
$$
\|v_n-v_k\|_{1,\sqrt{\varphi}} \leq \|\mathbf{1}_{B^c(0,R)}(v_n-v_k)\|_1 + \|\mathbf{1}_{B(0,R)}\varphi(v_n-v_k)\|_1.
$$
The first term in the RHS can be made arbitrarily small, by choosing $R$ sufficiently large, in view of \eqref{vub}.
The second term in the RHS converges to $0$ since $\mathbf{1}_{B(0,R)}\varphi \in L^{1+\sigma}$ for a small $\sigma>0$ such that $r:=\frac{1+\sigma}{\sigma}>r_c$, while $v_n-v_k \rightarrow 0$ strongly in $L^{r}(B(0,R))$ according to  Claim \ref{claim_t2}(2).

The proof of Claim \ref{claim_t2} is completed.
\end{proof}

By \eqref{r_est2}, \eqref{tri_ast}, the $C_0$ semigroups $e^{-tP^\varepsilon}$ are contractions in $L^r$ and quasi contractions in $L^1_{\scriptscriptstyle \sqrt{\varphi}}$. Thus, Claim \ref{claim_t1} and Claim \ref{claim_t2} verify the hypothesis of the Trotter Approximation Theorem, and so Proposition \ref{prop4}(\textit{i}), (\textit{ii}) is proved.

\smallskip

\textit{Proof of {\rm(\textit{iii})}.} Set $v_n:=(\mu+P^{\varepsilon_n})^{-1}f$ and $v=(\mu+\Lambda_r(b))^{-1}f$, $f \in L^r$. We have
$$
\langle f,h\rangle=\mu\langle v_n,h \rangle + \langle v_n, (-\Delta)^{\frac{\alpha}{2}}h\rangle - \langle v_n,b_{\varepsilon_n}\cdot \nabla h\rangle - \langle v_n, ({\rm div\,}b_{\varepsilon_n})h\rangle, \qquad h \in C_c^\infty.
$$
Since for $\sigma>0$ sufficiently small, $b_{\varepsilon_n} \rightarrow b$, ${\rm div\,}b_{\varepsilon_n} \rightarrow {\rm div\,}b$ in $L^{1+\sigma}_{\loc}$, while  $v_n \rightarrow v$ strongly in $L^r$ for $r:=\frac{1+\sigma}{\sigma}$, we have
$$
\langle f,h\rangle=\mu\langle v,h \rangle + \langle v, (-\Delta)^{\frac{\alpha}{2}}h\rangle - \langle v,b\cdot \nabla h\rangle - \langle v, ({\rm div\,}b)h\rangle.
$$

The proof of Proposition \ref{prop4} is completed. 
\end{proof}

\bigskip

\section{Operator realization of $(-\Delta)^{\frac{\alpha}{2}} -  \nabla \cdot b$ in $L^{r'}$, $r' \in ]1,r'_c[$}

\label{appA_}

We are using notation introduced in the beginning of Section \ref{rellich_app}.

\begin{proposition}
\label{prop5}
Let $0< \delta<4$.  There exists an operator realization $\Lambda^*_{r'}(b)$ of $(-\Delta)^{\frac{\alpha}{2}} -  \nabla \cdot b$ in $L^{r'}$, $r' \in ]1,r'_c[$, $r_c=\frac{2}{2-\sqrt{\delta}}$, as the (minus) generator of a contraction $C_0$ semigroup
$$
e^{-t\Lambda^*_{r'}(b)}=s\mbox{-}L^{r'}\mbox{-}\lim_{n} e^{-t(P^{\varepsilon_n})^*} \quad (\text{loc.\,uniformly in $t \geq 0$}), \quad r' \in ]1,r'_c[,
$$
for a sequence $\{\varepsilon_n\} \downarrow 0$.

We have
$$
\langle e^{-t\Lambda_{r}(b)}f,g\rangle=\langle f,e^{-t\Lambda^*_{r'}(b)}g\rangle, \quad t>0, \quad f \in L^{r}, \quad g \in L^{r'},
$$
where $\Lambda_{r}(b)$ has been constructed in Proposition \ref{prop4}.
\end{proposition}

\begin{proof}[Proof of Proposition \ref{prop5}]
Let $0 \leq f \in C_c^\infty$, set $0 \leq v:=(\mu+(P^\varepsilon)^*)^{-1}f$. We multiply the equation $(\mu+(-\Delta)^{\frac{\alpha}{2}}-\nabla \cdot b_\varepsilon + U_\varepsilon)v=f$ by $v^{r'-1}$, integrate, and then argue as in the proof of ($S_1$) in Section \ref{proof_thm2_sect}, to obtain
$$
\mu\|v\|_{r'}^{r'} + \frac{4}{rr'}\|A^{\frac{1}{2}}v^{\frac{r'}{2}}\|_2^2 + \langle (-\nabla \cdot b_\varepsilon + U_\varepsilon)v,v^{r'-1}\rangle \leq \|f\|_{r'}\|v\|_{r'}^{r'-1}.
$$
We write $-\nabla \cdot b_\varepsilon v=-({\rm div\,}b_{\varepsilon})v - b_\varepsilon \cdot \nabla v$ and 
use that $-{\rm div\,}b_{\varepsilon}=-\kappa(d-\alpha)|x|_\varepsilon^{-\alpha} - U_\varepsilon$
to obtain
$$
\mu\|v\|_{r'}^{r'} + \frac{4}{rr'}\|A^{\frac{1}{2}}v^{\frac{r'}{2}}\|_2^2 - \kappa(d-\alpha) \langle |x|_\varepsilon^{-\alpha}v^{\frac{r'}{2}},v^{\frac{r'}{2}} \rangle - \langle b_\varepsilon \cdot \nabla v,v^{r'-1}\rangle  \leq \|f\|_{r'}\|v\|_{r'}^{r'-1}.
$$
By  integration by parts, 
\begin{align*}
&-\langle b_\varepsilon \cdot \nabla v,v^{r'-1}\rangle \equiv -\frac{2}{r'}\kappa\langle |x|_\varepsilon^{-\alpha}x \cdot \nabla v^{\frac{r'}{2}},v^{\frac{r'}{2}}\rangle \\
& =\frac{2}{r'}\biggl(\kappa \frac{d-\alpha}{2} \langle |x|_\varepsilon^{-\alpha}v^{\frac{r'}{2}},v^{\frac{r'}{2}}\rangle+\frac{1}{2}\langle U_\varepsilon v^{\frac{r'}{2}},v^{\frac{r'}{2}}\rangle \biggr).
\end{align*}
Thus,
$$
\mu\|v\|_{r'}^{r'} + \frac{4}{rr'}\|A^{\frac{1}{2}}v^{\frac{r'}{2}}\|_2^2 - \kappa(d-\alpha)\frac{1}{r} \langle |x|_\varepsilon^{-\alpha}v^{\frac{r'}{2}},v^{\frac{r'}{2}} \rangle \leq \|f\|_{r'}\|v\|_{r'}^{r'-1}.
$$
so by the Hardy-Rellich inequality \eqref{hardy_ineq},
$$
\mu\|v\|_{r'}^{r'} + \frac{2}{r}\left( \frac{2}{r'}-\sqrt{\delta}\right)\|A^{\frac{1}{2}}v^{\frac{r'}{2}}\|_2^2  \leq \|f\|_{r'}\|v\|_{r'}^{r'-1},
$$
where $\frac{2}{r'}-\sqrt{\delta}>0$. 
Thus, we have proved
\begin{equation}
\label{mu_est2}
\tag{$\ast$}
\mu^{r'-1}\frac{2}{r}\left( \frac{2}{r'}-\sqrt{\delta}\right)\|A^{\frac{1}{2}}v^{\frac{r'}{2}}\|_2^2   \leq \|f\|_{r'}, \qquad r' \in ]1,r_c'[, 
\end{equation}
\begin{equation}
\label{mu_est}
\tag{$\ast\ast$}
\|\mu(\mu+(P^\varepsilon)^*)^{-1}\|_{r' \rightarrow r'}  \leq 1, \qquad r' \in ]1,r_c'[, 
\end{equation}
\begin{equation}
\label{mu_est3}
\tag{$\ast\ast\ast$}
\|\mu(\mu+(P^\varepsilon)^*)^{-1}\|_{1 \rightarrow 1}  \leq 1.
\end{equation}
(\eqref{mu_est3} evidently follows from \eqref{mu_est} by taking $r' \downarrow 1$.)

\begin{claim}
\label{claim3}
$\mu(\mu+(P^\varepsilon)^*)^{-1} \rightarrow 1$ strongly in $L^{r'}$ as $\mu \rightarrow \infty$ uniformly in $\varepsilon$, for every $r' \in [1,r_c'[$.
\end{claim}
\begin{proof}[Proof of Claim \ref{claim3}]
1.~First, we consider the case $r'=1$. By \eqref{mu_est3}, it suffices to prove the convergence only on $f \in C_c^\infty$. We write
\begin{align*}
\mu(\mu+(P^\varepsilon)^*)^{-1}f =\mu(\mu+A)^{-1}f &- \mu(\mu+(P^\varepsilon)^*)^{-1}(-\nabla \cdot b_\varepsilon + U_\varepsilon)(\mu+A)^{-1}f \\
=\mu(\mu+A)^{-1}f & + \mu(\mu+(P^\varepsilon)^*)^{-1}b_\varepsilon \cdot \nabla(\mu+A)^{-1}f \\
& + \mu(\mu+(P^\varepsilon)^*)^{-1}({\rm div\,}b_\varepsilon)(\mu+A)^{-1}f \\
& + \mu(\mu+(P^\varepsilon)^*)^{-1}(-U_\varepsilon)(\mu+A)^{-1}f\\
=:\mu(\mu+A)^{-1}f&+K_1+K_2+K_3.
\end{align*}
Since $\mu(\mu+A)^{-1}f \rightarrow f$ in $L^1$ as $\mu \rightarrow \infty$, we have to show that $K_i \downarrow 0$, $i=1,2,3$, uniformly in $\varepsilon>0$.
Indeed, we have
\begin{align*}
\|K_1\|_1 & \leq \langle(\mu+(P^\varepsilon)^*)^{-1}\mathbf{1}_{B(0,1)}|b|\mu(\mu+A)^{-1}|\nabla f|\rangle \\
& +\langle\mu(\mu+(P^\varepsilon)^*)^{-1}C_1(\mu+A)^{-1}|\nabla f|\rangle, \qquad C_1:=\|\mathbf{1}_{B^c(0,1)}|b|\|_\infty \\
& \leq \|\mathbf{1}_{B(0,1)}|b|\|_1 \mu^{-1} \|\nabla f\|_\infty + C_1\mu^{-1}\|\nabla f\|_1.
\end{align*}
The terms $\|K_2\|_1$ and $\|K_3\|_1$ are estimated similarly (we only note that  $\|\mathbf{1}_{B(0,1)}U_\varepsilon\|_1 \leq C$, $C \neq C(\varepsilon)$, cf.\,the proof of Claim \ref{claim_t1}). Thus, $K_i \downarrow 0$, $i=1,2,3$, uniformly in $\varepsilon>0$, as needed. 

(Clearly, the proof above holds also for any $r'<\frac{d}{d-\beta}\wedge r_c'$.)

The convergence in the general case $r' \in [1,r_c'[$ follows by \eqref{mu_est}, the convergence in $L^1$ and H\"{o}lder's inequality.

The proof of Claim \ref{claim3} is completed.
\end{proof}

\begin{claim}
\label{claim4}
Fix $\mu>0$, $r' \in ]1,r_c'[$. Then there exists a sequence $\varepsilon_m \downarrow 0$ such that
$$
\|(\mu+(P^{\varepsilon_m})^*)^{-1}f-(\mu+(P^{\varepsilon_k})^*)^{-1}f\|_{r'} \rightarrow 0 \quad \text{ as $m,k \rightarrow \infty$}
$$ 
for every $f \in L^{r'}$.
\end{claim}

\begin{proof}[Proof of Claim \ref{claim4}]
We argue as in the proof of Claim \ref{claim_t2}. According to Remark \ref{rem_conv_2}, it suffices to carry out the proof for some $r' \in ]1,r_c'[$. 

We fix a countable subset $F$ of $C_c^\infty$ such that $F \cap \{f \geq 0\}$ is dense in $L^{r'} \cap \{f \geq 0\}$. It suffices to prove the convergence on  $0 \leq f \in F$. 

Set $0 \leq v_n:=(\mu+(P^{\varepsilon_n})^*)^{-1}f$ for some $\{\varepsilon_n\} \downarrow 0$.
Let us show that
\begin{equation}
\label{vub2}
v_n(x) \leq C_f |x|^{-d-\alpha} \quad \text{  $\forall\,x \in B^c(0,R)$}
\end{equation}
for all $R$ sufficiently large,
with constant $C_f=C(\|f\|_1, \sprt \,f)$. Indeed,
let $\sprt f \subset B(0,R_f)$, $R \geq 2R_f \wedge 1$.
Using the upper bound $e^{-t(P^\varepsilon)^*}(x,y) \leq Ce^{-t(-\Delta)^{\frac{\alpha}{2}}}(x,y)\varphi_t(x)$, $C \neq C(\varepsilon)$, and $e^{-t(-\Delta)^\frac{\alpha}{2}}(x,y) \approx t^{-\frac{d}{\alpha}}\wedge\frac{t}{|x-y|^{d+\alpha}}$, $\varphi_t(y) \approx t^{\frac{d-\beta}{\alpha}}|y|^{-d+\beta} \vee 1$, we have
($x \in B^c(0,R)$, $y \in B(0,R_f)$)
\begin{align*}
(\mu+(P^{\varepsilon})^*)^{-1}(x,y) & \leq C_1 \int_0^\infty e^{-\mu t} \bigl( t^{-\frac{d}{\alpha}}\wedge\frac{t}{|x-y|^{d+\alpha}}\bigr) \bigl(t^{\frac{d-\beta}{\alpha}}|x|^{-d+\beta} \vee 1\bigr)dt \\
& \leq C_2(R_f) \int_0^\infty e^{-\mu t} \bigl( t^{-\frac{d}{\alpha}}\wedge\frac{t}{|x|^{d+\alpha}}\bigr) \bigl(t^{\frac{d-\beta}{\alpha}}|x|^{-d+\beta} \vee 1\bigr)dt \\
& = C_2(R_f)\biggl[\int_0^{|x|^\alpha} e^{-\mu t}t|x|^{-d-\alpha}dt + \int_{|x|^\alpha}^\infty e^{-\mu t}t^{-\frac{d}{\alpha}}t^{\frac{d-\beta}{\alpha}}|x|^{-d+\beta} dt \biggr] \\
& \leq C_3(\mu,R_f) \biggl[|x|^{-d-\alpha}e^{-\frac{\mu}{2}|x|^\alpha} + |x|^{-d+\beta}|x|^{-\beta - \alpha}\biggr],
\end{align*}
where in the first integral we used $e^{-\mu t}t \leq c_1(\mu)e^{-\frac{\mu}{2}t}$, in the second integral $e^{-\mu t}t^{-\frac{\beta}{\alpha}} \leq c_2(\mu) t^{-\frac{\beta+\alpha}{\alpha}-1}$ for $t \geq |x|^\alpha \geq 1$. Thus, for all $x \in B^c(0,R)$,
$
v_n(x)=\langle (\mu+(P^{\varepsilon_n})^*)^{-1}(x,\cdot)f(\cdot)\rangle$ satisfies \eqref{vub2} with constant $C_f:=C_3(\mu,R_f) \|f\|_1 <\infty.
$

Using \eqref{vub2}, we obtain that, 
for a given $\gamma>0$, there exists $R=R_\gamma$ such that
\begin{equation}
\label{c_diff_}
\tag{$\bullet$}
\|\mathbf{1}_{B^c(0,R)}(v_n-v_k)\|_{r'} < \gamma \quad \text{ for all } n,k.
\end{equation}

By \eqref{mu_est2}, since $r'>1$, we appeal to the fractional Rellich-Kondrashov Theorem as in the proof of Claim \ref{claim_t2}  obtaining a subsequence $\{v_{n_l}\}$ such that
\begin{equation}
\label{c_diff2_}
\tag{$\bullet\bullet$}
\{\mathbf{1}_{B(0,R)}v_{n_l}\} \text{ is a Cauchy sequence in $L^{r'}$.}
\end{equation}

Combining \eqref{c_diff_} and \eqref{c_diff2_} as in the proof of Claim \ref{claim_t2}, we obtain a subsequence $\{\varepsilon_{n_k}\}$ of $\{\varepsilon_n\}$ such that $\{(\mu+(P^{\varepsilon_{n_k}})^*)^{-1}f\}$ constitutes a Cauchy sequence in $L^{r'}$ for every $0 \leq f \in F$. Re-denoting ${\varepsilon_{n_k}}$ by $\{\varepsilon_k\}$, we complete the proof of Claim \ref{claim4}.
\end{proof}

By \eqref{mu_est}, $e^{-t(P^\varepsilon)^*}$ is a contraction semigroup in $L^{r'}$. Claims \ref{claim3} and \ref{claim4} now verify the conditions of the Trotter Approximation Theorem, which thus yields the first assertion of Proposition \ref{prop5}. The second assertion follows from the first one and Proposition \ref{prop4}.

The proof of Proposition \ref{prop5} is completed.
\end{proof}

\appendix

\bigskip

\section{Heat kernel of the fractional Schr\"{o}dinger operator}
\label{app_MS}

The arguments from the previous sections can be specified to the operator
$$
(-\Delta)^{\frac{\alpha}{2}} - V, \quad
V(x)=\delta c^{2}_{d,\alpha}|x|^{-\alpha},\quad 0<\delta \leq 1, \quad 0<\alpha<2,$$
where, recall,
$$c_{d,\alpha}=\frac{\gamma(\frac{d}{2})}{\gamma(\frac{d}{2}-\frac{\alpha}{2})}, \quad \gamma(s):=\frac{2^s\pi^\frac{d}{2}\Gamma(\frac{s}{2})}{\Gamma(\frac{d}{2}-\frac{s}{2})},$$
to yield sharp two-sided weighted bounds on its fundamental solution.
These bounds are known, see \cite{BGJP}, where the authors use a different technique.

We will need the symmetric variant of the desingularization method developed in \cite{MS}.
Let $X$ be a locally compact topological space and $\mu$ a measure on $X$. Let $H$ be a non-negative selfadjoint operator in the (complex) Hilbert space $L^2=L^2(X,\mu)$ with the inner product $\langle f,g\rangle =\int_X f\bar{g} d \mu$. We assume that $H$ possesses the Sobolev embedding property:

There are constants $j>1$ and $c_S>0$ such that, for all $f\in D(H^\frac{1}{2})$,
\[
c_S \|f \|^2_{2j} \leq \|H^\frac{1}{2}f\|_2^2 \tag{$M_1$},
\]
there exists a family of real valued weights $\varphi=\{\varphi_s\}_{s>0}$ on $X$ such that, for all $s>0$,
\[
\varphi_s, \; 1/\varphi_s \in L^2_\loc(X,\mu) \tag{$M_2$}
\] 
and there exists a constant $c_1$ independent of $s$ such that, for all $0<t\leq s$,
\[
\|\varphi_s e^{-tH}\varphi_s^{-1} f\|_1 \leq c_1 \|f\|_1, \;\; f\in \mathcal
\varphi_s L^\infty_{com} (X,\mu). \tag{$M_3$}
\]

\begin{theorem}[{\cite[Theorem A]{MS}}]
\label{weightM1}

In addition to $(M_1)-(M_3)$ assume that
\[
\inf_{s>0, x\in X} |\varphi_s(x)| \geq c_0 >0. \tag{$M_4$}
\]
Then $e^{-tA}, t>0$ are integral operators, and there is a constant $C=C(j,c_s,c_1,c_0)$ such that, for all $t>0$ and $\mu$ a.e. $x,y \in X$,
\[
|e^{-tH}(x,y)|\leq Ct^{-j^\prime} |\varphi_t(x) \varphi_t(y)|, \;\; j^\prime=j/(j-1). 
\]
\end{theorem}

We now apply the previous theorem to the fractional Schr\"{o}dinger operator.
First, we consider the case $0<\delta<1$. The case $\delta=1$ requires few modifications, see below.

According to the Hardy-Rellich inequality 
$\| (-\Delta)^{\frac{\alpha}{4}}f\|_2^2 \geq c^{2}_{d,\alpha} \||x|^{-\frac{\alpha}{2}}f\|_2^2$ (see \eqref{hardy_ineq}),
the form difference $H=(-\Delta)^\frac{\alpha}{2}\dotminus V$ is well defined \cite[Ch.VI, sect 2.5]{Ka}.

Define $\beta$ by $\delta c^{2}_{d,\alpha}=\frac{\gamma(\beta)}{\gamma(\beta-\alpha)}$, and let $\varphi(x)\equiv \varphi_s(x)=\eta(s^{-\frac{1}{\alpha}}|x|)$, where $\eta \in C^2(\mathbb R - \{0\})$ is such that
\[
\eta(r)=\left \{
\begin{array}{ll}
r^{-d+\beta}, & 0<r< 1, \\
\frac{1}{2}, & r\geq 2.
\end{array}
\right.
\]

\begin{theorem}
\label{thm1_s} 
Under constraints $0<\delta<1$ and $0<\alpha<2$, $e^{-tH}$ is an integral operator for each $t>0$. The weighted Nash initial estimate 
\[
e^{-tH}(x,y)\leq c t^{-\frac{d}{\alpha}}\varphi_t(x)\varphi_t(y),\quad c=c_{d,\delta,\alpha},
\]
is valid for all $t>0$, $x,y\in \mathbb R^d-\{0\}$.
\end{theorem} 

The proof of Theorem \ref{thm1_s} follows the proof of Theorem \ref{thm2} (where we appeal to Theorem \ref{weightM1} instead of Theorem A), but is easier since operator $H$ is self-adjoint.
(Here and below, we refer to \cite{KiS2} for details.)

Having at hand Theorem \ref{thm1_s}, one can obtain the upper and lower bounds
\[
e^{-tH}(x,y)\approx e^{-t(-\Delta)^\frac{\alpha}{2}}(x,y)\varphi_t(x)\varphi_t(y). 
\]
The proof is similar to the proof of Theorem \ref{thm2_}, but is easier since operator $H$ is self-adjoint and we can appeal to the monotonicity arguments.

\medskip

\noindent\textbf{The case $\delta=1$.}
The following construction is standard: Define $H$ as the (minus) generator of a $C_0$ semigroup
$$
U^t:=s\mbox{-}L^2\mbox{-}\lim_{\varepsilon \downarrow 0}e^{-tH^{\varepsilon}},
$$
where $H^\varepsilon=(-\Delta)^{\frac{\alpha}{2}}- V_\varepsilon$, $V_\varepsilon(x)=\delta c^{2}_{d,\alpha}|x|_\varepsilon^{-2}$, $|x|_\varepsilon:=\sqrt{|x|^2+\varepsilon}$, $\varepsilon>0$.
(Indeed, set $u_\varepsilon(t)=e^{-tH^\varepsilon} f$, $f\in L^2_+.$ Since $\{V_\varepsilon\}$ is monotonically increasing as $\varepsilon \downarrow 0$, $0 \leq u_\varepsilon \uparrow u$ to some $u\geq 0$. Since $\|u_\varepsilon\|_2 \leq \|f\|_2$, we have
\[
\|u\|_2 \leq \|f\|_2, \quad \|u_\varepsilon \|_2 \uparrow \|u\|_2, \qquad  u_\varepsilon \overset{s}\rightarrow u=: U^t f, \quad U^{t+s}=U^tU^s, \quad U^tf\overset{s}\rightarrow f \text{ as } t\downarrow 0,
\]
$$
U^t \Real L^2 \subset \Real L^2, \quad U^t L^2_+ \subset L^2_+ \qquad (\Rightarrow |U^tf| \leq U^t|f| \text{ a.e. for all }f \in L^2),
$$
so, for $f \in L^2$, $\|U^t f\|_2 \leq \|U^t|f|\|_2 \leq \|f\|_2$.
By \cite[Corollary 2.5]{FLS}, 
$$
\langle (H+1) u, u \rangle \geq C_d\|u\|_{2j}^2,  \quad j \in [1,\frac{d}{d-\alpha}[,
$$  
so Theorem \ref{weightM1} yields the sub-optimal weighted Nash initial estimate $e^{-t(H+1)}(x,y) \leq Ct^{-j'}\varphi_t(x)\varphi_t(y)$ with $j'>\frac{d}{\alpha}$.
Now, we follow the observation in \cite{BGJP}: $e^{-H}(x,y) \leq eC\varphi_1(x)\varphi_1(y)$, so one can repeat the proof of the upper bound to obtain
$$
e^{-H}(x,y) \leq \tilde{C} e^{-(-\Delta)^{\frac{\alpha}{2}}}(x,y)\varphi_1(x)\varphi_1(y),
$$
and so, by scaling, 
$$
e^{-tH}(x,y) \leq \tilde{C} e^{-t(-\Delta)^{\frac{\alpha}{2}}}(x,y)\varphi_t(x)\varphi_t(y), \quad t>0.
$$
The proof of the lower bound remains the same.

\bigskip

\bigskip

\section{Lyapunov's function to $(-\Delta)^\frac{\alpha}{2}-V$}

\label{appA__}

Set $I_\alpha=(-\Delta)^{-\frac{\alpha}{2}}$, the Riesz potential defined by the formula
$$I_\alpha f(x):=\frac{1}{\gamma(\alpha)}\langle |x-\cdot|^{-d+\alpha}f(\cdot)\rangle, \quad \gamma(\alpha):=\frac{2^\alpha\pi^\frac{d}{2}\Gamma(\frac{\alpha}{2})}{\Gamma(\frac{d}{2}-\frac{\alpha}{2})}.$$ The identity 
\begin{equation*}
\frac{\gamma(\beta-\alpha)}{\gamma(\beta)}|x|^{-d+\beta} =I_\alpha |x|^{-d+\beta-\alpha}, \quad 0<\alpha<\beta<d,
\end{equation*}
follows e.g.\,from $I_{\beta}=I_\alpha I_{\beta-\alpha}$.

It follows that $\tilde{\varphi}_1(x)=|x|^{-d+\beta}$ is a Lyapunov's function to the formal operator $(-\Delta)^\frac{\alpha}{2}-V$:
$$
(-\Delta)^\frac{\alpha}{2}|x|^{-d+\beta}=V(x)|x|^{-d+\beta}, \quad V(x)=\frac{\gamma(\beta)}{\gamma(\beta-\alpha)}|x|^{-\alpha}.
$$

\bigskip

\section{The range of an accretive operator}

\label{appC}

Let $P$ be a closed operator on $L^1$ such that $\Real\langle(\lambda+ P)f,\frac{f}{|f|}\rangle\geq 0$ for all $f \in D(P)$,
and $R(\mu + P)$ is dense in $L^1$ for a $\mu>\lambda$. 

Then $R(\mu + P)=L^1$.

Indeed, let $y_n \in R(\mu + P)$, $n=1,2,\dots$, be a Cauchy sequence in $L^1$; $y_n=(\mu+P)x_n$, $x_n \in D(P)$. Write $[f,g] := \langle f,\frac{g}{|g|}\rangle$. Then
\begin{align*}
(\mu-\lambda)\|x_n-x_m\|_1&=(\mu-\lambda)[x_n-x_m,x_n-x_m] \\
&\leq (\mu-\lambda)[x_n-x_m,x_n-x_m] + [(\lambda+P)(x_n-x_m),x_n-x_m] \\
& = [(\mu+P)(x_n-x_m),x_n-x_m] \leq \|y_n-y_m\|_1.
\end{align*}
Thus, $\{x_n\}$ is itself a Cauchy sequence in $L^1$. Since $P$ is closed, the result follows.

%

\bigskip

\section{Weighted Nash initial estimate for $\alpha=2$}

\label{app}

For $\alpha=2$, the operator $(-\Delta)^{\frac{\alpha}{2}}+\kappa|x|^{-\alpha}x \cdot \nabla$ becomes $-\Delta + \sqrt{\delta} \frac{d-2}{2}|x|^{-2}x \cdot \nabla$.

\textit{Let $d \geq 3$, $0<\delta<4$. For every $r \in ]r_c,\infty[$, $r_c:=\frac{2}{2-\sqrt{\delta}}$ the limit}
\begin{equation}
\label{c}
\tag{$\ast$}
e^{-t\Lambda}=s{\mbox-}L^r{\mbox-}\lim_{\varepsilon \downarrow 0}e^{-t\Lambda^\varepsilon} \quad \text{(loc.\,uniformly in $t \geq 0$)}
\end{equation}
\textit{exists and determines a $C_0$ semigroup; here} $\Lambda^\varepsilon:=-\Delta + b_\varepsilon \cdot \nabla$, $b_\varepsilon(x)=\sqrt{\delta} \frac{d-2}{2}|x|_\varepsilon^{-2}x$, $D(\Lambda^\varepsilon)=(1-\Delta)^{-1}L^r$; see e.g.\,\cite[Theorem 4.1]{KiS1}. 

\textit{Then $e^{-t\Lambda}$, $t>0$, is an integral operator, and }
\begin{equation*}
e^{-t\Lambda}(x,y) \leq C t^{-\frac{d}{2}}\varphi_t(y), \quad x,y \in \mathbb R^d, y \neq 0, \quad t>0,
\end{equation*}
where $\varphi_t(y)=\eta(t^{-\frac{1}{2}}|y|)$, 
and $\eta$ is a $C^2(]0,\infty[)$ function such that
\[
\eta(r)=\left \{
\begin{array}{ll}
r^{-\sqrt{\delta}\frac{d-2}{2}}, & 0<r< 1,\\
\frac{1}{2}, & r\geq 2.
\end{array}
\right. 
\]

\begin{proof}
This is a consequence of 
Theorem A applied to $(e^{-t\Lambda^\varepsilon},\varphi_s)$. 
To verify ($S_1$), ($S_2$), ($S_3$), ($S_4$), we argue as in the proof of Theorem \ref{thm2}. 
We note that in the local case ($S_4$) can also be proved by direct calculations:

Proof of ($S_4$). It suffices to prove a priori estimate 
\begin{equation}
\label{le}
\tag{$\ast\ast$}
\|\varphi_s^\varepsilon e^{-t\Lambda^\varepsilon} (\varphi_s^\varepsilon)^{-1} g\|_{1 \rightarrow 1} \leq e^{c\frac{t}{s}}\| g\|_1, \quad c \neq c(\varepsilon), \quad g \in C_c, \quad 0<t \leq s,
\end{equation}
where $\varphi_s^\varepsilon(y)=\eta(s^{-\frac{1}{2}}|y|_\varepsilon)$.
Then, taking \eqref{le} for granted, we obtain $\|\varphi_s e^{-t\Lambda} \varphi_s^{-1} g\|_{1 \rightarrow 1} \leq e^{c\frac{t}{s}}\| g\|_1$ by \eqref{c} and Fatou's Lemma, and so
($S_4$) follows upon taking $s=t$.

\smallskip

Proof of  \eqref{le}. 

1) Set $$H(\varphi^\varepsilon)=-\Delta +\nabla\cdot (b_\varepsilon+2\frac{\nabla\varphi^\varepsilon}{\varphi^\varepsilon})+W_\varepsilon, \quad D(H(\varphi^\varepsilon))=W^{2,2},$$
$$
W_\varepsilon=-\frac{\nabla\varphi^\varepsilon}{\varphi^\varepsilon}\cdot\big(b_\varepsilon+\frac{\nabla\varphi^\varepsilon}{\varphi^\varepsilon}\big)-{\rm div}\big(b_\varepsilon+\frac{\nabla\varphi^\varepsilon}{\varphi^\varepsilon}\big).
$$
In fact, we have $H(\varphi^\varepsilon)=\varphi_s^\varepsilon \Lambda^\varepsilon (\varphi_s^\varepsilon)^{-1}$.

\smallskip

2) $\varphi_s^\varepsilon e^{-t\Lambda^\varepsilon}(\varphi_s^\varepsilon)^{-1}=e^{-tH(\varphi^\varepsilon)}.$

Indeed, put $F^t=\varphi^\varepsilon e^{-t \Lambda^\varepsilon}(\varphi^\varepsilon)^{-1}$. 2) is valid because $s\mbox{-}L^2\mbox{-}\lim_{t\downarrow 0}t^{-1}(1-F^t)f$ exists for all $f\in W^{2,2}$ and coincides with $H(\varphi^\varepsilon)f$.  

\smallskip

3) We have $b_\varepsilon+\frac{\nabla\varphi^\varepsilon}{\varphi^\varepsilon}=0$ in $B(0,\sqrt{s})$. Moreover, $|b_\varepsilon+\frac{\nabla\varphi^\varepsilon}{\varphi^\varepsilon}| \leq \frac{c_1}{\sqrt{s}}$, $|{\rm div}\big(b_\varepsilon+\frac{\nabla\varphi^\varepsilon}{\varphi^\varepsilon}| \leq \frac{c_2}{s}$,
the potential $|W_\varepsilon| \leq \frac{c}{s}$, $c \neq c(\varepsilon)$. It follows that
$e^{-tH(\varphi^\varepsilon)}$ is a quasi contraction on $L^1$: $\|e^{-tH(\varphi^\varepsilon)}\|_{1 \rightarrow 1} \leq e^{\frac{c}{s}t}$, $0<t \leq s$. In view of 2), \eqref{le} follows.
\end{proof}

\end{document}